\documentclass{article}

\title{Existence of Invariant Probability Measures for   Stochastic\\ Differential  Equations with  Finite Time Delay} 
\author{M. van den Bosch$^{\rm a, }$\footnote{Corresponding author.\\    Email addresses: \url{mark-bosch@hotmail.com}, \url{vangaans@math.leidenuniv.nl} and \url{S.M.VerduynLunel@uu.nl}.\\ To appear in SIAM
Journal on Applied Dynamical Systems. }\,\,, O.\,W. van Gaans$^{\mathrm a}$, S.\,M. Verduyn Lunel$^{\mathrm b}$}

\date{\today}

\usepackage{todonotes}
\usepackage{lipsum} 

\usepackage{subfiles}
\usepackage{xr}

\usepackage{geometry}
\usepackage[utf8]{inputenc}
\usepackage[english]{babel}
\usepackage{graphicx} 
\usepackage{xcolor}

\usepackage{mathtools} 
\usepackage{amsmath} 
\usepackage{amssymb} 
\usepackage{amsthm}
\usepackage{xspace}
\usepackage{enumerate}
\usepackage{mathrsfs}  

\usepackage[hidelinks]{hyperref}
\usepackage{biblio}
\numberwithin{equation}{section}
\newcommand\yesnumber{\addtocounter{equation}{1}\tag{\theequation}}

\newtheorem{theorem}{Theorem}[section]
\newtheorem{lemma}[theorem]{Lemma}
\newtheorem{proposition}[theorem]{Proposition}
\newtheorem{corollary}[theorem]{Corollary}

\theoremstyle{definition}
\newtheorem{definition}[theorem]{Definition}
\newtheorem{remark}[theorem]{Remark}
\newtheorem{assumption}[theorem]{Assumption}

\newtheorem{examplex}[theorem]{Example}

\newcommand{\cadlag}{càdlàg\xspace}
\newcommand{\caglad}{càglàd\xspace}
\newcommand{\R}{\mathbb{R}}
\newcommand{\Rplus}{[0,\infty)}
\newcommand{\tauRplus}{\mathbb D[-\tau,\infty)}
\newcommand{\taunul}{\mathbb D[-\tau,0]}
\newcommand{\N}{\mathbb{N}}
\newcommand{\dint}[3]{\int_{#1}^{#2}{#3} \, \mathrm{d}}

\def\gb #1{\bigl( #1 \bigr)}
\def\ph {\varphi}

\def\ph
{\varphi}

\def\th
{\theta}

\usepackage{tcolorbox}

\tcbuselibrary{breakable}
\tcbset{
  width=0.6\textwidth,
  halign=justify,
  center,
  breakable,
  colback=white, 
  colframe=black!50
}

\begin{document}

\maketitle
\begin{center}\small
    \textsc{
    $^{\mathrm a}$Mathematical Institute,  Leiden University,\\ P.O. Box 9512, 2300 RA Leiden, The Netherlands}

    \
    
    \textsc{$^{\mathrm b}$Department of Mathematics, University of Utrecht,\\ P.O. Box 80010, 3508 TA Utrecht, The Netherlands}
\end{center}

\

\begin{abstract}
   \noindent We provide sufficient conditions for the existence of invariant probability measures for generic  stochastic differential equations with finite time delay. This is achieved by means of the Krylov--Bogoliubov method. Furthermore, we focus on stochastic delay equations whose deterministic coefficient satisfies a one-sided bound, which enables us to show that boundedness in probability of a solution $X(t)$ entails boundedness in probability of its solution segment $X_t$. This implies that for a large set of systems, we can infer that an invariant measure  exists  if only there is at least one solution that is bounded in probability.
   Applications include, but are not limited to, the stochastic Mackey--Glass equations \cite{artikel2-MG} and the stochastic Wright's equation \cite{BosGaaVer1-24}. The noise driving the dynamical system is allowed to be an integrable Lévy process.

\end{abstract}


\section{Introduction}

The main objective of this paper  is to establish sufficient conditions   guaranteeing the existence of an invariant measure, defined on the space of \cadlag functions $D[-\tau,0],$ and thus that  of a stationary solution, for generic autonomous initial value problems of the form
  \begin{equation}\label{eq:SDDE-main}\left\{
	\begin{array}{rlll}
		\mathrm d X(t)&=&a(X_t)\,\mathrm d t+b(X_{t-}) \,\mathrm d M(t), &\quad \text { for } t\geq 0, \\[.05cm]
		X(u)&=&\Phi(u),& \quad \text { for } u \in[-\tau, 0].
	\end{array}\right.
\end{equation}
In here, the process $M=(M(t))_{t\geq 0}$ is a semimartingale (with restrictions specified later), and  $a,b:D[-\tau,0]\to\mathbb R$ are locally Lipschitz functionals \pagebreak with respect to the supremum norm $\|\cdot\|_\infty$ and are, in addition, assumed \textit{proper} for well-posedness; see {\S}\ref{sec:existunique}. Throughout we write  $D\Rplus$, $D[-\tau,\infty),$ and $D[-\tau,0]$ for the spaces of \cadlag functions defined on $[0,\infty)$, $[-\tau,\infty)$, and $[-\tau,0]$,  respectively, along with $C\Rplus$, $C[-\tau,\infty),$ and $C[-\tau,0]$  being the  spaces of  continuous functions.

We obtain  the existence of an invariant measure via a procedure thanks to  Krylov and Bogoliubov \cite{kryloff1937theorie}; one can find the existence theorem(s) for our delayed setting in Appendix \ref{app:D}. 
Thanks to the latter, it all comes down to finding sufficient conditions for which we have partial tightness of one set of  solution segments; see Appendix \ref{app:B} for general tightness results. 
The main steps are illustrated in Figure \ref{fig:placeholder}. 

\begin{figure}[!t]
    \centering
    \includegraphics[width=0.75\linewidth]{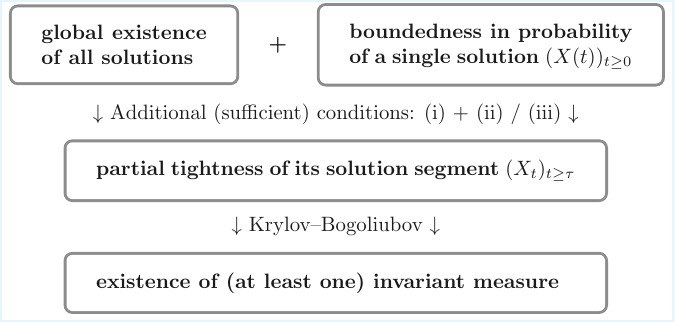}
     \caption{A schematic on how we prove the existence of an invariant measure for system \eqref{eq:SDDE-main}.}
    \label{fig:placeholder}
\end{figure}

Since we focus on systems with finite time delays only,  it is possible to derive  general results with regard to boundedness in probability of solutions and their segments, see {\S}\ref{boundednesssolutions}, as well as  tightness of solution segments, see {\S}\ref{sec:towards}. Tightness of the segment process $(X_t)_{t\geq \tau}$ of a  solution $(X(t))_{-\tau\leq t<\infty}$ to \eqref{eq:SDDE-main} is ensured when, in addition to global existence of all solutions and boundedness in probability of $(X(t))_{t\geq 0}$, the following two sufficient conditions are satisfied:
 \begin{enumerate}
     \item[(i)]  the noise coefficient $b$ is  bounded, i.e., there is a constant $\beta\geq 0$ such that 
     \begin{equation}
         |b(\ph)|\leq \beta,\quad \ph\in D[-\tau,0];
     \end{equation}
     \item[(ii)] the process $\big(\sup_{u\in[t-\tau,t]}|a(X_u)|\big)_{t\geq \tau}$ is bounded in probability, i.e.,
	\begin{equation} 
		\lim _{R \rightarrow \infty} \sup _{t \geq \tau}\mathbb  P\left(\sup _{u\in [t-\tau,t]} |  a(X_u) |>R\right)=0. \label{eq:difficult-main}
	\end{equation}
 \end{enumerate}
Moreover,  we prove that  (i) and (ii) can be replaced by the following  condition:
 \begin{itemize}
     \item[(iii)] the process $\big(\|X_t\|_\infty\big)_{t\geq 0}$ is bounded in probability, i.e.,
     \begin{equation} 
		\lim _{R \rightarrow \infty} \sup _{t \geq 0}\mathbb  P\left(  \|  X_t \|_\infty>R\right)=0. \label{eq:difficult-main2}
	\end{equation}
 \end{itemize}
We  provide three separate proofs demonstrating that these conditions imply partial tightness: one where $M$ is a Brownian motion ({\S}\ref{sec:tight_cont});    one where $M$ is a square-integrable Lévy process ({\S}\ref{sec:special}); and   one inspired by \cite{article:reiss} using semimartingale characteristics (see Appendix \ref{app:C}), which allows $M$ to be an integrable L\'evy process ({\S}\ref{sec:special}).  In particular,  the full segment process $(X_t)_{t\geq 0}$ is tight in the continuous case and  (iii) is in fact sufficient and necessary; see Proposition \ref{prop:Dtight}.  Finally, we claim that most  if not all results mentioned above can be generalised to systems of delay equations.

The main results of this article are Theorem \ref{eq:SDDEW} and Theorems \ref{cor:main}--\ref{cor:cor:main}. Furthermore, in {\S}\ref{Sec3.4.2} we show  that|for delay equations with stochastic negative feedback|conditions  (ii) and  (iii) are automatically satisfied when the solution $(X(t))_{t\geq0}$ is bounded in probability. In particular, this  conclusion   can be generalised to   stochastic delay equations whose   deterministic coefficient $a$ in \eqref{eq:SDDE-main} has a one-sided bound. This result is due to Proposition \ref{thm:interstep} and   summarised  in Corollary \ref{cor:main_cont} for a Brownian motion.  Finally, let us point out that these general results are exploited in   \cite{artikel2-MG, BosGaaVer1-24}.

\paragraph{Organisation} The paper is organised as follows.
As a service to the dynamical systems community, we outline in {\S}\ref{sec:existunique} a SDDE framework with finite time delays in the spirit of \cite[Ch. 5]{book:mao} and provide general existence and uniqueness results together with simple proofs for certain types of noise. 
In {\S}\ref{boundednesssolutions} we study the boundedness in probability of solutions and solution segments, and its connection to finite time blowups. In particular, we prove in {\S}\ref{sec:bound-prelimedness} that boundedness in probability of a solution $X(t)$ entails boundedness in probability of its solution segment $X_t$ when the deterministic coefficient $a$ in \eqref{eq:SDDE-main} has a one-sided bound. 
In {\S}\ref{sec:towards} we provide sufficient conditions regarding the existence of an invariant measure for generic stochastic differential equations with a finite time horizon, and verify those conditions for delay equations with stochastic negative feedback. 
See Appendix \ref{app:A} for the classes of stochastic integrators we consider.




\section{Existence and uniqueness of SDEs with finite time delay}\label{sec:existunique}

Fix a   probability space $(\Omega, \mathcal{F}, \mathbb  P)$ together with a filtration $\mathbb F=(\mathcal F_t)_{0\leq t<\infty}$ 
that satisfies the usual conditions.  Let us denote by  $\mathbb D\Rplus$ the space of $\mathbb F$-adapted \cadlag  processes $(X(t))_{0\leq t<\infty}$. One can trivially extend the filtration  by setting $\mathcal F_s=\mathcal F_0$ for $-\tau\leq s<0$, with $\tau>0$ fixed, and in a similar fashion we   write $\tauRplus$    for the space of $(\mathcal F_t)_{-\tau \leq t<\infty} $-adapted \cadlag processes $(X(t))_{-\tau\leq t<\infty}$. 
The space $\taunul$ consists  of all \cadlag  processes $(X(t))_{-\tau \leq t\leq 0}$ where $X(s)$ is $\mathcal F_0$-measurable for all  $s\in[-\tau,0]$.

A process $X=(X(t))_{-\tau\leq t<\infty}\in \tauRplus$ gives rise to the {segment process} $(X_t)_{t\geq 0}$,  where each {segment} $X_t$ is an $\mathcal F_t$-measurable \cadlag process   on $[-\tau,0]$  defined by
\begin{equation}\label{eq:segment} 	    
    X_t(\theta):=X(t+\theta),\quad - \tau\leq \theta\leq0.
\end{equation} 
 Note that the segment process can be regarded as an $\mathbb F$-adapted $D[-\tau,0]$-valued stochastic process, thus where each $X_t$ is a $D[-\tau,0]$-valued random variable; see Corollary \ref{remark:segment}.

Throughout this section, we let $M=(M(t))_{t\geq 0}$  be a semimartingale (adapted to $\mathbb F)$, e.g., a Brownian motion or a L\'evy process, unless stated otherwise.  The purpose of  {\S}\ref{Sec2.1.1} is to provide a rigorous mathematical interpretation for autonomous equations of the form
\begin{equation}\label{eq:SDDE11}
    \left\{
    \begin{array}{rlll}
 	\mathrm dX(t)&=&a(X_t)\,\mathrm d t+b(X_{t-}) \,\mathrm d M(t), &\quad \text { for } t\geq 0, \\[.05cm]
 		X(u)&=&\Phi(u)& \quad \text { for } u \in[-\tau, 0],
    \end{array}
    \right.
\end{equation}
subject to an {initial process} $\Phi\in\mathbb D[-\tau,0]$. In  {\S}\ref{Sec2.1.2} we consider the non-autonomous case.
 All  results throughout this section naturally extend to higher dimensional systems.
%
The framework  we   develop is in the spirit of \cite[Ch. 5]{book:mao}, which provides existence and uniqueness results for (strong)  solutions to stochastic
differential equations driven by Brownian motion together with a finite time horizon. 
Existence and uniqueness results with more general delay terms can be found in, e.g., \cite{book:Rao, book:protter}.

\subsection{Autonomous equations}\label{Sec2.1.1}
\noindent   We write $\ph(s-)=\lim_{t\nearrow s}\ph(s)$ and $\|\varphi\|_\infty=\sup_{-\tau\leq s\leq 0}|\varphi(s)|$ for any \cadlag function  $\ph$. Observe that \eqref{eq:SDDE11}  is shorthand  for the following integral equation:
\begin{equation}\label{eq:interpret}
    X(t)=\Phi(0)+\dint0t{\mathbf A_\Phi(X)(s)}s+\dint0t{\mathbf B_\Phi(X)(s-)}M(s),\quad t\geq 0,
\end{equation}
where
\begin{enumerate} 
    \item the maps $\mathbf A_\Psi,\mathbf B_\Psi:\mathbb D\Rplus\to\mathbb D\Rplus$, for some fixed $\Psi\in\mathbb D[-\tau,0]$, are defined pathwise for any process $Y=(Y(t))_{t\geq 0}\in\mathbb D\Rplus$ by
    \begin{equation}\label{eq:main}
        \mathbf A_\Psi(Y)(s,\omega)=a\left(Y^\Psi_s(\omega)\right)\quad\text{and} \quad 	\mathbf B_\Psi(Y)(s,\omega)=b\left(Y^\Psi_s(\omega)\right),\quad s\geq0;
    \end{equation}
\item the  process $(Y^{\Psi}_t)_{t\geq 0}$ in  \eqref{eq:main} is the segment process of   $Y^\Psi\in \tauRplus$, where
\begin{equation}
Y^\Psi(s)=\left\{\begin{array}{ll}\label{eq:extended}
	\Psi(s), & -\tau \leq s<0 \\
	Y(s), & s \geq 0;
\end{array}\right.
\end{equation}
\item the functionals  $a:D[-\tau,0]\to\R$ and $b:D[-\tau,0]\to\R$    in  \eqref{eq:main} are  locally Lipschitz continuous with respect to the supremum norm $\|\cdot \|_\infty$  and \textit{proper} (see  Definition \ref{def:proper}). 
\end{enumerate} 
\noindent The first integral in \eqref{eq:interpret} is to be understood as a Lebesgue--Stieltjes integral. The second integral is a stochastic integral as in \cite{book:protter}. For further information regarding the functionals or integrals, see the discussion succeeding the proof of Proposition \ref{prop:p_est}.

Let us  clarify what is meant by a (strong) solution of initial value problem \eqref{eq:SDDE11}. For this we need the following notion. We denote by $X^T$  the stopped process of $X$ at stopping time $T$, which is defined by
$
	X^T(t,\omega)=X({t\wedge T(\omega)},\omega)=X(t,\omega)\mathbf{1}_{\{t<T(\omega)\}}+X(T(\omega),\omega)\mathbf{1}_{\{t\geq T(\omega)\}}
$. 

\begin{definition}\label{def:solution}
	Let $X=(X(t))_{-\tau\leq t<\infty} $ be a stochastic process and   $T_\infty$   a stopping  time. Then $(X,T_\infty)$, often abbreviated by $X$ again, is said to be a \texttt{local\,(strong)} \texttt{solution}   of  \eqref{eq:SDDE11} on the interval  $[-\tau,T_\infty)$, whenever $X_0=\Phi$   holds  $\mathbb P$-a.s.,   $X^{T_k}\in\mathbb D\Rplus$, and  $\mathbb P$-a.s.\  we have  
	\begin{equation}\label{eq:interpret2}
		X(t\wedge T_k)=\Phi(0)+\dint0{t\wedge T_k}{\mathbf A_\Phi(X^{T_k})(s)}s+\dint0{t\wedge T_k}{\mathbf B_\Phi(X^{T_k})(s-)}M(s),\quad t\geq 0,
	\end{equation}
  for all integers $k\geq 1$, where $(T_{k})_{k \geq 1}$ is any non-decreasing sequence of finite stopping times such that $T_{k} \nearrow T_{\infty}$ holds $\mathbb P$-a.s.,\,as $k\to\infty.$ 
Further, if $\limsup _{k \rightarrow \infty}|X(t\wedge T_k)(\omega)|=\infty$ when $T_{\infty}(\omega)<\infty$,  for almost all $\omega\in\Omega,$ then $(X,T_\infty)$ is  a \texttt{maximal\,local} \texttt{solution} with $T_{\infty}$ being the   \texttt{explosion\,time}. A maximal local solution  is    \texttt{unique}  if for any other maximal local solution $(Y,S_\infty)$ we have    $T_{\infty}=S_{\infty}$    $\mathbb P$-a.s.,\,such that
\begin{equation}
	X^{T_k\wedge S_k}\quad\text{and}\quad Y^{T_k\wedge S_k}\quad\text{are indistinguishable},
\end{equation}for all integers $k\geq 1$. Finally, a local solution $(X,T_\infty)$ is a \texttt{global\,(strong)\,solution} of \eqref{eq:SDDE11} if $T_\infty=\infty$ holds $\mathbb P$-a.s., or in other words, when $X$ satisfies \eqref{eq:interpret}.
\end{definition}

Clearly, if $X$ is a global solution, then $X\in\mathbb D\Rplus.$ In addition, when we talk about a solution, we may as well always consider a version with \cadlag paths everywhere, because of the fact that these two versions are indistinguishable \cite[Thm. I.2]{book:protter}. We usually set $X=0$ on the stochastic interval $[T_\infty,\infty).$



 

As a service to the community, we will prove the existence and uniqueness result below by
     restricting ourselves to  $M$ that are martingales of class (HDol); see Appendix \ref{app:A} for the definition of this class and other classes. Within this setting,  the proof is completely analogous to the Brownian case, see \cite[Ch. 5]{book:mao}, and it can be divided into essentially three main steps. It immediately follows that the assertion is true for semimartingales of class (HSqL) as well, since their martingale part is in (HSqLM) $\subseteq$ (HDol) and because the predictable part is directly proportional to $t$. For general semimartingales $M$, we refer to the results scattered throughout \cite{book:Rao} and \cite{book:protter}. These monographs follow two distinct approaches on how to prove this, yet require rather sophisticated tools. 

\begin{theorem}\label{thm:main-exist} 
	Suppose $a, b: D[-\tau,0] \rightarrow \mathbb{R}$ are proper  locally Lipschitz functionals. Then for every  initial process $\Phi\in \mathbb D[-\tau,0]$, there exists a stopping time $T_\infty$ and  a  stochastic  process $X$ such that $(X,T_\infty)$ is a local solution to   initial value problem \eqref{eq:SDDE11}. 
The stopping time $T_\infty$ can be chosen such that $(X,T_\infty)$ is a maximal solution of   \eqref{eq:SDDE11}. This maximal solution   is   unique. 
	
If, in addition, the functionals $a,b$ satisfy  the linear growth condition, then  \eqref{eq:SDDE11} admits a unique global solution. This global solution is in fact a semimartingale. 
\end{theorem}
\begin{proof}  For general semimartingales $M$, we refer to  \cite{book:Rao,book:protter}. Now let $M$ be of class (HDol); see Appendix \ref{app:A}. We shall proceed by dividing the proof into three steps. In Step 1,  we show existence and uniqueness of global solutions assuming that the initial condition is square integrable, i.e., $\mathbb E\|\Phi\|_{\infty}^2<\infty$, and that the functionals $a,b$ are globally Lipschitz continuous with respect to the supremum norm $\|\cdot\|_\infty$. In Step 2, we extend the results to the case where $a$ and $b$ are locally Lipschitz.
In Step 3, we show that   square integrability  of the initial condition is a superfluous condition to pose. 

\textit{Step 1.} 
First, suppose $X$ is a global solution. Then the linear growth condition combined with an application of H\"older's inequality, Doob's maximal inequality \cite[Thm. 1.3.8]{book:karatzas}, and Gr\"onwall's inequality \cite[Thm. 1.8.1]{book:mao}, yields
\begin{equation}
    \mathbb E\left[\sup_{-\tau\leq t\leq T}|X(t)|^2\right]\leq (1+4\mathbb E\|\Phi\|_\infty^2)e^{3K_{\rm lin}T(T+4\lambda)}-1,\label{eq:growth}
\end{equation}
for any $T\geq 0.$ For more details, we refer to the proof of \cite[Lem. 2.3.2]{book:mao} for the continuous case or the proof of Proposition \ref{prop:p_est}.
Growth estimate \eqref{eq:growth} together with global Lipschitz continuity yields uniqueness. Indeed, suppose $X$ and $Y$ are two global solutions, then for any $0\leq t\leq T$ we have
\begin{equation}\begin{aligned}
    \mathbb E\left[\sup_{-\tau\leq s\leq t}|X(s)-Y(s)|^2\right]&\leq 2K_{\rm lip}(T+4\lambda)\int_0^t\mathbb E\|X_s-Y_s\|_\infty^2\,\mathrm ds\\&\leq 2K_{\rm lip}(T+4\lambda)\int_0^t \mathbb E\left[\sup_{-\tau\leq r\leq s}|X(r)-Y(r)|^2\right]\mathrm ds.
\end{aligned}\end{equation}
Since $T\mapsto \mathbb E\left[\sup_{-\tau\leq s\leq T}|X(s)-Y(s)|^2\right]$ is integrable thanks to  \eqref{eq:growth}, we obtain that, in fact, 
 the latter equals zero after invoking Gr\"onwall's inequality.
 
 Existence of a global solution $X$ is obtained via a Picard's iteration. Define for any  $n\geq 0$ the  process
\begin{equation}
    X^{n+1}(t)=\Phi(0)+\int_0^t\mathbf A_\Phi(X^{n})(s)\,\mathrm ds+\int_0^t \mathbf B_\Phi(X^{n})(s-)\,\mathrm dM(s),\quad t\geq 0.
\end{equation}
With similar techniques as  before, one can show by  induction that,  for any $n\geq 0,$ we have
\begin{equation}
    \mathbb E\left[\sup_{0\leq t\leq T}|X^{n+1}(t)-X^n(t)|^2\right]\leq \frac{C[LT]^n}{n!},\label{eq:L2}
\end{equation}
where $L=2K_{\rm lip}(T+4\lambda)$ and $C=2K_{\rm lin}(T+4\lambda)(1+\mathbb E\|\Phi\|_\infty^2)T.$ Consequently, we can show with the help of Borel-Cantelli that there exists a process $X=(X(t))_{t\geq 0}$ such that $X\in L^2(\Omega\times [0,T])$, for all $T\geq0,$ and for which we have $X^n\to X$ uniformly in $t$ with probability 1. Moreover, we clearly have convergence $X^n(t)\to X(t)$, for any $t\geq 0,$ 
  in $L^2(\Omega)$ as well, thanks to \eqref{eq:L2}, which allows us to show that $X$ is a solution to
\eqref{eq:interpret}. 
  See \cite[Thm. 2.3.1]{book:mao} for further details.

\textit{Step 2.} 
We shall follow  a standard truncation procedure and it is completely in line with \cite[Thm. 2.3.4]{book:mao}. For each integer $n \geq 1$,  introduce the truncated functionals
\begin{equation}
a_n(\varphi)= \begin{cases}a(\varphi) & \text { if }\|\varphi\|_\infty \leq n, \\ a(n \varphi /\|\varphi\|_\infty) & \text { if }\|\varphi\|_\infty>n,\end{cases}
\end{equation}
and $b_n(\varphi)$ similarly. Then $a_n$ and $b_n$ are globally Lipschitz continuous, for any  $n\geq 1.$ Step 1 gives us a unique solution $X^n$ to the equation
\begin{equation}
X^n(t)=\Phi(0)+\int_{0}^t a_n\left(X^n_s\right) \mathrm d t+\int_{0}^t b_n\left(X^n_{s-}\right)\mathrm d M(s), \quad t\geq 0 .
\end{equation}
Define the increasing sequence of stopping times $(T_n)_{n\geq 1}$ by
\begin{equation}
T_n=n\wedge \inf \left\{t \geq 0:\left|X^n(t)\right| \geq n\right\}.\end{equation}
If we assume, in addition, that $a$ and $b$ satisfy the linear growth condition, then one can show that $T_n\to\infty$ holds $\mathbb P$-a.s., as $n\to\infty,$ and that $X(t):=\lim_{n\to\infty}X^n(t)$ is well-defined and solves   \eqref{eq:interpret}. Indeed, we have   $X^n(t)=X^{n+1}(t)$, for $0\leq t\leq T_n$, and $X(t\wedge T_n)=X^n(t),$ for any $n\geq 1.$
However, whenever a functional no longer satisfies linear growth,  an explosion may occur in finite time, hence only guaranteeing a maximal local solution.
Uniqueness can be proved by means of a stopping
time argument as well.

\textit{Step 3.} Instead of using a truncation argument as   in \cite[p. 134]{book:kloeden}, we proceed as follows. Suppose now $\Phi\in\mathbb D[-\tau,0]$ is merely $\mathcal F_0$-measurable and introduce  the measurable map
\begin{equation}
m:\Omega\to(0,1],\quad \omega\mapsto e^{-\|\Phi(\omega)\|_\infty^2}.
\end{equation}
Note that $\mathbb E\big[e^{-\|\Phi\|_\infty^2}\big]>0$  because $\|\Phi\|_\infty<\infty$ holds ($\mathbb P$-a.s.), since $\Phi$ has \cadlag sample paths.
This allows us to define a new probability measure $\mathbb Q$ on $(\Omega,\mathcal F)$ as follows:
\begin{equation}
    \mathbb Q(A)=\frac{\int_A m(\omega)\,\mathrm d\mathbb P(\omega)}{\int_\Omega m(\omega)\,\mathrm d\mathbb P(\omega)}=\left(\mathbb E\big[e^{-\|\Phi\|_\infty^2}\big]\right)^{-1}\mathbb E\big[\mathbf1_Ae^{-\|\Phi\|_\infty^2}\big].
\end{equation}
Clearly $\mathbb P$ and $\mathbb Q$ are equivalent measures, i.e., $\mathbb P\ll \mathbb Q\ll \mathbb P.$ 
Note that $\Phi$ is square integrable with respect to the new measure $\mathbb Q$. Indeed, we have
\begin{equation}
   \int_\Omega \|\Phi\|_\infty^2 \,\mathrm d\mathbb Q = \left(\mathbb E\big[e^{-\|\Phi\|_\infty^2}\big]\right)^{-1}\int_\Omega \|\Phi\|_\infty^2 e^{-\|\Phi\|_\infty^2}\,\mathrm d\mathbb P\leq \left(\mathbb E\big[e^{-\|\Phi\|_\infty^2}\big]\right)^{-1}e^{-1}<\infty.
\end{equation}
From Steps 1 and  2, we  conclude that there exists a (local) solution to  initial value problem \eqref{eq:SDDE11}, where equality holds $\mathbb Q$-a.s., and with the stochastic integral being computed under the law of this new probability measure. In particular, we have   $X^{T_k}\in L^2(\Omega\times [0,T];\mathbb Q\times\mathrm ds)$, for all $T\geq0,$  for any appropriate sequence of stopping times $(T_k)_{k\geq 1}$ as in Definition \ref{def:solution}.
 Since $\mathbb Q\ll\mathbb P,$ it follows from \cite[Thm. II.14]{book:protter}  that all the stochastic integrals are $\mathbb Q$-indistinguishable from their corresponding versions computed under the law of $\mathbb P$. Under this law, all the stochastic integrals clearly exist in the ucp-sense, which is sufficient, but not necessarily with $L^2(\Omega,\mathcal F,\mathbb P)$-convergence. The solution we found satisfies initial value problem \eqref{eq:SDDE11} also $\mathbb P$-a.s.. This completes the proof.
\end{proof}




Theorem \ref{thm:main-exist} can, in fact, be slightly improved. Indeed, replacing the linear growth condition by a monotone condition, see, e.g., \cite[Thm. 2.3.6]{book:mao}, also results into global solutions. 
Additionally,  an estimate like \eqref{eq:growth} holds for every $p\geq 2.$ We believe this can be proved more easily  by not invoking It\^o's formula \cite{book:applebaum,book:sulem,book:protter},  due to the discontinuities of  the integrator $M$.

 \begin{proposition}\label{prop:p_est}
     Suppose $M=(M(t))_{t\geq 0}$ is of class \textnormal{(HDol)} or \textnormal{(HSqL)}, let $X=(X(t))_{t \geq 0}$  be a global solution of problem \eqref{eq:SDDE11}, and assume $a,b$ satisfy the linear growth condition. For $p\geq 2,$ if $\mathbb E[\sup_{-\tau\leq t\leq 0}|\Phi(t)|^p]<\infty$ holds, then $\mathbb E[\sup_{-\tau\leq t\leq T}|X(s)|^p]<\infty$ for all $T\geq 0.$ 
 \end{proposition}
 \begin{proof}
      Fix $T>0$. As before, it suffices to consider the class (HDol) only. Introduce the  increasing sequence of stopping times
     \begin{equation}
T_n=T\wedge \inf \left\{t \in\left[0, T\right]:\left|X(t)\right| \geq n\right\},\quad n\geq 1.\end{equation}
Then the elementary estimate $|a+b+c|^p\leq 3^{p-1}(|a|^p+|b|^p+|c|^p)$ yields
\begin{equation}
\begin{aligned}
    \mathbb E\left[\sup_{-\tau\leq s\leq t}|X(s\wedge T_n)|^p\right]\leq 3^{p-1}\mathbb E\|\Phi\|_{\infty}^p&+3^{p-1}\mathbb E\left(\int_0^{t\wedge T_n}|a(X(s))|\,\mathrm ds\right)^p\\&+3^{p-1}\mathbb E\left|\sup_{0\leq s\leq t}\int_0^{s\wedge T_n} b(X_{r-})\,\mathrm dM(r)\right|^{p},
\end{aligned}
\end{equation}
for any $0\leq t\leq T.$
Applying the Burkholder--Davis--Gundy inequality
\cite[Ch. VII]{DelMey80} 
results into 
\begin{equation}
    \mathbb E\left|\sup_{0\leq s\leq t}\int_0^{s\wedge T_n} b(X_{r-})\,\mathrm dM(r)\right|^{p}\leq C_p \mathbb E\left|\int_0^{t\wedge T_n}b(X_{s-})^2\,\mathrm d[M](s)\right|^{p/2},\label{eq:Burk}
\end{equation}
for some $C_p>0$  depending solely on $p\geq 2.$
We  like to  point out that working with the predictable quadratic variation would be slightly unfavourable now, since this would result into an additional term  \cite[{Lem. VII.3.34}]{book:jacod}, as opposed to \eqref{eq:Burk}. We  refer to \cite{kuhn2023maximal} for a similar discussion; see \cite[Thm. 4.19]{kuhn2023maximal} and \cite[Thm. 4.21]{kuhn2023maximal} in particular.

Since  $([M](t))_{t\geq 0}$ is a non-decreasing process, we can invoke H\"older's inequality to obtain
\begin{equation}
\begin{aligned}
\left|\int_0^{t\wedge T_n}b(X_{s-})^2\,\mathrm d[M](s)\right|^{p/2}&\leq C_p \big([M](T)\big)^{(p-2)/2} \int_0^{t\wedge T_n}|b(X_{s-})|^p\,\mathrm d[M](s)\\
&\leq C_p \big(1+[M](T)\big)\big([M](T)\big)^{k_p} \int_0^{t\wedge T_n}|b(X_{s-})|^p\,\mathrm d[M](s),\label{eq:unnecessary}
\end{aligned}
\end{equation}
where $k_p=\lfloor (p-2)/2\rfloor\geq 0 $ is the  greatest integer smaller or equal to the value $(p-2)/2$. If $M$ is continuous, then it is a Brownian motion with drift and  $[M](t)=\sigma^2t$ is deterministic, making the final step in  \eqref{eq:unnecessary} and  the part  what will follow unnecessarily cumbersome. 
For any integer $k\geq 1$ and non-negative $(\Omega,\mathcal F)$-measurable function $A$, we get 
\begin{align}
    \mathbb E \big[[M](T)^{k}A\big]&=\mathbb E \big[[M](T)^{k-1}A[M](T)\big]=\mathbb E\left[[M](T)^{k-1}A\int_0^T\,\mathrm d[M](s)\right]\nonumber\\&=\int_{\Omega\times[0,T]}[M](T)^{k-1}A\,\mathrm d\mu_M\leq \lambda T\mathbb E\big[[M](T)^{k-1}A\big]\leq ...\leq (\lambda T)^{k}\mathbb E [A].
\end{align}
Consequently,  set $K_p=C_p(1+\lambda T)(\lambda T)^{k_p}$ and we have
\begin{equation}
\begin{aligned}
    \mathbb E\left|\sup_{0\leq s\leq t}\int_0^{s\wedge T_n} b(X_{s-})\,\mathrm dM(s)\right|^{p}&\leq K_p\mathbb E  \int_0^{t\wedge T_n}|b(X_{s-})|^p\,\mathrm d[M](s) \\
    &\leq  K_p\lambda \mathbb E  \int_0^{t\wedge T_n}|b(X_{s})|^p\,\mathrm ds\\&
    \leq 2^{p-1}K_p\lambda \mathbb E\int_0^{t\wedge T_n}\left[1+\sup_{s-\tau\leq r\leq s}|X(r)|^p\right]\mathrm ds\\
    &\leq 2^{p-1}K_p\lambda \int_0^{t}1+\mathbb E\left[\sup_{-\tau\leq r\leq s}|X(r\wedge T_n)|^p\right]\mathrm ds,
\end{aligned}
\end{equation}
where we have used  Fubini's theorem and that $a,b$ satisfy the linear growth condition.

A similar estimate can be obtained for the deterministic part. Gr\"onwall's lemma now implies
\begin{equation}
    \mathbb E\left[\sup_{-\tau\leq s\leq t}|X(s\wedge T_n)|^p\right]\leq (1+3^{p-1}\mathbb E\|\Phi\|_\infty)e^{K T}-1,\quad 0\leq t\leq T,
\end{equation}
for some $K>0.$ Taking $n\to\infty$ proves the assertion.
 \end{proof}

\paragraph{About the functionals}  Let us write $E$, e.g., $E[0,\infty)$, where $E$ can  be either $C$ or $D$. Of course, one could include the space of \caglad functions, but we will not need it here. 

Given two normed spaces $(X,\|\cdot\|_X)$ and $(Y,\|\cdot\|_Y)$,  a   map $f:X\to Y$ is said to be  locally Lipschitz continuous if for all integers $n\geq 1$ there exists a   constant $K_n> 0$ such that, for $x_1,x_2\in X$ satisfying   $\|x_1\|_X\wedge \|x_2\|_X\leq n$, we have
    \begin{equation}
        \|f(x_{1})- f(x_{2})\|_Y \leq K_n \|x_{1}- x_{2}\|_X.
    \end{equation}
    The map $f$ is (globally) Lipschitz continuous if $K_n=K_{\rm lip}$ for all $n\geq 1,$ for some $K_{\rm lip}>0.$ The following definition is  a special case, where   $Y=\mathbb R$ and $X=E[0,\infty)$ is endowed with the supremum norm $\|\cdot\|_\infty.$ For completeness, we also add the definition of the linear growth condition.
\begin{definition}\label{def:loc}
	A functional $f:E[-\tau,0]\to\R$ is called \texttt{(globally)\,\,Lipschitz} \texttt{continuous}  when it satisfies:
	\begin{equation}\begin{aligned}\label{eq:lip}
	 \exists K_{\rm lip}>0 \text{ such that } \forall \varphi,\psi\in E[-\tau,0] &\text{ we have }|f(\varphi)-f(\psi)|\leq K_{\rm lip} \sup_{s\in[-\tau,0]}|\varphi(s)-\psi(s)|.
  \end{aligned}
	\end{equation}
A functional  $f:E[-\tau,0]\to\R$ is called \texttt{locally\,\,Lipschitz\,\,continuous} when it satisfies:
\begin{equation}\label{eq:loclip}
	\begin{aligned}&\text{
	 $\forall n\in\N,\,\exists K_n>0$ such that $\forall \varphi,\psi\in E[-\tau,0]$ we have}\\[.3cm]& \quad\quad\,\,\,\sup_{s\in[-\tau,0]}|\varphi(s)|\vee|\psi(s)|\leq n\implies |f(\varphi)-f(\psi)|\leq K_n \sup_{s\in[-\tau,0]}|\varphi(s)-\psi(s)|. 
 \end{aligned}
\end{equation}
A functional $f:E[-\tau,0]\to\R$ is said to satisfy  the \texttt{linear\,\,growth\,\,condition} when it satisfies:\vspace{-.1cm}
\begin{equation} \text{
  $\exists K_{\rm lin}>0$ such that $\forall \varphi\in E[-\tau,0]$ we have  }  |f(\varphi)|\leq K_{\rm lin} \left(1+ \sup_{s\in[-\tau,0]}|\varphi(s)|\right). \end{equation}
  \end{definition}
Recall that   globally Lipschitz  
implies local Lipschitz continuity together with the linear growth condition. On the contrary,    a locally Lipschitz  functional satisfying the linear growth condition is not necessarily  globally Lipschitz.  An additional note, linear growth     not only implies the existence and uniqueness of global solutions, but also allows for a variation of constants formula to hold    \cite{article:stojkovic}. 

\begin{examplex}
  Suppose $d\geq 1$ and consider $h:\R^d\to\R$ to be (locally) Lipschitz continuous. Then any  $f:E[-\tau,0]\to \R$, defined by 
	\begin{equation}
		f(\varphi)=h\left(\int_{[-\tau,0]}{\varphi(s)}\lambda_1(\mathrm ds),\ldots,\int_{[-\tau,0]}{\varphi(s)}\lambda_d(\mathrm ds)\right),\quad \varphi\in E[-\tau,0],
	\end{equation}
	where $\lambda_1,...,\lambda_d$ are  finite signed Borel measures on $[-\tau,0]$, is  (locally) Lipschitz  continuous.  For  other  examples, we refer to   \cite[p. 257]{book:protter} and \cite[p. 1413]{article:reiss}.  See also Example \ref{example:MG}; take $\gamma(t)=\gamma\geq0$ and $r(t)=r\geq 0$ to obtain autonomous equations.
\end{examplex}
\begin{definition}\label{def:proper}
	A (locally) Lipschitz continuous functional $f:D[-\tau,0]\to\R$ is said to be a \texttt{proper} functional if for all $\psi\in D[-\tau,\infty)$  the mapping \begin{equation}\label{eq:extra}
\mathfrak F_ {f,\psi}:\Rplus\to\R,\,	t\mapsto f(\psi_t),
\end{equation}  where $\psi_t(\theta)=\psi(t+\theta)$ for $-\tau\leq \theta\leq 0$, is  \cadlag on $[0,\infty)$. 
\end{definition}

Fix $\psi\in D[-\tau,\infty)$.  If $f:D[-\tau,0]\to\R$ is (locally) Lipschitz continuous, then the  mapping
$
\mathfrak F_ {f,\psi}$
is not necessarily \cadlag on $ \Rplus $. 
Indeed,
 let $f:D[-\tau,0]\to\R$  be the jump size functional  \begin{equation}f(\ph)=\Delta\ph(0)=\ph(0)-\ph(0-),\quad \ph\in D[-\tau,0].\end{equation}
 This $f$ is globally Lipschitz  and we have  $f(\psi_t)=\Delta\psi(t)=\psi(t)-\psi(t-)$, for any $\psi\in D[-\tau,\infty)$, yet the mapping   $	t\mapsto f(\psi_t)$ is only \cadlag for   continuous  $\psi$. 
	Furthermore, observe that if we    restrict  our scope to the continuous setting,  then being proper---that is,   for all $\psi \in C[-\tau,\infty)$ we want $\mathfrak F_ {f,\psi}$ to be continuous---is automatically satisfied and thus a superfluous condition to pose. Indeed, for every  $\psi\in C[-\tau,\infty)$, we have that $\psi$ is uniformly continuous on a compact interval. For any $t\geq 0$ fixed we can subsequently find   for  arbitrary $\varepsilon>0$  a sufficiently small $\delta>0$ such that
	\begin{equation}
		|\mathfrak F_ {f,\psi}(t)-\mathfrak F_ {f,\psi}(s)|\leq K_n \sup_{u\in[-\tau,0]}|\psi(t+u)-\psi(s+u)|\leq \varepsilon,
	\end{equation}
which holds for all time $s\geq 0$ with $|t-s|=|(t+u)-(s+u)|<\delta.$
 
  Definition \ref{def:proper} seems a bit out of the ordinary at first glance, but  is  necessary to make sure that the maps $\mathbf A_\Psi,\mathbf B_\Psi$ are well-defined. In fact, it is  a very natural condition to pose. A \textit{proper} (locally) Lipschitz continuous functional $f:D[-\tau,0]\to\R$ gives rise to the map \begin{equation}\label{eq:rise}    F:D[-\tau,\infty)\to D[0,\infty),\,\varphi\mapsto [t\mapsto f(\varphi_t)],\end{equation}
  and satisfies 
 the relationship
\begin{equation}
	  F(\varphi)(t)=f(\varphi_t),\quad t\geq 0.
\end{equation}
Observe  that $F$ is (lo)lidet; see the definition below (from \cite{article:stojkovic}). Furthermore,  note that lidet maps  are  the  deterministic counterpart of the (random) Lipschitz functionals considered in \cite{book:Rao, book:protter}. 
\begin{definition} A map $F: E[-\tau, \infty) \rightarrow E[0, \infty)$ is  said to be a \texttt{Lipschitz} \texttt{functional\,\,of} \texttt{deter\-ministic} \texttt{type}, abbreviated by  \texttt{lidet}, when it satisfies:
  		\begin{equation}\label{eq:lidet}
  		\begin{aligned}&\text{
  			$\exists K>0$ such that $\forall \varphi,\psi\in E[-\tau,\infty)$ and $\forall t\geq 0$ we have}\\[.3cm]&  \quad\quad\quad\quad\quad\quad\quad\quad\quad\quad\quad\quad\quad\quad\quad|F(\varphi)(t)-F(\psi)(t)|\leq K \sup_{s\in[t-\tau,t]}|\varphi(s)-\psi(s)|.
  		\end{aligned}
  	\end{equation}  	
  	A map $F: E[-\tau, \infty) \rightarrow E[0, \infty)$ is called a \texttt{locally\,\,Lipschitz functional\,\,of} \texttt{deterministic} \texttt{type}, abbreviated by  \texttt{lolidet}, when it satisfies:
  	\begin{equation}\label{eq:lolidet}
  		\begin{aligned}&\text{
  				$\forall n\in\N,\,\exists K_n>0$ such that $\forall \varphi,\psi\in E[-\tau,\infty)$ and $\forall t\geq0$ we have}\\[.3cm]& \quad  \sup_{s\in[t-\tau,t]}|\varphi(s)|\vee|\psi(s)|\leq n\implies |F(\varphi)(t)-F(\psi)(t)|\leq K_n \sup_{s\in[t-\tau,t]}|\varphi(s)-\psi(s)|. 
  		\end{aligned}
  	\end{equation}
\end{definition}

 
Ultimately, while a proper (locally) Lipschitz continuous functional gives rise to a (lo)lidet map, see identification  \eqref{eq:rise}, we observe that a converse also holds true.
\begin{definition}
	A map $F: E[-\tau, \infty) \rightarrow E[0, \infty)$ is called \texttt{autonomous} if
	\begin{equation}
		F(\varphi(s+\cdot\,))(t)=  F(\varphi)(t+s), \quad 
	\end{equation}
for all  $t, s \geq 0.$
\end{definition}
 
\begin{proposition}\label{prop:converse}
For any autonomous \textnormal(lo\textnormal)lidet map $  F: D[-\tau, \infty) \rightarrow D[0, \infty)$, there is a unique proper \textnormal(locally\textnormal) Lipschitz continuous   $f:D[-\tau,0]\to \R$ satisfying the relationship 
\begin{equation}\label{eq:relationship}
	F(\varphi)(t)=f(\varphi_t),\quad t\geq 0,
\end{equation}
for all $\varphi\in D[-\tau,\infty).$  
\end{proposition}
\begin{proof}
	Due to the fact $F$ is (lo)lidet, we obtain   $F(\varphi)(0)=F(\psi)(0)$ for all $\varphi,\psi\in D[-\tau,\infty)$ satisfying $\varphi=\psi$ on $[-\tau,0]$. This results into a   well-defined functional 
		$f:D[-\tau,0]\to\R$, given by
\begin{equation}
\left.	f(\varphi\right|_{[-\tau,0]})=  F(\varphi)(0), \quad \varphi\in D[-\tau,\infty).
\end{equation}   
Clearly, $f$ is (locally) Lipschitz. Under the additional assumption that $  F$ is {autonomous}, we obtain
\begin{equation}
	\left.  F(\varphi)(t)=  F(\varphi(t+\,\cdot\,))(0)=f(\varphi(t+\,\cdot\, )\right|_{[-\tau,0]})=f(\varphi_t),\quad t\geq 0.\label{eq:vgl}
\end{equation}
We deduce that the functional $f$ is also  proper and  its uniqueness is trivial.
\end{proof}


\paragraph{About the integrals}  Let us recall that the first integral in \eqref{eq:interpret} is  a Lebesgue--Stieltjes integral. In here, the $s$  may be replaced by $s-$, because the stochastic processes
$\dint0\cdot {Y(s)}s$ and $\dint0\cdot {Y(s-)}s$
for any $Y=(Y(t))_{t\geq 0}\in\mathbb D\Rplus$ are indistinguishable. Indeed, we have $\mathbb P$-almost surely  
\begin{equation}
	\left|	\dint0t {Y(s)}s-\dint0t {Y(s-)}s\right|\leq \dint0t {|\Delta Y(s)|}s=\dint{[0,t]\cap\{\Delta Y\neq 0\}}{}{|\Delta Y(s)|}s=0, 
\end{equation}
where $\Delta Y(s)=Y(s)-Y(s-),$
since  $t\mapsto \Delta Y(t)(\omega)$ is equal to zero  except for at most countably many times $t$, for almost every $\omega\in\Omega$. See \cite{book:apostol,book:carter,  book:kallenberg, article:horst} for more on Lebesgue--Stieltjes integration.

The second integral in  \eqref{eq:interpret} is a stochastic integral, which is again a semimartingale and takes values in $\mathbb D\Rplus$. A stochastic integral is defined by taking limits of integrated simple predictable processes, where convergence is in the ucp-sense.  We thus follow the construction as in \cite{book:protter}. Note that the integrand of the second integral  is an adapted \caglad process, provided that $X\in\mathbb D[0,\infty)$, making \eqref{eq:interpret} well-defined.
The first integral (after changing $s$ into $s-$)  can,  in fact, be interpreted as a stochastic integral with respect to a (deterministic) semimartingale as well; the two  interpretations coincide.  
In the case of $M=(M(t))_{\geq 0}$ being  a Brownian motion, or  more generally a continuous square integrable  martingale,   one is able to  construct stochastic integrals with convergence in   $L^2(\Omega)$.  The advantage  is that integrands do not need  to be predictably measurable. It suffices, for instance, to assume progressive measurablity together with the integrabilty condition $L^2(\Omega\times [0,T])$, $T\geq 0$; this implies that the stochastic integral is again a square integrable martingale. One can extend this definition to a larger class of integrands by weakening the integrability condition slightly, resulting into an integral which is then merely a local martingale \cite{book:kallenberg,book:karatzas,book:mao,book:revuz}. This allows for more general maps $\mathbf A_\Psi$ and $\mathbf B_\Psi$, in the continuous setting. In practice  it  suffices to have $\mathbf A_\Psi$ and $\mathbf B_\Psi$ going from $\mathbb D\Rplus$ to itself. 
It is worth pointing out that the stochastic integral can as well be extended to the class of  progressively measurable integrands for suitable non-continuous (local) martingales \cite{book:chung,book:jacod,unpublished:timo}, e.g., square integrable  martingales in (HDol). 
However, one now needs to be cautious. For instance, the Lebesgue–Stieltjes integral|if it exists|may no longer coincide with the stochastic integral if the integrand is not predictable; see \cite[Sec. 4.3]{artikel2-MG} for more details. 



 \subsection{Non-autonomous equations}\label{Sec2.1.2}
     \noindent In this section, we  briefly   discuss    stochastic differential equations with   finite time delays where the coefficients $a$ and $b$ may now also explicitly depend on time: 
  \begin{equation}\label{eq:SDDE2}\left\{
 	\begin{array}{rlll}
 		\mathrm d X(t)&=&a(X_t,t)\,\mathrm d t+b(X_{t-},t-) \,\mathrm d M(t), &\quad \text { for } t\geq 0, \\[.05cm]
 		X(u)&=&\Phi(u),& \quad \text { for } u \in[-\tau, 0].
 	\end{array}\right.
 \end{equation}
System \eqref{eq:SDDE2} is shorthand notation for the integral equation
\begin{equation}\label{eq:interpret3}
	X(t)=\Phi(0)+\dint0t{\mathbf A_\Phi(X)(s)}s+\dint0t{\mathbf B_\Phi(X)(s-)}M(s),\quad t\geq 0,
\end{equation}
where    
\begin{enumerate}  
		\item the maps $\mathbf A_\Psi,\mathbf B_\Psi:\mathbb D\Rplus\to\mathbb D\Rplus$, for some fixed $\Psi\in\mathbb D[-\tau,0]$, are defined pathwise for any process $Y\in\mathbb D\Rplus$ by
		\begin{equation}\label{eq:main2}
			\mathbf A_\Psi(Y)(s,\omega):=a\left(Y^\Psi_s(\omega),s\right)\quad\text{and}\quad 	\mathbf B_\Psi(Y)(s,\omega):=b\left(Y^\Psi_s(\omega),s\right),\quad s\geq0;
			\end{equation}
	\item the process $Y^\Psi=(Y^{\Psi}_t)_{t\geq 0}$ in  \eqref{eq:main2} is the segment process of    $Y^\Psi\in\mathbb D[-\tau,\infty)$, where   $Y^\Psi(s)=\Psi(s)$ for $s\in[-\tau,0)$ and $Y^\Psi(s)=Y(s)$ for $s\geq 0$;
	\item the functionals  $a:D[-\tau,0]\times [0,\infty)\to\R$ and $b:D[-\tau,0]\times [0,\infty)\to\R$ in \eqref{eq:main2} are assumed to be \textit{time-proper} (\textit{locally}) \textit{Lipschitz}   (see Definition \ref{def:timeprop}).
\end{enumerate}  
\begin{definition}\label{def:timeprop}
	A map $f:D[-\tau,0]\times \R\to\R$ is called \texttt{time-proper\,\,(locally)} \texttt{Lipschitz} when it satisfies the following two conditions:
	\begin{enumerate}[~~~(i)]
		\item for every   $t\geq 0 $, the first component  $D[-\tau,0]\to \R, \varphi\mapsto f(\varphi,t)$ of $f$ is   (locally) Lipschitz continuous with respect to the supremum norm $\|\cdot\|_\infty$;
		\item   for all $\psi\in D[-\tau,\infty)$,  the mapping\begin{equation}\label{eq:extra2}
			\mathfrak G_ {f,\psi}:\Rplus\to\R,\,	t\mapsto f(\psi_t,t)
		\end{equation}
	is also \cadlag.
	\end{enumerate}
\end{definition}

A (local, maximal, global) solution to problem \eqref{eq:SDDE2} can be defined in a similar fashion as Definition \ref{def:solution}. All results in {\S}\ref{Sec2.1.1} clearly extend to the non-autonomous case. We only highlight the existence and uniqueness result.
 
 
 \begin{theorem}\label{thm:main-exist3} 
 	Suppose $a, b: D[-\tau,0]\times \R \rightarrow \mathbb{R}$ are time-proper  locally Lipschitz functions. Then for every    $\Phi\in \mathbb D[-\tau,0]$, there exists a stopping time $T_\infty$ and a  stochastic process $X$ such that $(X,T_\infty)$ is a local solution to the initial value problem \eqref{eq:SDDE2}. 
 	The stopping time $T_\infty$ can be chosen such that $(X,T_\infty)$ is a maximal solution of   \eqref{eq:SDDE2}. This maximal solution  is   unique. 
 	
If, in addition, $a(\cdot,t)$ and $b(\cdot,t)$ satisfy  the linear growth condition for every $t\geq 0$ with a uniform constant $K_{\rm lin}>0$, then there exists a unique global solution to \eqref{eq:SDDE2}. This global solution is again a semimartingale. 
 \end{theorem}
\begin{proof}
	The proof is completely analogous to the autonomous case. 
\end{proof}
 
We end this section with an example of a \textit{time-proper locally Lipschitz} function $f$, relevant to the study of delay equations with stochastic negative feedback. 

%
%

 \begin{examplex}\label{example:MG}
Let    $h:\R \to\R$ be locally Lipschitz. Define  $f:D[-\tau,0]\times [0,\infty)\to\R$ as
 \begin{equation}f(\varphi,t)=-\gamma(t)+r(t)e^{-\varphi(0)}h\big(e^{\varphi(-\tau)}\big),\quad \varphi\in D[-\tau,0],\end{equation}
where $\gamma, r:[0,\infty)\to[0,\infty)$ are assumed to be bounded functions.
Exploiting the fact that both $h$ and the exponential function are locally Lipschitz continuous, one easily verifies that $\varphi \mapsto f(\varphi,t)$ is locally Lipschitz for any $t\geq 0$ fixed. 
If, in addition, we assume     $\gamma$  as well as   $r$ to be \cadlag functions, we can conclude that
\begin{equation}
	t\mapsto -\gamma(t)+r(t)e^{-\psi(t)}h\big(e^{\psi(t-\tau)}\big) 
\end{equation}
is \cadlag too, for all $\psi\in D[-\tau,\infty)$. From this we conclude that  $f$ is time-proper locally Lipschitz.
Indeed,    taking a pointwise sum and product of \cadlag functions is again \cadlag, and  the composition of a continuous function with a \cadlag function is also \cadlag. Be aware that  the composition of two \cadlag functions is not necessarily right-continuous.
 \end{examplex}




\section{Boundedness in probability of solution  segments}\label{boundednesssolutions}
\noindent In this section, we study  the non-autonomous system
  \begin{equation}\label{eq:nonauto}\left\{
	\begin{array}{rlll}
		\mathrm d X(t)&=&a(X_t,t)\,\mathrm d t+b(X_{t-},t-) \,\mathrm d M(t), &\quad \text { for } t\geq 0, \\[.05cm]
		X(u)&=&\Phi(u)& \quad \text { for } u \in[-\tau, 0].
	\end{array}\right.
\end{equation}
Without further explicit mention, we assume $\Phi\in\mathbb D[-\tau,0)$ and that $a$ and $b$ are time-proper locally Lipschitz. Recall that for any $\varphi\in D[-\tau,0]$ we have   $\|\varphi\|_\infty=\sup_{s\in[-\tau,0]}|\varphi(s)|<\infty,$ because $[-\tau,0]$ is compact. In {\S}\ref{sec:bound-prelim} we provide the relevant definitions and a few additional preliminary notions.
In {\S}\ref{sec:bound-prelimedness} we study the impact of  one-sided  constraints imposed on the nonlinearity $a$ and how this results into  solution  segments being bounded in probability. Finally, in {\S}\ref{sec:blowups} we  discuss how boundedness in probability of solutions and their segments relate to  finite time blowups and global existence.

\subsection{Preliminaries}\label{sec:bound-prelim}
\noindent  Let $I$ be some index set, e.g., take $I=\mathbb N$ or $I=\Rplus$. A family
$(Z_{\eta})_{\eta\in I}$ of real-valued random variables is    \texttt{bounded\,\,above {\normalsize\textnormal{(resp.,}} below{\normalsize\textnormal)} in\,\,probability}  if for every $\varepsilon>0$ there exists a real number $M_{\varepsilon} \in \mathbb{R}$ such that for all $\eta \in I$ we have 
 \begin{equation}
 \mathbb{P}\left(Z_{\eta} >  M_{\varepsilon}\right) < \varepsilon \quad\quad\quad\quad\quad\quad\quad\big(\text {resp., }\mathbb{P}\left(Z_{\eta} <  M_{\varepsilon}\right) <  \varepsilon \big).\quad
\end{equation}
 Differently put, we have $(Z_{\eta})_{\eta\in I}$ is bounded above (resp., below) in probability if and only if
 \begin{equation}
 	\lim_{R\to\infty}\sup_{\eta\in I}  \mathbb{P}\left(Z_{\eta} \geq  R\right)=0\quad\quad\quad\left(\text {resp., }	\lim_{R\to\infty}\sup_{\eta\in I}  \mathbb{P}\left(Z_{\eta} \leq -R\right)=0 \right).
 \end{equation}
A real-valued family $(Z_{\eta})_{\eta\in I}$ is said to be \texttt{bounded\,\,in\,\,probability} if it is both bounded above and below in probability. That is, for every $\varepsilon>0$ there exists a real number $M_{\varepsilon} \in \mathbb{R}$ such that for all $\eta \in I$ we have 
$
\mathbb{P}\left(|Z_{\eta}| >  M_{\varepsilon}\right) < \varepsilon,
$
or, in short,
$
\lim_{R\to\infty}\sup_{\eta\in I}  \mathbb{P}\left(|Z_{\eta}| \geq  R\right)=0.
$ We can generalise this  concept.
\begin{definition}
	Let $(X,\|\cdot\|)$ be a normed space. An $X$-valued family $(Z_{\eta})_{\eta\in I}$ is  \texttt{bounded\,\,in prob\-ability} if for every $\varepsilon>0$ there exists a constant $M_{\varepsilon} \in \mathbb{R}$ such that for all $\eta \in I$ we have 
	\begin{equation}
	\mathbb{P}\left(\|Z_{\eta}\| >  M_{\varepsilon}\right) < \varepsilon.
	\end{equation}
	In other words, $(Z_{\eta})_{\eta\in I}$ is  {bounded in probability}  when
	$
	\lim_{R\to\infty}\sup_{\eta\in I}  \mathbb{P}\left(\|Z_{\eta}\| \geq  R\right)=0
	$ holds.
\end{definition}

Any $X$-valued random variable, and even any finite sequence of $X$-valued random variables, is bounded in probability. This is an immediate consequence of the continuity of the measure $\mathbb P$.
The following lemma provides another sufficient condition for  a family to be bounded in probability.

 \begin{lemma}\label{lem:handy!}
 	Suppose $(Z_{\eta})_{\eta \in I}$ is a family of $X$-valued random variables with
 	\begin{equation}\label{eq:handy!}
 	\sup _{\eta \in I} \mathbb{E} \|Z_{\eta}\|^{2}<\infty.
 	\end{equation}
 	Then $(Z_{\eta})_{\eta \in I}$ is bounded in probability.
 \begin{proof}An application of Markov's inequality, for each $\eta\in I$,  yields
 	\begin{equation}
 	\sup_{\eta\in I}	\mathbb{P}\left(\|Z_{\eta}\|\geq R\right)=\sup_{\eta\in I}\mathbb{P}\left(\|Z_{\eta}\|^{2}>R^{2}\right)\leq \frac{1}{R^2}\sup_{\eta\in I}\mathbb{E} \|Z_{\eta}\|^{2}.
 	\end{equation}
 	 Taking $R\to \infty$  shows that $(Z_{\eta})_{\eta \in I}$ is bounded in probability. 
 \end{proof}
 \end{lemma}


If we have two families of real-valued random variables $(Y_\eta)_{\eta\in I}$ and $(Z_\eta)_{\eta\in I}$ such that
\begin{equation}
	Y_\eta \leq Z_\eta,\quad \forall \eta\in I\quad\quad\quad\big(\textnormal{resp., } Y_\eta \geq Z_\eta,\quad \forall \eta\in I\big),
\end{equation}
and  $(Z_\eta)_{\eta\in I}$  is bounded above (resp., below)  in probability, then   $(Y_\eta)_{\eta\in I}$ is so, too. 
Another useful observation is that  if two families $(Y_\eta)_{\eta\in I}$ and  $(Z_\eta)_{\eta\in I}$ are  bounded, bounded above or bounded below in probability, then so is its sum  $(Y_\eta+Z_\eta)_{\eta\in I}$. The latter  follows  from  $
	\mathbb P(Y_\eta+Z_\eta >R)\leq \mathbb  P(Y_\eta>R/2)+ \mathbb P(Z_\eta>R/2)$
and $
	\mathbb P(Y_\eta+Z_\eta <R)\leq \mathbb P(Y_\eta<R/2)+\mathbb P(Z_\eta<R/2).
$


\begin{remark}
Throughout this work, we will often encounter  that $(\|Y_t\|_\infty)_{t\geq 0}$ needs to be bounded in probability (in $\mathbb R$). One could  reformulate this as  
   $(Y_t)_{t\geq 0}$   being  bounded in probability in the Banach space $ (D[-\tau,0],\|\cdot\|_\infty).$ A reason to be cautious  is because we know   by  Corollary \ref{remark:segment}     that  segments $Y_t$ are no $(D[-\tau,0],\|\cdot\|_\infty)$-random variables; measurability cannot be with respect to the Borel $\sigma$-algebra induced by the uniform topology. 
We could speak of bounded in probability in $ (D[-\tau,0],\|\cdot\|_\infty)$ while simultaneously  we consider the Borel $\sigma$-algebra induced by the Skorokhod topology, but this  deviates from the conventions in Appendix \ref{app:B} and may lead to confusion. There is of course nothing to worry about when we restrict to  the continuous setting $(C[-\tau,0],\|\cdot\|_\infty).$
   
	
	  
\end{remark}

\subsection{The effect of one-sided bounds on the deterministic dynamics}\label{sec:bound-prelimedness}
 A sufficient condition for global existence (and uniqueness) of solutions to \eqref{eq:nonauto}  would be to assume that $a$ and $b$ are of linear growth.  In particular, this holds when $a$ and $b$ are  bounded. In the next proposition, we show that a finite time blowup to $+\infty$ (resp., $-\infty)$ cannot occur when $a$ is bounded from above (resp., below).

\begin{proposition}\label{prop:global} Assume $M=(M(t))_{t\geq 0 }$ is of class \textnormal{(HDol)}  or \textnormal{(HSqL)} and suppose there is a  constant $\beta\geq 0$ with
 \begin{equation}
 	b(\varphi,\,\cdot\,)^2\leq \beta^2,\quad\text{for all }\varphi\in D[-\tau,0].
 \end{equation}
 If there is a non-negative constant $\alpha_{\textnormal{max}}\geq 0$ $($resp.,  $\alpha_{\textnormal{min}}\geq 0)$ such that  
	 \begin{equation}
	 	a(\varphi,\,\cdot\,)\leq \alpha_{\textnormal{max}}\quad\big(\text{resp., 
   } a(\varphi,\,\cdot\,)\geq -\alpha_{\textnormal{min}}\big),\quad\text{for all }\varphi\in D[-\tau,0], 
	 \end{equation}
then for any local solution $X$ to initial value problem \eqref{eq:nonauto} we have \begin{equation}\textstyle \liminf _{t\to T_\infty(\omega)}X(t,\omega)=-\infty\quad\big(\text{resp., } \limsup _{t\to T_\infty(\omega)}X(t,\omega)=
+\infty\big)\end{equation} 
  whenever $T_\infty(\omega)<\infty$ holds, for almost every $\omega\in\Omega$. 

\end{proposition}
\begin{proof} It suffices to take $M$ to be of class (HDol). By symmetry, we only need to consider the case where we have an upper bound by $ \alpha_{\max}.$ 
    We may assume  $\sup _{\theta \in[-\alpha,0]}|\Phi(\theta)| \leq R$  for some  $R\geq 0$ without loss; see  the  proof of Proposition \ref{thm:interstep} for more details. 
		Let $(X,T_\infty)$  denote any solution to equation \eqref{eq:nonauto} with such an initial condition and suppose $(T_k)_{k\geq 1}$ is some sequence of finite stopping  times as in  Definition \ref{def:solution}. Fix  an arbitrary instant $t\geq 0$, then
	\begin{equation}\begin{aligned}		
		\sup_{0\leq s\leq t}X(s\wedge T_k)&\leq R+\alpha_{\max} t+\sup_{0\leq \eta \leq t\wedge T_k}\left|\int_0^{\eta}b(X_{s-},s-)\,\mathrm dM(s)\right|\\
		&\leq R+\alpha_{\max} t+\sup_{0\leq \eta <t\wedge T_\infty}\left|\int_0^{\eta}b(X_{s-},s-)\,\mathrm dM(s)\right|.	
	\end{aligned}\end{equation}
Appealing to the Burkholder--Davis--Gundy inequality \cite[Ch. VII]{DelMey80}---or alternatively,   exploiting\linebreak Doob's maximal inequality \cite[Thm. 1.3.8]{book:karatzas} in combination with results in \cite[Ch. I.4]{book:jacod} regarding the predictable quadratic variation---gives us
\begin{equation}
	\begin{aligned}
	   	\mathbb{E} \sup _{0\leq \eta \leq  t\wedge T_k}\left|\int_{0}^\eta b\left(X_{s-},s-\right) \mathrm d M(s)\right|^{2}  
	&  \leq 4 \mathbb{E} \int_{[0,t\wedge T_k]} b\left(X_{s-},s-\right)^{2} \mathrm d [M](s) \\
 	&  = 4 \mathbb{E} \int_{[0,t\wedge T_k]} b\left(X_{s-},s-\right)^{2} \mathrm d \langle M\rangle(s) \\
	& \leq 4 \lambda \beta^{2}  t,\label{eq:upperbound}
\end{aligned}
\end{equation} for some $\lambda>0$, as a result of $M$ being of class (HDol).
Since the upper bound in \eqref{eq:upperbound} does not depend on $T_k$, for any integer $ k\geq 1$, we obtain
\begin{equation}
	\mathbb{E} \sup _{0\leq \eta <t\wedge T_\infty}\left|\int_{0}^\eta b\left(X_{s-},s-\right) \mathrm d M(s)\right|^{2}  \leq  4 \lambda \beta^{2}  t.
\end{equation}
An application of Markov's inequality yields the following statement: for any $\varepsilon>0$ there exists a sufficiently large value $R_\varepsilon(t)>0$ such that
\begin{equation}
	\mathbb P\left(\sup_{0\leq \eta <t\wedge T_\infty}\left|\int_0^{\eta}b(X_{s-},s-)\,\mathrm dM(s)\right|\geq R_\varepsilon(t)\right)<\varepsilon.
\end{equation}
In conclusion, for any $\varepsilon>0$ there exists a measurable set $\Omega_\varepsilon\subset \Omega$ with $\mathbb P(\Omega_\varepsilon)\geq 1-\varepsilon$ and
\begin{equation}
	\sup_{0\leq s\leq t}X(s\wedge T_k(\omega),\omega)\leq R+\alpha_{\max}t+R_\varepsilon(t),\quad \omega\in\Omega_\varepsilon.
\end{equation}
This holds for any integer $k\geq 1$, hence 
\begin{equation}
	 \sup_{0\leq s\leq t\wedge T_\infty(\omega)}X(s,\omega)\leq R+\alpha_{\max}t+R_\varepsilon(t),\quad \omega\in\Omega_\varepsilon.
\end{equation}
As a result, because we have $\mathbb P(\Omega_\varepsilon)\to 1$ as $\varepsilon\to 0$, we see  that $\limsup _{t\to T_\infty(\omega)}X(t,\omega)=
    +\infty$ simply cannot happen unless $T_\infty(\omega)=\infty$, which holds true for almost every $\omega\in\Omega$. 
\end{proof}


One-sided constraints on the deterministic part 
$a$ provide significant control over the stochastic solutions. In particular, the  below demonstrates that if a global solution is bounded in probability, then the segment process is also bounded in probability.

\begin{proposition}\label{thm:interstep}
	Assume $M=(M(t))_{t\geq 0 }$ is of class  \textnormal{(HDol)} or \textnormal{(HSqL)} and suppose  $X=(X(t))_{-\tau\leq t<\infty}$ is a 
  global solution to  problem \eqref{eq:nonauto}. Let there be a non-negative constant $\alpha_{\textnormal{max}}\geq 0$ or $ \alpha_{\textnormal{min}}\geq 0 $ such that  
	 \begin{equation}
	 	a(\varphi,\,\cdot\,)\leq \alpha_{\textnormal{max}}\quad\text{or}\quad a(\varphi,\,\cdot\,)\geq -\alpha_{\min},\quad\text{for all }\varphi\in D[-\tau,0], 
	 \end{equation}
and suppose there is a non-negative constant $\beta\geq 0$ with
 \begin{equation}
 	b(\varphi,\,\cdot\,)^2\leq \beta^2,\quad\text{for all }\varphi\in D[-\tau,0].
 \end{equation}
If $(X(t))_{t\geq 0}$ is bounded in probability,  then so is $(\|X_{t}\|_{\infty})_{t\geq 0}$. In particular, if $(X(t))_{t\geq 0}$ is bounded above in probability,  then $(\sup_{\theta\in[-\tau,0]}X(t+\theta))_{t\geq 0}$ is bounded above in probability, and if $(X(t))_{t\geq 0}$ is bounded below in probability,  then $(\inf_{\theta\in[-\tau,0]}X(t+\theta))_{t\geq 0}$ is bounded below in probability 
\end{proposition}

\begin{proof}
As before, we only consider $M$ to be of class (HDol).
It suffices to assume  $a(\varphi,\,\cdot\,) \leq \alpha_{\max}$ for all $\varphi \in D[-\tau,0];$ the proof for the other case is similar. 

First, observe that for $t \geq 0$ and $\theta \in[-\tau,0]$ arbitrary, we have
\begin{equation}
	 X(t+\theta)=X((t-\tau)\vee0)+\int_{(t-\tau)\vee0}^{t+\theta} a (X_{s},s ) \,\mathrm d s+\int_{((t-\tau)\vee0,t+\theta]} b (X_{s-},s- ) \,\mathrm d M(s),\label{eq:expres1}
\end{equation}
whenever $t\geq -\theta$, and $
	X(t+\theta)=\Phi(t+\theta)$ if $t<-\theta.$ 
This yields
\begin{equation}
	X(t+\theta)\leq X((t-\tau)\vee0)+\alpha_{\max}\tau + \sup _{\eta \in[-\tau,0]}\left|\int_{((t-\tau)\vee 0,(t+\eta)\vee 0]} b(X_{s-},s-) \,\mathrm d M(s)\right|,\label{eq:above}
\end{equation}
for $t\geq -\theta.$ Moreover, for $t \geq 0$ and $\theta \in[-\tau,0]$, with $t\geq -\theta$, we have
\begin{equation}
\label{eq:expres2}X(t)=X(t+\theta)+\int_{t+\theta}^{t} a\left(X_{s},s\right) \mathrm d s+\int_{(t+\theta,t]} b\left(X_{s-},s-\right)\mathrm d M(s).
\end{equation}
Consequently, for $t\geq -\theta$, this gives us
\begin{align} 
	X(t+\theta) &=X(t)-\int_{t+\theta}^{t} a\left(X_{s},s\right) \mathrm d s-\int_{(t+\theta,t]} b(X_{s-},s-)\, \mathrm dM(s) \nonumber\\\label{eq:below}
	& \geq X(t)-\alpha_{\max}\tau-\int_{((t-\tau)\vee 0,t]}  b(X_{s-},s-)\,\mathrm d M(s)+\int_{((t-\tau)\vee 0,t+\theta]}b(X_{s-},s-)\,\mathrm dM(s) \nonumber \\
	& \geq X(t)-\alpha_{\max}\tau-2 \sup _{\eta \in[-\tau,0]}\left|\int_{((t-\tau)\vee 0,(t+\eta)\vee 0]} b(X_{s-},s-)\, \mathrm d M(s)\right|.
\end{align} 
It follows that
\begin{equation}\begin{aligned}\label{eq:bnd}
&\sup _{\theta \in[-\tau,0]}|X(t+\theta)| \leq\sup_{\theta\in[-\tau,0]} |\Phi(\theta)|+|X((t-\tau)\vee 0)|+|X(t)|\\
&\hspace{4cm}+\alpha_{\max}\tau+2 \sup _{\theta \in[-\tau,0]}\left|\int_{((t-\tau)\vee 0,(t+\theta)\vee 0]} b\left(X_{s-},s-\right)\mathrm d M(s)\right|
\end{aligned}
\end{equation}
holds, for $t\geq 0$ and $\theta\in[-\tau,0]$. Importantly, the  term $\sup_{\theta\in[-\tau,0]}|\Phi(\theta)|$ in   \eqref{eq:bnd} enabled us to remove the restriction $t\geq -\theta$ as in \eqref{eq:above} and \eqref{eq:below}.


 Thanks to the Burkholder--Davis--Gundy inequality \cite[Ch. VII]{DelMey80}, we obtain
\begin{equation}
	 \begin{aligned}
	& \mathbb{E} \sup _{\theta \in[-\tau,0]}\left|\int_{((t-\tau)\vee 0,(t+\theta)\vee 0]} b\left(X_{s-},s-\right) \mathrm d M(s)\right|^{2}  \\ & =	\mathbb{E} \sup _{\eta \in[0,\tau]}\left|\int_{((t-\tau)\vee 0,(t-\tau+\eta)\vee 0]} b\left(X_{s-},s-\right) \mathrm d M(s)\right|^{2}\\& 
 	\leq 4 \mathbb{E} \int_{(t-\tau)\vee 0}^{t} b\left(X_{s-},s-\right)^{2} \mathrm d [ M](s) \\
		& \leq 4 \lambda\beta^{2}\tau,
	\end{aligned}
\end{equation}
from which we can conclude together with Lemma \ref{lem:handy!} that 
\begin{equation}\left(\sup _{\theta \in[-\tau,0]}\left|\int_{((t-\tau)\vee 0,(t+\theta)\vee 0]} b\left(X_{s-},s-\right) \mathrm d M(s)\right|\right)_{t \geq 0}\end{equation} is bounded in probability. 

Now assume  $\sup _{\theta \in[-\tau,0]}|\Phi(\theta)| \leq R'$   for some $R'\geq 0$ fixed. Then the process $(\|X_{t}\|_{\infty})_{t \geq 0}$ is bounded in probability, thanks to estimate \eqref{eq:bnd} and by assumption  as $(X(t))_{t\geq 0}$ is bounded in probability. If $\sup _{\theta \in[-\tau,0]}|\Phi(\theta)| $   is not bounded uniformly on the entire sample space $\Omega$, then we need to proceed as follows. Define  for every  integer $R'\in\mathbb Z_{\geq 0}$ the measurable set
\begin{equation}
	\Omega_{R'}:=\left\{\omega\in\Omega:\sup _{\theta \in[-\tau,0]}|\Phi(\theta,\omega)| \leq R'\right\}.
\end{equation}
Then $(\|X_{t}\|_{\infty}\mathbf 1_{\Omega_{R'}})_{t \geq 0}$ is bounded in probability. We obtain that the   family $(\|X_{t}\|_{\infty})_{t \geq 0}$ is bounded in probability, irrespective  of the initial data, since $\mathbb P( \Omega_{R'}^c)=\mathbb P(\|\Phi\|_\infty>R')\to 0$   as $R'\to\infty$. 

The assertions regarding bounded above and below follow similarly, for which one only appeals to the estimate in \eqref{eq:above} and \eqref{eq:below}, respectively. 
\end{proof}

Observe that taking a supremum over the interval $[-\tau,0]$ is simply a particular choice; the proof goes averbitim for any interval $[-\tau^*,0]$, $\tau^*>0$. Moreover, we like to point out the equality \begin{equation}
    \sup_{u\in[t-\tau,t]}\|X_{u}\|_{\infty}=\sup_{u\in[t-2\tau,t]}|X({u})|.
\end{equation}
This gives rise to the following corollary.

\begin{corollary}\label{cor:intervals}  
	 Assume $M=(M(t))_{t\geq 0 }$ is of class \textnormal{(HDol)} or \textnormal{(HSqL)} and suppose   $X=(X(t))_{-\tau\leq t<\infty}$ is a 
  global solution to  problem \eqref{eq:nonauto}. Let there be a non-negative constant $\alpha_{\textnormal{max}}\geq 0$ or $ \alpha_{\textnormal{min}}\geq 0 $ such that  
	 \begin{equation}
	 	a(\varphi,\,\cdot\,)\leq \alpha_{\textnormal{max}}\quad\text{or}\quad a(\varphi,\,\cdot\,)\geq -\alpha_{\min},\quad\text{for all }\varphi\in D[-\tau,0], 
	 \end{equation}
and suppose there is a non-negative constant $\beta\geq 0$ with
 \begin{equation}
 	b(\varphi,\,\cdot\,)^2\leq \beta^2,\quad\text{for all }\varphi\in D[-\tau,0].
 \end{equation}
If $(X(t))_{t\geq 0}$ is bounded in probability,  then for any $\tau^*>0$ the process \begin{equation}\textstyle \left(\sup_{u\in[t-\tau^*,t]}|X({u})|\right)_{t\geq (\tau^*-\tau)\vee0}\end{equation} is bounded   in probability,  and in particular $(\sup_{u\in[t-\tau,t]}\|X_{u}\|_{\infty})_{t\geq \tau}$ is bounded in probability too.
\end{corollary}
Take note that if $(X(t))_{t\geq 0}$ is bounded in probability, then so is $X=(X(t))_{-\tau\leq t<\infty},$ because we have $\lim_{R\to \infty}\sup_{-\tau\leq t\leq 0}\mathbb P(|X(t)|>R)\leq \lim_{R\to 0}\mathbb P(\|\Phi\|_\infty>R)=0.$ Further,
if we only know  that $X$ is a local solution, then we   obtain along similar lines  that $(\|X_t\|_\infty\mathbf 1_{t<T_\infty})_{t\geq 0}$ is bounded in probability, provided   $X$ is bounded in probability. 
This, however, does not  necessarily imply that $(\|X_t\|_\infty)_{t\geq 0}$ is bounded in probability, as we will  see in the next section.



\subsection{Additional comments on finite time blowups}\label{sec:blowups}
In this section, we  show how boundedness in probability of solutions relate to finite time blowups and global existence. 
Let $X=(X(t))_{t\geq 0}$  be  a (local) solution to \eqref{eq:nonauto}  with  explosion time $T_\infty$. By convention, we  set $X=0$ on $[T_\infty,\infty)$. 
If we now suppose that $X$ is uniformly bounded $\mathbb P$-a.s.\ on $[-\tau,T_\infty)$, then this implies   $T_\infty = \infty$ $\mathbb P$-a.s.. 
On the other hand, if $X$ is only known to be bounded in probability, then  $\mathbb P(T_\infty=t)=0$ holds, for all $t\geq 0,$ but this does not exclude the possibility of $\mathbb P(T_\infty<\infty)=\mathbb P(\cup_{t\geq 0}\{T_\infty=t\})>0$ being strictly positive because of the uncountable union.
 Boundedness in probability gives us in particular that   $\lim_{R\to\infty}\sup_{0\leq s\leq t}\mathbb P(|X(s)|>R)=0$ holds, for all $t\geq 0$, while the value of $\mathbb P(T_\infty\leq t)=\lim_{R\to\infty}\mathbb P(\sup_{0\leq s\leq t} |X(s)|>R)$  remains inconclusive due to the fact that the supremum is inside the probability now. 

More control is required to guarantee global existence. For instance, 
 if  the segments are bounded in probability, i.e., if $(\|X_t\|_\infty)_{t\geq 0}$ is bounded in probability, then
 \begin{equation}
 \begin{aligned}
        \mathbb P(T_\infty<\infty)&=\mathbb P\left(\bigcup_{n=1}^\infty \{T_\infty\in[(n-1)\tau,n\tau)\}\right)\\&=\sum_{n=1}^\infty \lim_{R\to\infty}\mathbb P\left(\sup_{(n-1)\tau\leq t\leq n\tau}|X(t)|>R\right)=0,
    \end{aligned}\end{equation}
by continuity of the  measure $\mathbb P.$  Likewise, if the  process $
    (\sup_{\theta\in[-\tau,0]}X(t+\theta))_{t\geq 0} $ is  bounded above in probability, then a finite time blowup towards $+\infty$ does not occur $\mathbb P$-a.s.. Indeed, we have
    \begin{equation}
    \begin{aligned}
        \mathbb P\left(\limsup_{s\to T_\infty}X (s)=+\infty,\,T_\infty<\infty\right)
        =\sum_{n=1}^\infty \lim_{R\to\infty}\mathbb P\left(\sup_{(n-1)\tau\leq t\leq n\tau}X(t)>R\right)=0. \label{eq:above-explosion}
    \end{aligned}
    \end{equation}
        Similarly, when $(\inf_{\theta\in[-\tau,0]}X(t+\theta))_{t\geq 0}$ is bounded below in probability, then a finite time blowup towards $-\infty$ will not happen. 
        
        Relaxing the global existence assumption in Proposition \ref{thm:interstep} still allows us to deduce that the stochastic process    $(\sup_{\theta\in[-\tau,0]}X(t+\theta))_{t\geq 0} $ is bounded above in probability whenever $X$ is bounded above in probability and  $a\leq \alpha_{\rm max}$. This is due to fact that the bound in  \eqref{eq:above} expresses $X(t+\theta)$ in terms of its history only, while \eqref{eq:bnd} contains the term $|X(t)|$ on the right hand side. In particular, we can conclude from \eqref{eq:above-explosion}  that a finite time blowup towards $+\infty$ will not occur $\mathbb P$-a.s., which is in line with Proposition \ref{prop:global}. Nevertheless, as mentioned previously, only assuming either $a\leq \alpha_{\rm max}$ or $a\geq -\alpha_{\rm min}$ is insufficient to deduce that $(\|X_t\|_\infty)_{t\geq 0}$ is bounded in probability; we merely have that the  process $(\|X_t\|_\infty\mathbf 1_{t<T_\infty})_{t\geq 0}$ is  bounded in probability. It is important to note that the latter does not imply the former, as the following example illustrates.

 \begin{examplex} Let $Z:\Omega\to\mathbb R$ be a non-negative random variable with probability distribution $\mu_Z$ absolutely continuous with respect to the Lebesgue measure.
 Consider the  initial value problem
 \begin{equation} \left\{
	\begin{array}{rlll}
		\mathrm d X(t)&=&X(t)^2\,\mathrm d t+\sigma\mathrm d M(t), &\quad \text { for } t\geq 0, \\[.05cm]
		X(u)&=&\Phi(u)& \quad \text { for } u \in[-\tau, 0],
	\end{array}\right.
\end{equation}
where $\Phi\in\mathbb D[-\tau,0]$ is given by $\Phi(u,\omega)=Z(\omega)^{-1}\mathbf 1_{Z(\omega)>0}$ for all $u\in [-\tau,0]$ and  $\omega\in\Omega$ (taking into account we set $\infty\cdot 0=0$). 
Proceeding with $\sigma=0$ yields
\begin{equation}
    X(t,\omega)=\begin{cases}
        \dfrac{1}{Z(\omega)-t},& 0\leq t<Z(\omega),\\
        0,&\rm else.
    \end{cases}
\end{equation}
 Note that 
\begin{equation}
    \mathbb P(\omega\in\Omega:X(t,\omega)>R)= \mathbb P(\omega\in\Omega: t<Z(\omega)<t+R^{-1})=\mu_Z((t,t+R^{-1})).
\end{equation} Since $\mu_Z$ is finite and does not contain any atoms, one can show 
$\lim_{R\to \infty}\sup_{t\geq 0}\mu_Z((t,t+R^{-1}))=0$, hence 
$(X(t))_{t\geq 0}=(X(t)\mathbf 1_{t<T_\infty})_{t\geq 0}$ is bounded in probability. In addition, we have $X(t)=\|X_t\|_\infty$ for $t<T_\infty,$ hence  $(\|X_t\|_{\infty}\mathbf 1_{t<T_\infty})_{t\geq 0}$ is bounded in probability. However, note that the stochastic process   $(\|X_t\|_{\infty})_{t\geq 0}$ is not bounded in probability, because this would imply  $\mathbb P(T_\infty<\infty)=0$.
\end{examplex}

\section{Towards  invariant measures and stationary solutions}\label{sec:towards}

The main objective of this section is to find sufficient conditions  that gaurantee the existence of an invariant measure, hence a stationary solution, of the  autonomous initial value problem 
  \begin{equation}\label{eq:SDDE2.2}\left\{
	\begin{array}{rlll}
		\mathrm d X(t)&=&a(X_t)\,\mathrm d t+b(X_{t-}) \,\mathrm d M(t), &\quad \text { for } t\geq 0, \\[.05cm]
		X(u)&=&\Phi(u),& \quad \text { for } u \in[-\tau, 0].
	\end{array}\right.
\end{equation}
Again, let us assume\footnote{Furthermore, we need that $\Phi$
 is independent of $M$, to  ensure that the segment process $(X_t)_{\geq 0}$ is Markov. This is however automatically satisfied when $M$ is a Brownian motion $W$ or a Lévy process $L$, because $\Phi$ is $\mathcal F_0$-measurable and  $W$ and $L$ are independent of $\mathcal F_0$ by construction. See Appendix \ref{app:D} for more information.} $\Phi\in\mathbb D[-\tau,0)$ and that $a$ and $b$ are proper locally Lipschitz. 
In {\S}\ref{sec:tight_cont} we restrict ourselves to  Brownian noise and in {\S}\ref{sec:special} we allow $M=(M(t))_{t\geq 0}$ to be an integrable L\'evy process. In {\S}\ref{Sec3.4.2} we demonstrate how the general theorems from the previous two sections can be applied to delay equations with stochastic negative feedback.

We shall now define what we mean by stationary distributions and invariant measures.


\begin{definition}\label{def:invariant}
	A solution $X=(X(t))_{-\tau\leq t<\infty}$ to  problem \eqref{eq:SDDE2.2} with maximal existence time, i.e., $T_\infty=\infty$ $\mathbb P$-a.s., is called \texttt{stationary} if the probability distribution of $X(t)$ coincides, for all $t \geq 0$, with the probability distribution of $\Phi(0)$. In that  case, the probability distribution of $\Phi(0)$ is  said to be a \texttt{stationary distribution} of the  delay equation in \eqref{eq:SDDE2.2}.
\end{definition}

In order to find stationary solutions, we will need to study  segment processes, as $C[-\tau,0]$ and $D[-\tau,0]$ are the natural state spaces.

\begin{definition}\label{def:invariantm}
A Borel probability measure $\nu$ on $E[-\tau,0]$\footnote{We allow $E\in\{C,D\}.$ The space $C[-\tau,0]$ is to be endowed with the uniform topology, as usual, but the space $D[-\tau,0]$ must be endowed with the {Skorokhod} topology. We refer to  Appendix \ref{Sec2.2.2} for more information. } is an \texttt{invariant} \texttt{measure} of the delay  equation in \eqref{eq:SDDE2.2} if the segment process $(X_t)_{t \geq 0}$, with initial condition $X_0 = \Phi$ distributed according to $\nu$, has the same distribution $\nu$ at every time $t \geq 0$. The push-forward measure of $\nu$ under the evaluation map
	\begin{equation}
 E[-\tau,0] \rightarrow \mathbb{R}:	\varphi \mapsto \varphi(0)
	\end{equation}
	is then a stationary distribution.
\end{definition}
Note that a stationary distribution is an invariant measure on $\mathbb{R}$. Also, 
an invariant measure (on $E[-\tau,0]$) contains richer information on the dynamical system than a stationary distribution does. While the terms ``stationary distribution'' and ``invariant measure'' are often used interchangeably throughout the literature, we  distinguish these notions in this paper for clarity.




We claim that the main results in this section can be extended to systems of delay equations. 

\subsection{Brownian motion as integrator}\label{sec:tight_cont}
    Let   $W=(W(t))_{t\geq 0}$ be a 
    Brownian motion and consider the initial value problem
    \begin{equation}\label{eq:SDDEW}\left\{
	\begin{array}{rlll}
		\mathrm d X(t)&=&a(X_t)\,\mathrm d t+b(X_{t}) \,\mathrm d W(t), &\quad \text { for } t\geq 0, \\[.05cm]
		X(u)&=&\Phi(u),& \quad \text { for } u \in[-\tau, 0].
	\end{array}\right.
 \end{equation} 
In order to prove the existence of  an invariant measure,   we need to show that there is a segment process that is tight in $C[-\tau,0]$; see  Appendix \ref{Sec2.2.3}
for general tightness results.
\begin{theorem}\label{thm:main_cont}
	Suppose  $X=(X(t))_{-\tau\leq t<\infty}$ is 
 a global solution to       \eqref{eq:SDDEW}.
 If the stochastic  process $\big(\sup_{u\in[t-\tau,t]}\|X_u\|_\infty\big)_{t\geq \tau}$ is bounded in probability, i.e.,
	\begin{equation} 
		\lim _{R \rightarrow \infty} \sup _{t \geq \tau}\mathbb  P\left(\sup _{u\in [t-\tau,t]} \|  X_u \|_\infty>R\right)=0, \label{eq:difficult}
	\end{equation}
	 then  the  segment process $(X_t)_{t\geq 0}$ is tight in $  C[-\tau,0]$. 
  
   In addition, if all other solutions to \eqref{eq:SDDEW} 
   exist globally, then problem \eqref{eq:SDDEW} admits an invariant measure, hence there is at least one stationary solution. 
\end{theorem}
\begin{remark}As a matter of fact, it suffices to show that $(\|X_t\|_\infty)_{t\geq 0}$ is bounded in probability,
instead of the process $(\sup_{u\in[t-\tau,t]}{\|X_u\|_\infty})_{t \geq \tau}$, because we can write
\begin{equation}
\textstyle\sup_{u\in[t-\tau,t]}{\|X_u\|_\infty}=\max\big\{\|X_t\|_{\infty},\|X_{t-\tau}\|_\infty\big\},\quad t\geq \tau.
\end{equation}
By not doing so, we see more clearly how \eqref{eq:difficult} and \eqref{eq:difficult2} in Theorem \ref{cor:main} (see {\S}\ref{sec:special}) relate. Observe that condition \eqref{eq:difficult} is not only sufficient but also necessary for tightness; see Proposition \ref{prop:Dtight}.  
\end{remark}
A direct consequence of either Proposition \ref{thm:interstep} or Corollary \ref{cor:intervals}   is the following.
 \begin{corollary}\label{cor:main_cont}

 Suppose  $X=(X(t))_{-\tau\leq t<\infty}$ is   a global solution to    \eqref{eq:SDDEW} and bounded in probability.
  If there is a non-negative constant $\alpha_{\textnormal{max}}\geq 0$ $($resp., $ \alpha_{\textnormal{min}}\geq 0 )$ such that  
	 \begin{equation}
	 	a(\varphi)\leq \alpha_{\textnormal{max}}\quad\big(\text{resp., } a(\varphi)\geq -\alpha_{\min}\big),\quad\text{for all }\varphi\in C[-\tau,0], 
	 \end{equation}
 and if there exists a non-negative constant $\beta\geq 0$ such that 
 \begin{equation}
	b(\varphi)^2\leq \beta^2,\quad\text{for all }\varphi\in C[-\tau,0],
\end{equation}
	 	 then  the  segment process $(X_t)_{t\geq 0}$ is tight in $C[-\tau,0]$. 

    In addition, if all other solutions to \eqref{eq:SDDEW} 
    exist globally, then problem \eqref{eq:SDDEW} admits an invariant measure, hence there is at least one stationary solution. 
\end{corollary}
In order to prove Theorem \ref{thm:main_cont}, we need the following tightness result. The proposition below is obtained by exploiting Kolmogorov's tightness criterion (Theorem \ref{thm:kolm}). It is worth noting that the following proposition requires no additional assumptions on the noise coefficient $b.$

 \begin{proposition}\label{thm:tightness}
Suppose  $X=(X(t))_{-\tau\leq t<\infty}$  is a global solution to \eqref{eq:SDDEW}.  If  the stochastic process $(\sup_{u\in[t-\tau,t]}{\|X_u\|_\infty})_{t \geq \tau}$ is bounded in probability, then the segment process $({X_t})_{t \geq 0}$ is tight in  $C[-\tau,0]$.
\end{proposition}

\begin{proof}
Let $R>0$ be arbitrary and define the truncated segment process
\begin{equation}X_t^R(\theta):=X(t+\theta)\mathbf{1}_{\{\|X_r\|_\infty \le R \text{ for all } r\in [t-\tau,t]\}},\quad\mbox{for all } \theta\in [-\tau,0] \mbox{ and } t\ge 0.\end{equation}
Fix now $t\ge \tau$ and $\theta_1,\theta_2\in [-\tau,0]$ with $\theta_2>\theta_1$. Introduce the set 
\begin{equation}{{A_{R;t}}}:=\bigl\{\omega\in\Omega\colon\, \|X_r(\omega)\|_\infty\le R \mbox{ for all }r\in [t-\tau,t]\,\bigr\}.\end{equation}
Then $X_t^R=X_t$ on $A_{R;t}$, and
\begin{equation}
\begin{aligned}
\bigl|X_t^R(\theta_2)-X_t^R(\theta_1)\bigr| &=|X(t+\theta_2)-X(t+\theta_1)|\mathbf{1}_{{A_{R;t}}}\\
&\le \int_{t+\theta_1}^{t+\theta_2} |a(X_s)|\mathbf{1}_{{A_{R;t}}}\,\mathrm ds   + \left|\int_{t+\theta_1}^{t+\theta_2} b(X_s)\, \mathrm dW(s)\right| \mathbf{1}_{{A_{R;t}}}.
\end{aligned}\end{equation}
On this set ${A_{R;t}}$,  observe that $b(X_s)$ and $b(X_s)\mathbf{1}_{\{\|X_s\|_\infty\le R\}}$ are equal. Thus, by the local character of stochastic integrals \cite[Thm. VIII.23]{DelMey80}, see also \cite[Thm. II.18]{book:protter}, we obtain
\begin{equation}\int_{t+\theta_1}^{t+\theta_2} b(X_s)\, \mathrm dW(s) = \int_{t+\theta_1}^{t+\theta_2} b(X_s)1_{\{\|X_s\|_\infty\le R\}}\, \mathrm dW(s)\quad\mathbb P\mbox{-a.s.\ on }{A_{R;t}}.\end{equation}
This yields
\begin{equation}
|X_t^R(\theta_2)-X_t^R(\theta_1)| \le \int_{t+\theta_1}^{t+\theta_2} |a(X_s)|\mathbf{1}_{{A_{R;t}}}\,\mathrm ds  +
\left|\int_{t+\theta_1}^{t+\theta_2} b(X_s) \mathbf{1}_{\{\|X_s\|_\infty\le R\}}\mathrm dW(s)\right|.
\end{equation}
Let $B_R := \{\varphi\in C[-\tau ,0] \colon \|\varphi\|_{\infty} \le R \,\}$. Since the functionals $a$ and $b$ are locally Lipschitz, they are bounded on $B_R$,
%
i.e., there are some $\alpha_R,\beta_R>0$ such that
\begin{equation}|a(\varphi)| \le \alpha_R\quad \hbox{ and }\quad |b(\varphi)| \le \beta_R,\quad\hbox{for all } \varphi \in B_R.\label{eq:beta_R}\end{equation}
Therefore, we have\begin{equation}
\begin{aligned}
\mathbb E \bigl| X_t^R(\theta_2)-X_t^R(\theta_1) \bigr|^4 &\le 
8 \mathbb E \left(\int_{t+\theta_1}^{t+\theta_2} |a(X_s)|\,\mathbf{1}_{A_{R;t}}\,\mathrm ds\right)^4\\
&\qquad\qquad + 8\mathbb E\left(\int_{t+\theta_1}^{t+\theta_2} b(X_s)\mathbf{1}_{\{\|X_s\|\le R\}}\, \mathrm dW(s)\right)^4\\
 &\le 8|\theta_2-\theta_1|^4\alpha^4_R + 8C_4\mathbb E\left( \int_{t+\theta_1}^{t+\theta_2} b(X_s)^2\, \mathrm ds\right)^2\\
&\le \tilde{C}|\theta_2-\theta_1|^2,
\end{aligned}\end{equation}
with  $\tilde{C}=8\alpha^4_R\tau^2 +8C_4\beta^4_R,$ and  where $C_4>0$ is the constant from the Burkholder--Davis--Gundy inequality \cite[Ch. IV.4]{book:revuz}.
Since $(X_t(-\tau))_{t\ge 0}$ is tight, we conclude that $(X_t^R)_{t\ge 0}$, for any $R>0,$ is tight in $C[-\tau,0]$ thanks to Theorem \ref{thm:kolm}. 

From the latter, we infer that $(X_t)_{t\ge 0}$ is tight. Indeed, let $\varepsilon>0$ be arbitrarily given. Since the process $(\sup_{u\in[t-\tau,t]}\|X_u\|_\infty)_{t\ge \tau}$ is bounded in probability by assumption, there is an $R_\varepsilon>0$ where
\begin{equation}\textstyle\mathbb P(\sup_{u\in[t-\tau,t]}{\|X_u\|_\infty\le R_\varepsilon}) > 1-\varepsilon/2,\quad\mbox{ for all } t\ge \tau.\end{equation}
Since $(X_t^{R_\varepsilon})_{t\geq 0}$ is tight in $C[-\tau,0]$, there is a compact set $K_\varepsilon\subset C[-\tau,0]$ such that
\begin{equation}
	\mathbb P(X^{R_\varepsilon}_t\in K_\varepsilon)>  1-\varepsilon/2,\quad\mbox{ for all } t\ge 0.
\end{equation}
From this, as $X_t^R=X_t$ on $A_{R;t}$ for any $R>0$, we conclude
\begin{equation}\begin{aligned}
    \mathbb P (X_t\not\in K_\varepsilon)&=\mathbb P(X_t\not\in K_\varepsilon,A_{R_\varepsilon;t})+\mathbb P(X_t\not\in K_\varepsilon,A_{R_\varepsilon;t}^c)\\&\leq \mathbb P(X_t^{R_\varepsilon}\not\in K_\varepsilon)+\mathbb P(A_{R_\varepsilon;t}^c)\\
    &= \mathbb P(X_t^{R_\varepsilon}\not\in K_\varepsilon)\color{black}+\mathbb P(\exists u\in[t-\tau,t]:\|X_u\|_{\infty}> R_\varepsilon)\\
    &\textstyle
    \leq \mathbb P(X_t^{R_\varepsilon}\not\in K_\varepsilon)\color{black}+\mathbb P(\sup_{u\in[t-\tau,t]}\|X_u\|_{\infty}> R_\varepsilon)\\
    &<\varepsilon/2+\varepsilon/2=\varepsilon,
\end{aligned}\end{equation}
for all $t\geq \tau.$ This shows that the partial segment process $(X_t)_{t\geq \tau}$ is tight. Tightness of the finite time horizon process $(X_t)_{0\leq t\leq\tau}$ follows from Proposition \ref{prop:wow}; 
the family $(X_t)_{t\geq 0}$ is  stochastically continuous in $C[-\tau,0]$, since $(X(t))_{t\geq 0}$ is a (stochastically) continuous process \cite[Lem. 2.3]{article:reiss}.
\end{proof}



\begin{proof}[Proof of Theorem \ref{thm:main_cont}]
  The result readily follows from Proposition \ref{thm:tightness} and the Krylov--Bogoliubov existence theorem in Appendix \ref{app:D}.
\end{proof}

\subsection{Integrable L\'evy processes as integrator}\label{sec:special}
In contrast to the continuous setting,   we will now  show  when the segment process $ (X_t)_{t\geq 0}$ is \textit{partially} tight, i.e., the partial segment  process $(X_t)_{t\geq \tau}$, from time $\tau$ and onwards, is tight in $D[-\tau,0]$.  
This turns out to be sufficient to deduce the existence of an invariant measure; we refer to Appendix \ref{app:D} to see what properties no longer hold  for $t<\tau$ when we are in the right-continuous setting.


\begin{theorem}\label{cor:main}
	Assume $M=(M(t))_{t\geq 0 }$ is of class \textnormal{(HIntL)}
 and suppose $X=(X(t))_{-\tau\leq t<\infty}$ is bounded in probability and   a global solution to   \eqref{eq:SDDE2.2}. 
 If the functional $b$ is  bounded by some $\beta\geq 0,$ i.e.,
 \begin{equation}
	b(\varphi)^2\leq \beta^2,\quad\text{for all }\varphi\in D[-\tau,0],\label{eq:bounded-b}
\end{equation}
 and  if the stochastic process $\big(\sup_{u\in[t-\tau,t]}|a(X_u)|\big)_{t\geq \tau}$ is bounded in probability, i.e.,
	\begin{equation} 
		\lim _{R \rightarrow \infty} \sup _{t \geq \tau}\mathbb  P\left(\sup _{u\in [t-\tau,t]} |  a(X_u) |>R\right)=0, \label{eq:difficult2}
	\end{equation}
	 then  the  partial segment process $(X_t)_{t\geq \tau}$ is tight in $  D[-\tau,0]$. 

  In addition, if all other solutions to \eqref{eq:SDDE2.2} 
  exist globally, then problem \eqref{eq:SDDE2.2} admits an invariant measure, hence there is at least one stationary solution. 
  
         
\end{theorem}

In case of Brownian noise, we can compare the slightly different bounded in probability criteria in Theorem \ref{thm:main_cont} and  Theorem \ref{cor:main}.  If the functional $a$ satisfies the linear growth condition, we observe that \eqref{eq:difficult} implies \eqref{eq:difficult2}, but generally speaking such an implication  does not hold. Nevertheless, it turns out that condition \eqref{eq:difficult3} is also sufficient|again|for tight segments in the case of  (HIntL). Clearly, the solution $X$ is bounded in probability when \eqref{eq:difficult3} holds.
\begin{theorem}\label{cor:cor:main}
	Assume $M=(M(t))_{t\geq 0 }$ is of class \textnormal{(HIntL)}
 and suppose $X=(X(t))_{-\tau\leq t<\infty}$ is 
 a global solution to 
 \eqref{eq:SDDE2.2}. 
 If the stochastic process $\big(\sup_{u\in[t-\tau,t]}\|X_u\|_\infty\big)_{t\geq \tau}$ is bounded in probability, i.e.,
	\begin{equation} 
		\lim _{R \rightarrow \infty} \sup _{t \geq \tau}\mathbb  P\left(\sup _{u\in [t-\tau,t]} \|  X_u \|_\infty>R\right)=0, \label{eq:difficult3}
	\end{equation}
	 then  the  partial segment process $(X_t)_{t\geq \tau}$ is tight in $  D[-\tau,0]$.

   In addition, if all other solutions to \eqref{eq:SDDE2.2} 
   exist globally, then problem \eqref{eq:SDDE2.2} admits an invariant measure, hence there is at least one stationary solution. 
   
 \end{theorem}

The goal is to investigate when $ (X_t)_{t\geq \tau}$ is tight in the Skorokhod space $D[-\tau,0]$. The following lemma shifts the problem and introduces a different family of processes of which we are expected to  show its tightness. This approach is entirely inspired by \cite{article:reiss}. 



 \begin{lemma}\label{lem:need2show}
	Suppose $X=(X(t))_{-\tau\leq t<\infty}\in \mathbb D[-\tau,\infty)$  is a stochastic process which is bounded in probability. Then the partial segment process $(X_t)_{t\geq \tau}$, from time $\tau$ and  onwards, is tight in $D[-\tau,0]$ whenever the family 
	\begin{equation}
		\big(X(t+s)-X(t-\tau),s\in[-\tau,0]\big)_{t\geq \tau}\label{eq:need2show222}
	\end{equation}
	is  tight in $ D[-\tau,0].$
\end{lemma}
 	\begin{proof}
 	 Define $Z_{t}(s):=X(t-\tau)$ for all $s \in[-\tau, 0], t\geq 0.$ We subsequently obtain that the family  $(Z_{t})_{t \geq 0}$ of constant processes on $[-\tau,0]$ is tight in $  C[-\tau, 0]$, since $X$ is bounded in probability for all time
   (i.e., on $[-\tau,\infty)$). As a result of Corollary \ref{cor:Ctight}, the family $(Z_{t})_{t \geq 0}$ is $C$-tight.
 	 
 	We proceed by assuming   the family in \eqref{eq:need2show222} is tight in $  D[-\tau,0]$. In other words,  we have  that $(X_t-Z_t)_{t\geq \tau}$  is tight in $  D[-\tau, 0]$. The sum $(\left(X_{t}-Z_{t}\right)+Z_{t})_{t\geq \tau}$ is also tight in $  D[-\tau, 0]$, which is due to Lemma \ref{lem:sum}. Therefore $(X_t)_{t\geq \tau}$ is tight in $  D[-\tau,0]$. 
 	\end{proof}
Our new objective is thus to show that the family  \eqref{eq:need2show222} is tight in $D[-\tau,0]$. In the remainder of this section,  we return  to the notation in  {\S}\ref{sec:existunique} to improve readability.  Let $X$ be a global solution to equation \eqref{eq:SDDE2.2}, then $X$ is in $\mathbb D[-\tau,\infty)$ and satisfies
\begin{equation} \label{eq:formal}
	X(t)=\Phi(0)+ \dint0t{\mathbf A(X)(s)}s+\dint0t{\mathbf B(X)(s-)}M(s),\quad t\geq 0,
\end{equation}
where we write $\mathbf A=\mathbf A_{\Phi}$ and $\mathbf B=\mathbf B_{\Phi}$;  we drop the dependence of  the initial condition.

     
We  provide two different proofs of the proposition below. The semimartingale $M=(M(t))_{t\geq 0}$ is assumed to be of class (HSqL) in one of those proofs.     Showing this result under the less restrictive hypothesis (HIntL) is rather advanced; it makes use of semimartingale characteristics (see Appendix \ref{app:C}) and the proof is to a large extend in line with the proof of \cite[Prop. 4.3]{article:reiss}.


\begin{proposition}\label{prop:main} Assume $M=(M(t))_{t\geq 0 }$ is of class \textnormal{(HIntL)} and suppose   $X=(X(t))_{-\tau\leq t<\infty}$  is a global solution to  \eqref{eq:SDDE2.2} and  bounded in probability. Then the family
		\begin{equation}
		\big(X(t+s)-X(t-\tau),s\in[-\tau,0]\big)_{t\geq \tau}\label{eq:need2show2}
	\end{equation}
		   is tight in $  D[-\tau,0]$ if  $\mathbf B$ is  bounded and  
$
		   	\lim _{R \rightarrow \infty} \sup _{t \geq \tau}\mathbb  P(\sup _{u\in [t-\tau,t]} | \mathbf A(X)(u) |>R)=0$
     holds.
	
\end{proposition}

\begin{proof}[Proof of Proposition \ref{prop:main} assuming \textnormal{(HSqL)}] It suffices to assume  $M=(M(t))_{t\geq 0}$ is a  martingale, as the predictable finite variation part of a process in (HSqL) is directly propertional to $t.$ In fact, the only thing we will need now is that  the predictable quadratic variation process $\langle M\rangle=(\langle M\rangle(t))_{t\geq 0}$ satisfies $\langle M\rangle(t)= \lambda t,$ for some $\lambda > 0$.
We are led to consider the family $(Y_t)_{t\geq \tau}$, where
 \begin{equation} 
 \begin{aligned}
 	Y_{t}(s) &=X(t+s)-X(t-\tau) \\
 	&=\int_{t-\tau}^{t+s}\mathbf A(X)(u) \,\mathrm{d} u+\int_{(t-\tau,t+s]} \mathbf B(X)(u-) \,\mathrm{d} M(u),
 \end{aligned}
\end{equation}
for $s \in[-\tau,0].$ Since $X\in\mathbb D[-\tau,\infty)$, we obtain that the  $Y_t$  are $D[-\tau,0]$-valued random variables; see Remark \ref{remark:segment}. Let us recall $X$ is a semimartingale with respect to the filtered probability space $(\Omega,\mathcal F,\mathbb F,\mathbb P)$  with $\mathbb F=(\mathcal F_s)_{-\tau\leq s<\infty}$. Introduce now for every  $t\geq \tau$  the filtration
$
	\mathbb F_t=(\mathcal F_s)_{s\in[t-\tau,t]} 
$ and observe that a segment $Y_t$ can be seen as semimartingale on $[-\tau,0]$ adapted to  $\mathbb F_t$.

Thanks to Lemma \ref{lem:sum}, it suffices to show that the families $(J_t)_{t\geq \tau}$ and $(I_t)_{t\geq \tau}$, defined by
\begin{equation}
	 J_{t}(s)=\int_{t-\tau}^{t+s} \mathbf A(X)(u)  \mathrm{d} u\quad\text{and}\quad   I_{t}(s)=\int_{(t-\tau,t+s]}\textbf{B}(X)(u-) \,\mathrm{d} M(u), \quad  s \in[-\tau,0],\label{eq:It-Jt}
\end{equation} 
are tight families, as  all the $J_t$ are $C[-\tau,0]$-valued random variables; see Corollary \ref{cor:Ctight}. 
We obtain ($C$-)tightness of $(J_t)_{t\geq \tau}$ in case the following holds:
	\begin{equation} 
	\lim _{R \rightarrow \infty} \sup _{t \geq \tau}\mathbb  P\left(\sup _{u\in [t-\tau,t]} | \mathbf A(X)(u) |>R\right)=0,
\end{equation}	
which is  immediate from Lemma \ref{lem:reduction}.
 
 Consequently, we are left with showing that $(I_t)_{t\geq \tau}$ is a tight family. 
 We have  that $I_t$ is a square integrable martingale on $(\Omega,\mathcal F,\mathbb F_t,\mathbb P)$. The predictable quadratic variation of $(I_t(s))_{s\in[-\tau,0]}$ reads
 \begin{equation}
	\langle I_t\rangle (s)=\int_{t-\tau}^{t+s}|\textbf{B}(X)(u-)|^2\,\mathrm d\langle M\rangle (s),\quad s\in[-\tau,0].
\end{equation}
 In accordance with Theorem \ref{thm:squaretight}, we are expected to   prove that  the family of predictable quadratic variations $(	\langle I_t\rangle )_{t\geq \tau}$ is $C$-tight. Since the functional   $\mathbf B$ is  bounded, say by $\beta\geq 0$, we find
\begin{align}
	\langle I_t\rangle(s) \leq  \beta^2[\langle M\rangle(t+s) -\langle M\rangle (t-\tau)]= \lambda \beta^2(s+\tau).
\end{align}
 Each process  $  \langle I_t \rangle $ gets strongly majorised by the same continuous process $s\mapsto \lambda \beta^2(s+\tau)$. According  to Lemma \ref{lem:increasing}, it now suffices to prove $C$-tightness of the continuous-time family $(Q_t)_{t\geq \tau}$, where each $Q_t$ equals the $t$-independent process $s\mapsto \lambda \beta^2(s+\tau)$. The $C$-tightness of $(Q_t)_{t\geq \tau}$ will follow from tightness, due to  Corollary \ref{cor:Ctight}, and tightness is  trivial due to the fact that a family consisting of a single measure only is tight in a complete separable metric space; see Appendix \ref{Sec2.2.3}.
\end{proof}
 \begin{proof}[Proof of Proposition \ref{prop:main}] 
 As in the previous proof, we consider  the family $(Y_t)_{t\geq \tau}$, where
 \begin{equation} 
 \begin{aligned}
 	Y_{t}(s) &=X(t+s)-X(t-\tau) \\
 	&=\int_{t-\tau}^{t+s}\mathbf{A}(X)(u) \,\mathrm{d} u+\int_{(t-\tau,t+s]} \mathbf{B}(X)(u-) \,\mathrm{d} M(u),\label{eq:Yt}
 \end{aligned}
\end{equation}
for $s \in[-\tau,0].$ 
Recall that for all $t \geq \tau$, the process $ I_t=(I_t(s))_{s\in[-\tau,0]} $ defined by 
\begin{equation}
 I_{t}(s)=\int_{(t-\tau,t+s]}\textbf{B}(X)(u-) \,\mathrm{d} M(u), \quad   s \in[- \tau,0],
 \end{equation}
 is a semimartingale on the filtered probability space $(\Omega,\mathcal F,\mathbb F_t,\mathbb P)$ with $
	\mathbb F_t=(\mathcal F_s)_{s\in[t-\tau,t]} 
$.
Combining the results in Example \ref{ex:char} and Proposition \ref{thm:char}
gives us  that the semimartingale characteristic of   $I_t$, denoted by $\left(B_{I_{t}}, C_{I_{t}}, v_{I_{t}}\right)$,  is given by
 \begin{equation}
 	\begin{aligned}
 	B_{I_{t}}(s)&=  \int_{t-\tau}^{t+s}\Big(b\, \mathbf{B}(X)(u-) +\int_{\R} x \mathbf{B}(X)(u-)\left(\mathbf{1}_{\{|x \mathbf{B}(X)(u-)|\leq 1\}}-\mathbf{1}_{\{|x|\leq 1\}} \right) \nu(\mathrm{d} x)\Big) \mathrm{d} u,\\
  C_{I_{t}}(s)&=  \sigma^{2} \int_{t-\tau}^{t+s}\left|\mathbf{B}(X)(u-)\right|^2 \,\mathrm{d} u,
 	\end{aligned}
 \end{equation}
for $s\in[-\tau,0]$ and where we chose $h(x)=x\mathbf{1}_{\{|x|\leq 1\}}$ to be the truncation function, together with
\begin{equation}
    \nu_{I_t}(S,A)=\int_S K_{I_t}(t+s,A)\,\mathrm ds,\quad S\in\mathcal B( [-\tau,0]),\quad  A\in\mathcal B(\R)
\end{equation}
    where the transition kernel is of the form
\begin{equation}
 K_{I_{t}}(u,A)=\int_{\R} \mathbf{1}_{A \backslash\{0\}}(\mathbf{B}(X)(u-) x) \,\nu(\mathrm{d} x),\quad u\in[t-\tau,t],\quad A \in \mathcal{B}(\mathbb{R}). 
 \end{equation}
Appealing to Example \ref{ex:sum} subsequently yields that  the semimartingale   characteristic $\left(B_{Y_{t}}, C_{Y_{t}}, v_{Y_{t}}\right)$ of $Y_t$ is given by $C_{Y_{t}}=C_{I_{t}},$ $v_{Y_{t}}=v_{I_{t}}$, and 
\begin{equation}
 B_{Y_{t}}(s)=B_{I_{t}}(s)+\int_{t-\tau}^{t+s}\mathbf A(X)(u)\,\mathrm{d} u,\quad s\in[-\tau,0].
 \end{equation}
We shall now investigate tightness of $(Y_{t})_{t \geq \tau}$  by means of Theorem \ref{thm:semi}. It suffices to verify that the family $(a_{Y_t})_{t\geq \tau}$ is $C$-tight, which is defined by
 \begin{equation}
 a_{Y_{t}}(s)=\operatorname{TV}\left(B_{Y_{t}}\right)(s)+C_{Y_{t}}(s)+\int_{[-\tau, s] \times \mathbb{R}}(|x|^{2} \wedge 1) \,v_{Y_{t}}(\mathrm{d} u, \mathrm{d} x), \quad s \in[-\tau,0].
 \end{equation}
This is because parts (i) and (ii) of  Theorem \ref{thm:semi} are immediate; part (ii) is a simple consequence of the fact that $\mathbf{B}$ is  bounded.
For any $t\geq \tau$, we have  that  $(a_{Y_{t}}(s))_{s\in[-\tau,0]}$ defines a non-decreasing process.
Hence, due to   Lemma \ref{lem:increasing}, it suffices to find a $C$-tight family $(A_{Y_t})_{t\geq \tau}$ of non-decreasing processes such that it {strongly} majorises the family 
  $(a_{Y_{t}})_{t \geq \tau}$, i.e.,  we need $A_{Y_t}-a_{Y_t}$, for all $t\geq \tau$, to be   a non-decreasing process as well.

We claim that  the process $(a_{Y_{t}}(s))_{s\in[-\tau,0]}$ gets  strongly majorised by $(A_{Y_t}(s))_{s\in[-\tau,0]},$ where the latter is given by
   \begin{equation}
   	\begin{aligned}\label{eq:AYt}
   		A_{Y_{t}}(s)=& \int_{t-\tau}^{t+s}\left|\mathbf{A}(X)(u)\right| \mathrm{d} u  +(s+\tau)\left((|b| \beta+c)+\sigma^{2} \beta^{2}+\int_{\mathbb{R}}\big((\beta^{2} x^{2}) \wedge 1\big) \,\nu(\mathrm{d} x)\right), 
   	\end{aligned}
   \end{equation}
   with $\beta> 0$   some bound on the functional $\mathbf{B}$, and   
  \begin{equation}
   c=\int_{\beta^{-1} \leq|x|\leq 1} \beta|x|\, \nu (\mathrm{d} x)+\nu(\mathbb{R} \backslash[-1,1]).
   \end{equation}
Note that $c$ is finite due to the integrability assumption on the noise $M$.
   The claim basically follows from straightforward estimation. We  demonstrate the most insightful   estimate below:
 \begin{equation}
 \begin{aligned}
 	\operatorname{TV}\left(B_{Y_{t}}\right)(s)=& \int_{t-\tau}^{t+s} \bigg| \mathbf{A}(X)(u)+b\, \mathbf{B}(X)(u-)\\
 	&\quad\quad \left.+\int_{\R} x \mathbf{B}(X)(u-)\left(\mathbf{1}_{\{|x \mathbf{B}(X)(u-)|\leq 1\}}-\mathbf{1}_{\{|x|\leq 1\}} \right) \nu(\mathrm{d} x)  \right| \mathrm{d} u \\
 	\leq& \int_{t-\tau}^{t+s} | \mathbf{A}(X)(u)|+|b|\int_{t-\tau}^{t+s}  |\mathbf{B}(X)(u-)|\,\mathrm du\\
 	&\quad\quad  +\int_{t-\tau}^{t+s}\int_{\R}\left| x \mathbf{B}(X)(u-)\right|\left|\mathbf{1}_{\{|x \mathbf{B}(X)(u-)|\leq 1\}}-\mathbf{1}_{\{|x|\leq 1\}} \right| \nu(\mathrm{d} x)  \,\mathrm{d} u \\
 	\leq & \int_{t-\tau}^{t+s} \left| \mathbf{A}(X)(u)\right|  \mathrm{d} u+(s+\tau)(|b| \beta+c).
 \end{aligned}
 \end{equation}
 Observe that $A_{Y_{t}}$ is a continuous process, for all $t\geq \tau$, thus proving $C$-tightness  reduces to proving tightness, according to  Corollary \ref{cor:Ctight}. The second term of \eqref{eq:AYt} is independent of  $t$. Again, recall that a family consisting of a single measure only is tight in a complete separable metric space; see Appendix \ref{Sec2.2.3}. Hence, it suffices  to show that $(|J|_t)_{t\geq 0}$, defined by
 \begin{equation}
 |J|_{t}(s)=\int_{t-\tau}^{t+s}\left|\mathbf{A}(X)(u)\right| \mathrm{d} u, \quad  s \in[-\tau,0],\label{eq:J_t}
 \end{equation}
is a tight family in $C[-\tau,0]$, thanks to Lemma \ref{lem:sum}. 
An immediate consequence of Lemma \ref{lem:reduction} is that this is realised by the condition
\begin{equation}\label{eq:conclusion}
	\lim _{R \rightarrow \infty} \sup _{t \geq \tau}\mathbb  P\left(\sup _{u\in [t-\tau,t]} | \mathbf{A}(X)(u) |>R\right)=0.
\end{equation}
Recall that the condition above is also sufficient for tightness of the first term in \eqref{eq:Yt}, i.e., $(J_t)_{t\geq \tau}$ in \eqref{eq:It-Jt}.
      In conclusion, it all comes down to verifying   \eqref{eq:conclusion}  in order to conclude that $(Y_t)_{t\geq \tau}$ is a tight family of   $D[-\tau,0]$-random variables.
 \end{proof}

\begin{proof}[Proof of Theorem \ref{cor:main}]
	Combining  the results in Lemma \ref{lem:need2show} and Proposition \ref{prop:main},  together with the Krylov--Bogoliubov existence theorem in Appendix \ref{app:D}, yields the assertion.
\end{proof}
\begin{proof}[Proof of Theorem \ref{cor:cor:main}]
In line with the proof of Proposition \ref{thm:tightness},   introduce for any $R>0$ and $t\geq 0$ the sets
$A_{R;t}=\{\omega\in\Omega\colon\, \|X_r(\omega)\|_\infty\le R \mbox{ for all }r\in [t-\tau,t]\}$. By the local character of stochastic integrals  \cite[Thm. II.18]{book:protter}, we obtain that
\begin{equation}\begin{aligned}
    J_{t}^R(s)&=\int_{t-\tau}^{t+s}\mathbf A(X)(u)\mathbf 1_{\{\|X_u\|_\infty \leq R\}}\,\mathrm du,\\
I_t^R(s)&=\int_{(t-\tau,t+s]}\mathbf B(X)(u-)\mathbf 1_{\{\|X_{u-}\|_\infty\leq R\}}\,\mathrm dM(u),
\end{aligned}
\end{equation}
for $s\in[-\tau,0],$ coincide on the event  $A_{R;t}$ with $J_t$ and $I_t$ as in \eqref{eq:It-Jt}, respectively. Following the proof of Proposition \ref{prop:main} with $\beta$ replaced by $\beta_R$, see \eqref{eq:beta_R}, enables us to conclude that $(J_t^R)_{t\geq \tau}$ is $C$-tight and $(I_t^R)_{t\geq \tau}$ is tight, for any $R>0$. This is because \eqref{eq:conclusion} is automatically satisfied now. Hence,  the family $(Y_t^R)_{t\geq 0}$ defined by $Y_t^R=J_t^R+I_t^R$ is tight in $D[-\tau,0]$ by Lemma \ref{lem:sum}. 
Observe that we have $Y_t^R=Y_t$ on $A_{R;t}$. Just like in the proof of Proposition \ref{thm:tightness},
we find for all $\varepsilon>0$ that there exists a compact set $K_\varepsilon\subset D[-\tau,0]$ such that
\begin{equation}\begin{aligned}
    \mathbb P (Y_t\not\in K_\varepsilon)&=\mathbb P(Y_t\not\in K_\varepsilon,A_{R_\varepsilon;t})+\mathbb P(Y_t\not\in K_\varepsilon,A_{R_\varepsilon;t}^c)\\&\leq \mathbb P(Y_t^{R_\varepsilon}\not\in K_\varepsilon)+\mathbb P(A_{R_\varepsilon;t}^c)\\
    &<\varepsilon/2+\varepsilon/2=\varepsilon,
\end{aligned}\end{equation}
since $(\sup_{u\in[t-\tau,t]}\|Y_u\|_\infty)_{t\geq \tau}$ is bounded in probability by assumption. This implies that the family  $(Y_t)_{t\geq \tau}$ is tight, which completes the proof.
\end{proof}

   \subsection{Delay equations with stochastic  negative feedback}\label{Sec3.4.2}
    

In this section we demonstrate how to apply Theorem \ref{cor:main}. A general approach  would be to search for a real-valued family $(Z_t)_{t\geq \tau}$ that is bounded in probability and such that
\begin{equation}
	\sup _{u\in [t-\tau,t]} | a(X_u)|\leq Z_t,\quad \text{for all }t\geq \tau.
\end{equation}
We  perform this method on   delay equations of the form \eqref{eq:MGeq4}, i.e., delay equations with stochastic negative feedback, as preparation for \cite{artikel2-MG}. Of course, invoking Theorem \ref{cor:cor:main} would also have sufficed.

 \begin{corollary}\label{cor:application}Suppose that $f:\mathbb R\to \mathbb R$ is  locally Lipschitz continuous, positive on $(0,\infty)$, and bounded from above.  In addition, assume    $\gamma,r>0$.  	Consider the    stochastic delay differential equation
 	\begin{equation}
 		\mathrm d Y(t)=\left[-\gamma+re^{-Y(t)}f\big(e^{Y(t-1)}\big)\right] \mathrm d t+a\left(Y_{t}\right) \mathrm d t+b\left(Y_{t-}\right) \mathrm d L(t),\label{eq:MGeq4}
 	\end{equation}
 	where $a,b$ are  time-independent and proper locally Lipschitz. Further, assume $L=(L(t))_{t\geq 0} $ is of class \textnormal{(HIntL).}   	Let  $\alpha,\beta\geq 0$ be  non-negative constants   such that
 	\begin{equation}
 		|a(\varphi)| \leq \alpha \quad \text { and } \quad b(\varphi)^{2} \leq \beta^{2}, \quad \text { for all } \varphi \in D[-1,0].
 	\end{equation}
 	Finally, assume that for every initial process  $\Phi$ 
    the corresponding solution exists globally.  If one of these   solutions   is bounded   in probability,
  then there exists an invariant measure and hence a stationary solution. 
\end{corollary}   
\begin{proof} 
	According to Theorem \ref{cor:main},  it suffices to show
	 	\begin{equation} \label{eq:bound!}
	 	\lim _{R \rightarrow \infty} \sup _{t \geq 1}\mathbb  P\left(\sup _{u\in [t-1,t]} | \mathbf A(Y)(u) |>R\right)=0, 
	 \end{equation}
	for some solution $Y=(Y(t))_{t\geq 0}$ bounded in probability, where  
	\begin{equation}
	\mathbf A(Z)(t)=-\gamma +r e^{-Z(t)}f\big(e^{Z(t-1)}\big)+a(Z_t), \quad t\geq 0.
	\end{equation}
Let $M$ be an upper bound for the nonlinearity $f$. A straightforward computation shows
	\begin{equation}\begin{aligned}
		\sup_{u\in[t-1,t]}|\mathbf A(Y)(u)|&\leq \alpha+\gamma +   r \sup_{u\in[t-1,t]}e^{-Y(u)}|f(e^{Y(u-1)})|\\
		&\leq\alpha+\gamma + r M \sup_{u\in[t-1,t]}e^{-Y(u)}\\
		&\leq\alpha+\gamma + r M \exp{\|Y_t\|_\infty}.	
	\end{aligned}\end{equation}
Hence, due to the fact that the exponential is invertible,  condition \eqref{eq:bound!} is satisfied whenever the  family $(\|Y_t\|_\infty)_{t\geq 0}$ is bounded in probability. Appealing to Proposition \ref{thm:interstep} completes the proof.
\end{proof}

\appendix



\section{Classes of stochastic integrators}\label{app:A}
Throughout this article, we fix a probability space $(\Omega,\mathcal F,\mathbb P)$, let $\mathbb F=(\mathcal F_t)_{t\geq 0}$ be a filtration that satisfies the usual conditions, and consider   \cadlag semimartingales adapted to $\mathbb F$. In both \cite{artikel2-MG} and this paper, we  specifically  work with  the following   classes of \cadlag integrators: (HSpec), (HIntL), (HDol), (HSqL), (HSqLM), (HJudi), (HReg),  and (HRegM). These classes  satisfy  the inclusions
\begin{equation}\begin{array}{ccccccccc}
\text{(HReg)}&\subset &\text{(HJudi)}&\subset &\text{(HSqL)}&\subset &\text{(HIntL)}&\subset &\text{(HSpec)}\\[.3cm]
\cup&&&&\cup&&&&\cup\\[.3cm]
    \text{(HRegM)}&&\subset &&\text{(HSqLM)}&&\subset &&\text{(HDol)}
\end{array}\end{equation}
and a Brownian motion belongs to all  classes.  For stochastic integration theory with respect to continuous integrators we refer to, e.g., \cite{da2014stochastic,book:Kurtz,evans2012stoch,book:kallenberg,book:karatzas,book:kloeden,book:mao,book:revuz,unpublished:peter,twardowska1996wong}.
In the text below we clarify what is  meant by each of these hypotheses and provide examples and additional information accordingly. 
 \begin{itemize}\item[]\begin{itemize}
 \item[(HSpec)] The process $Z=(Z(t))_{t\geq 0}$ starts at zero, i.e., $Z(0)=0$, and is a special semimartingale, i.e., a semimartigale which admits a unique decomposition $Z(t)=A(t)+M(t)$, where $(M(t))_{t\geq 0}$ is a local martingale and $(A(t))_{t\geq 0}$ a finite variation process that is predictably measurable.
 \end{itemize} \end{itemize}
 \begin{itemize}\item[]\begin{itemize}
 \item[(HIntL)] The process $L=(L(t))_{t\geq 0}$ is  a L\'evy process and integrable.
 \end{itemize} \end{itemize}

  L\'evy processes are semimartingales, as a consequence of the  L\'evy--It\^o decomposition \cite{ book:applebaum,book:kyp,book:protter,book:sato}. In particular, any L\'evy process is a special semimartingale if and only if its L\'evy measure $\nu$ satisfies $\int_\mathbb R (x^2\wedge |x|)\nu(\mathrm dx)<\infty$, i.e., when the process has a finite first moment \cite[Prop. II.2.29]{book:jacod}. 
  
  

\begin{itemize}\item[]\begin{itemize}
    \item[\quad(HDol)] The  process $M=(M(t))_{t\geq 0}$  is a square integrable martingale whose Dol\'eans measure $\mu_M$ is absolutely continuous with respect to the product measure $ \mathrm ds\times \mathbb P$---which is abbreviated by $\mu_M\ll \mathrm ds\times \mathbb P$---and  its Radon--Nikodym derivative is bounded by $\lambda>0$. 
  \end{itemize}  \end{itemize}  

The Dol\'eans measure \cite{book:chung,unpublished:timo} is defined on the predictable $\sigma$-algebra $\mathcal P$ by
\begin{equation}
    \mu_M(A)=\int_\Omega\int_0^\infty \mathbf 1_A(s,\omega)\,\mathrm d[M](s,\omega)\mathbb P(\mathrm d\omega)=\mathbb E\int_0^\infty\mathbf 1_A\mathrm d[M],\quad A\in\mathcal P,
\end{equation}
where $([M](t))_{t\geq 0}$, given by $[M](t)=M(t)^2-2\int_0^tM(s-)\,\mathrm dM(s),$  is the quadratic variation process. Because $M$ is square integrable, there exists a unique predictable increasing process $(\langle M\rangle (t))_{t\geq 0}$, known as the predictable quadratic variation or the angle bracket process, such that $M^2-\langle M\rangle$ is a martingale; see \cite[p. 24]{book:jacod} and \cite[p. 116]{book:protter}. The process $\langle M\rangle$ is  known as the compensator of $[M]$. In the continuous setting, we  have  $[M]=\langle M\rangle$. Note that the  predictable quadratic variation may  not exist for a general semimartingale \cite[{p. 125}]{book:protter}. Lastly, a simple computation in line with \cite[p. 33]{book:chung} shows that $\mu_M=\nu_M,$ where
\begin{equation}
     \nu_M(A)=\int_\Omega\int_0^\infty \mathbf 1_A(s,\omega)\,\mathrm d\langle M\rangle (s,\omega)\mathbb P(\mathrm d\omega)=\mathbb E\int_0^\infty\mathbf 1_A\mathrm d\langle M\rangle,\quad A\in\mathcal P.
\end{equation}
Surprisingly,
this observation is---to the best of the authors'  knowledge---nowhere highlighted.  We refer to \cite{Bosch24} for  more details. 

The importance of the bounded Radon--Nikodym derivative assumption in (HDol) is due to the fact that this results into the following inequality: 
\begin{equation}
     \int_{[0,t]\times \Omega}f^2\mathrm d\mu_M=\mathbb E\int_0^tf(s)^2\mathrm d[M](s)=\mathbb E\int_0^tf(s)^2\mathrm d\langle M\rangle(s)\leq \lambda \mathbb E\int_0^tf(s)^2\,\mathrm ds,\label{eq:Dol-ineq}
 \end{equation}
for any appropriate process $(f(t))_{t\geq 0}$. This property is exploited throughout the entire paper.
Furthermore, note that all results  assuming  (HDol)
are stated for square integrable (true) martingales, but by localisation extend to locally square integrable local martingales. Recall that  for continuous processes  any local martingale is also   locally square integrable \cite[p. 26]{unpublished:peter}, but this is not necessarily true for  \cadlag processes and so there is a distinction.

\begin{itemize}\item[]\begin{itemize}
 \item[(HSqL)] The process $L=(L(t))_{t\geq 0}$ is a  L\'evy process and square integrable.
 \end{itemize} \end{itemize}

 \begin{itemize}\item[]\begin{itemize}
 \item[(HSqLM)] The process $L=(L(t))_{t\geq 0}$ is of class (HSqL) and  a (true) martingale, therefore a square integrable L\'evy martingale.
 \end{itemize} \end{itemize}

Suppose   $M$ is a square integrable L\'evy martingale. Then the L\'evy--Khintchine formula \cite{book:applebaum, book:sato} tells us that the predictable quadratic variation is deterministic and given by $\langle M\rangle(t)= \lambda t,$  $ t\geq 0,$ where $\lambda=\mathbb E[M(1)^2].$   This implies  \eqref{eq:Dol-ineq}; with equality in fact. We obtain
\begin{equation}
     \mu_M( (s,t]\times A))=\lambda(t-s)\mathbb P(A)=\lambda \mathrm (\mathrm ds\times \mathbb P)( (s,t]\times A)),\quad 0\leq s<t,\,A\in\mathcal P,
 \end{equation}
 hence $\mu_M\ll\mathrm ds\times \mathbb P$, and the Radon--Nikodym derivative is the constant $\lambda$. In conclusion, square integrable L\'evy martingales satisfy (HDol).

 Let $(\gamma,\sigma^2,\nu)$ denote the characteristic triplet of a L\'evy process $L=(L(t))_{t\geq 0}$ with respect to the truncation function $x\mapsto x\mathbf 1_{[-1,1]}(x)$. If $L$ is of class (HSqL), then we know from the L\'evy--It\^o decomposition that $L$ can be written as
 \begin{equation}
     L(t)=\left(\gamma + \int_{|x|>1}x\nu(\mathrm dx)\right)t+M(t),\quad t\geq 0,\label{eq:SqL}
 \end{equation}
 where $M=(M(t))_{t\geq 0}$ is the martingale part of $L$, which is   of class (HSqLM). This means that in practice, when considering stochastic differential equations, it suffices to take processes in (HSqLM) instead of  (HSqL); the predictable part of $L$ is directly proportional to $t$ and  can be substituted in the other part of the equation. So, for
$
    \mathrm dX(t)=a(X_t,t)\,\mathrm d t+b(X_{t-},t-) \,\mathrm d L(t),
$ we can also write it as $
    \mathrm dX(t)=a_{\rm new}(X_t,t)\,\mathrm d t+b(X_{t-},t-) \,\mathrm d M(t),
$ with $a_{\rm new}(X_t,t)=a(X_t,t)+(\gamma+\int_{\{|x|>1\}}x\nu(\mathrm dx))b(X_t,t)$.





 \begin{itemize}\item[] \begin{itemize}
 \item[(HJudi)] The process $L=(L(t))_{t\geq 0}$ is a square integrable L\'evy process that is of finite intensity, i.e., $\nu(\mathbb R)<\infty$, where $\nu$ is the associated L\'evy measure.
\end{itemize}\end{itemize}  

The finite intensity property above implies that  $L$ has a finite number of jumps on  any compact time interval \cite[Thm. 21.3]{book:sato}. Either \cite[Thm. 2.3.9]{book:applebaum} or \cite[Lem. 2.8]{book:kyp} subsequently tells us that  any $L$ of type (HJudi)  is a \texttt{jump\,\,diffusion\,\,process}. That is, a sum of two independent processes: a Brownian motion $W=(W(t))_{t\geq 0}$ which is scaled with the dispersion coefficient $\sigma^2$ and includes a drift; and a compound Poisson process $Z=(Z(t))_{t\geq 0}$   with jump measure $\frac1{\nu(\mathbb R)}\nu.$ Indeed, we have
\begin{equation}
    L(t)=\left(\gamma - \int_{\{|x|\leq 1\}}x\nu(\mathrm dx)\right)t+\sigma W(t)+Z(t),\quad Z(t)=\sum_{k=1}^{N(t)}Z_k,\quad t\geq 0.\label{eq:not_canon}
\end{equation}
We say  $N=(N(t))_{t\geq 0}$  is a Poisson process associated to the L\'evy process $L$.       
Writing the above in terms of the canonical decomposition of special semimartingales, as in \eqref{eq:SqL}, yields
\begin{equation}
    L(t)=\left(\gamma + \int_{|x|>1}x\nu(\mathrm dx)\right)t+\sigma W(t)+\big[Z(t)-\lambda_N \mathbb EZ_1t\big],\quad t\geq 0.
\end{equation}
In particular, we have $\lambda_N \mathbb EZ_1=\int_{\mathbb R}x\nu(\mathrm dx).$ A L\'evy process of finite intensity is a martingale if and only if $\gamma = -\int_{\{|x|>1\}}x\nu(\mathrm dx).$ Suppose  $Z_1$ is centred, i.e.,  $\mathbb EZ_1=0,$ then we have  $\gamma = \int_{\{|x|\leq 1\}}x\nu(\mathrm dx)$ and  $Z$ is a compound Poisson process.

Note that  $\mathbb EZ_1^2<\infty$ holds if and only if $L$ is of class (HJudi), and then the quadratic variation  and its compensator (whether $L$ is a martingale or not) equal
\begin{equation}
    [L](t)=\sigma^2t+\sum_{s\leq t}(\Delta Z(t))^2=\sigma^2t+\sum_{k=1}^{N(t)}Z_k^2\quad\text{and}\quad\langle L\rangle(t)=\lambda t,\quad t\geq 0,
\end{equation}
respectively, where $\lambda =\sigma^2+\lambda_N\mathbb EZ_1^2$ and  $\lambda_N=\mathbb E [N(1)]$  is the intensity of the Poisson process $N$.


\begin{itemize}\item[] \begin{itemize}
 \item[(HReg)] The process $L=(L(t))_{t\geq 0}$ is  a L\'evy process of class (HJudi)  and   satisfies the following two additional properties:
  \begin{itemize}\item[] \begin{itemize}
     \item[(P1)] the process experiences no continuous drift, i.e., $\gamma = \int_{\{|x|\leq 1\}}x\nu(\mathrm dx);$
      \item[(P2)] jumps are $\mathbb P$-a.s.\ uniformly bounded by some $\zeta\geq 0,$ i.e., we have $|\Delta L(t)|\leq \zeta $  $\mathbb P$-a.s., where $\Delta L(t)=L(t-)-L(t)$ and $L(t-)=\lim_{s\nearrow t}L(s).$
 \end{itemize}\end{itemize}
 Such processes are referred to as \texttt{regulated\,\,L\'evy\,\,processes}.
 \end{itemize}\end{itemize}
  \begin{itemize}\item[] \begin{itemize}
 \item[(HRegM)] The process $L=(L(t))_{t\geq 0}$ is a regulated L\'evy process and a martingale,
 hence called a \texttt{regulated\,\,L\'evy\,\,martingale}.
\end{itemize}\end{itemize} 
Regulated Lévy processes do not arise in the present paper, but they play a central role in \cite{artikel2-MG}.





\section{Preliminaries on properties of segments}\label{app:B}

Suppose $(\Omega,\mathcal F,\mathbb F,\mathbb P)$  is a filtered probability space satisfying the usual conditions. We will encounter $X$-valued random variables, where $X$ is some topological space; typically, a normed or metric space. We turn $X$ into a measurable space by taking  the Borel $\sigma$-algebra $\mathcal B(X)$ into account. 

Appendices \ref{Sec2.2.2} and \ref{sec:AA}  cover the basics of the function space $D[-\tau,0]$ and its possible topologies. In particular, we show that   the segment process $(Y_t)_{t\geq 0}$ of a process $(Y(t))_{-\tau\leq t<\infty}\in\mathbb D[-\tau,\infty)$ may be regarded as an $\mathbb F$-adapted $D[-\tau,0]$-valued stochastic process\footnote{That is, each $Y_t:(\Omega,\mathcal F_t)\mapsto (X,\mathcal B(X))$ is a $\mathcal F_t$-measurable $X$-valued random variable with $X=D[-\tau,0]$.},  but one needs to be cautious  as  the  topology matters. Next, in Appendix \ref{Sec2.2.3} we introduce the notion of tightness and provide  necessary and sufficient conditions for segment processes to be tight. For the proofs in {\S}\ref{sec:towards}, we require the additional tightness results  in Appendices \ref{sec:applications} and \ref{app:C}.

\subsection{Uniform topology and Skorokhod topology}\label{Sec2.2.2}
  \noindent An extensive study of both the uniform topology and Skorokhod topology on the space of \cadlag functions can be found in, e.g., \cite{book:billingsley,book:jacod}; they examine \cadlag functions with  domain  $[0,1]$ and $[0,\infty)$, respectively.  Either setting is easily converted to  one with  \cadlag functions defined on $[-\tau,0]$, and vice versa. 
  \

  We can equip  $C[-\tau,0]$ and $D[-\tau,0]$ with the supremum norm $\|\cdot\|_\infty$, i.e., we consider \begin{equation}(C[-\tau,0],\|\cdot\|_\infty) \quad\text{and}\quad (D[-\tau,0],\|\cdot\|_\infty),\end{equation} 
  and note that both are Banach spaces. Further, we  say that   $\|\cdot\|_\infty$ induces the \texttt{uniform topology} on the function spaces. An   additional feature is that $(C[-\tau,0],\|\cdot\|_\infty)$  is a separable space, while $(D[-\tau,0],$ ${\|\cdot\|_\infty})$ fails to be separable \cite[p. 325]{book:jacod}. 
    Recall that separability is a topological property,
    which leads us to the Skorokhod topology.

  \begin{definition}
  	Let $a,b\in\R$ with $a<b.$ The space $D[a, b]$ of \cadlag functions $\varphi:[a,b]\to\R$ can be endowed with the \text{Skorokhod metric} $d_{S}$\index{$d_S$}, which is given by
  	\begin{equation}\label{eq:sko1}
  	d_{S}(\varphi, \psi):=\inf _{\lambda \in \Lambda[a, b]}\big(\|\varphi \circ \lambda-\psi\|_{\infty}+\|\mathrm{Id}-\lambda\|_{\infty}\big),\quad \varphi,\psi\in D[a,b],
  	\end{equation}
  	where $\Lambda[a, b]:=\{\lambda:[a, b] \rightarrow[a, b]: \lambda$ is an increasing homeomorphism$\} .$ The metric $d_S$ induces a topology called the \texttt{Skorokhod\,\,topology}. Any $\lambda\in\Lambda[a,b]$ is called a \texttt{change\,\,of\,\,time}.
  \end{definition}
  
  Clearly, we have $d_S(\varphi,\psi)\leq \|\varphi-\psi\|_\infty$---choose  $\lambda=\text{Id}$ as homeomorphism---which yields that the Skorokhod topology is weaker than the uniform topology.  A slightly less trivial fact is that the subspace topology of the Skorokhod topology on the space of continuous functions $C[a,b]$ coincides with the uniform topology  \cite[Prop. 1.17]
  {book:jacod}.
In the uniform topology, two functions $\varphi$ and $\psi$ are near one another  if the graph of $\varphi$ can be carried onto the graph of the function  $\psi$
by a uniformly small perturbation of the ordinates, i.e., a uniform deformation in the vertical axis, while the abscissas is
kept fixed, i.e., no deformations are made on the horizontal  axis. On the other hand, we also allow a uniformly small change of the
time scale in the Skorokhod topology. While doing this, however, we need to take a penalisation factor into account. We illustrate the effect of a change of time in Figure \ref{fig:Skorokhod}.

\begin{figure}[!t]
	\centering
 \includegraphics[width=0.307\linewidth]{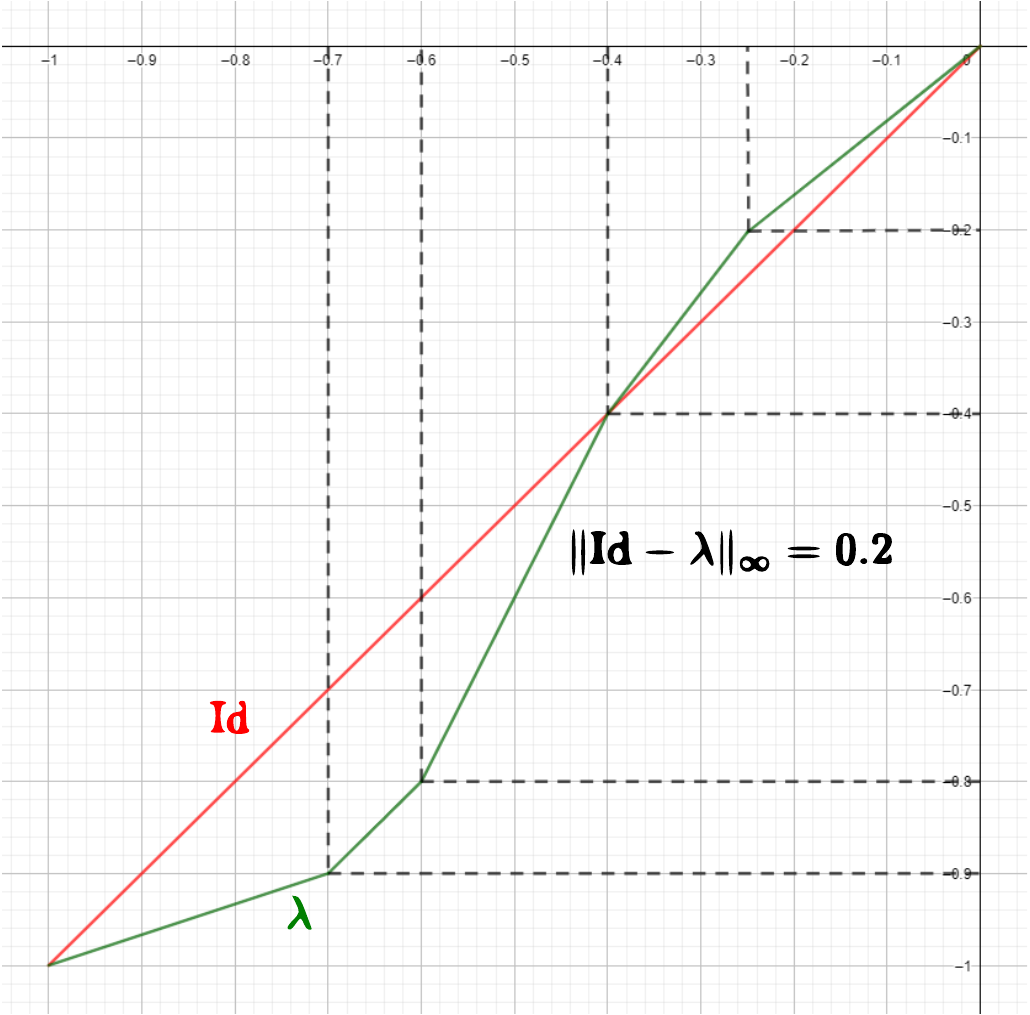}
	\includegraphics[width=0.332\linewidth]{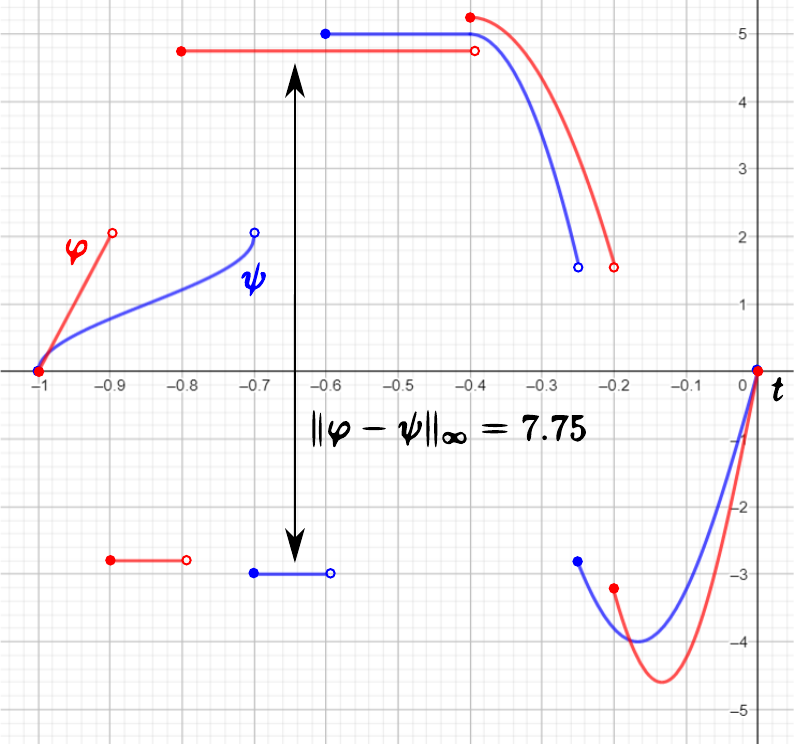}
	\includegraphics[width=0.332\linewidth]{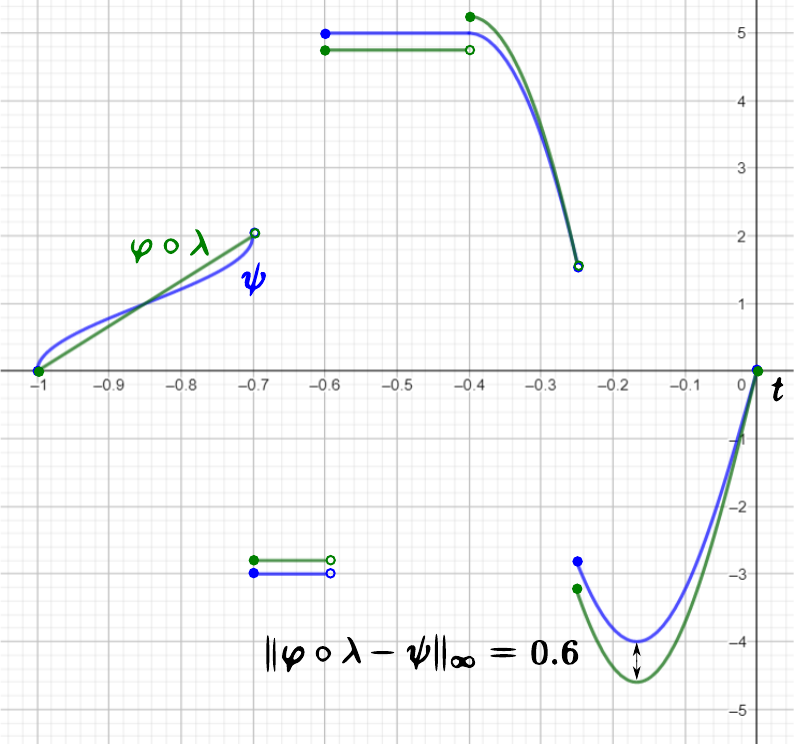}
	\caption{On the left, we see an illustration of some change of time $\lambda\in\Lambda[-1,0]$. In the centre, we see two \cadlag functions $\psi$ and $\varphi$ in $D[-1,0]$ with $\|\varphi-\psi\|_\infty=7.75.$ The right illustration shows the graph of the two \cadlag functions  $\psi$ and $\varphi\circ \lambda$ whose difference in supremum equals  0.6.}
	\label{fig:Skorokhod}
\end{figure}

We wish to endow the space of \cadlag functions $D[-\tau,0]$ with a topology for which it becomes separable, because this is essential  for tightness. One can achieve this via the Skorokhod topology; see  Theorem \ref{thm:Polish}. Nevertheless, the space of \cadlag functions is not complete under the metric $d_S$ \cite[Ex. 12.2]{book:billingsley}. Due to the fact that separability is a topological property, the problem can be solved by finding an equivalent metric for which   $D[-\tau,0]$ becomes complete as well.  
  \begin{theorem}\label{thm:Polish}
  	The metrics $d_S$ and $d_S^\circ$\index{$d_S^\circ$} on $D[a,b]$, with
  	\begin{equation}\label{eq:sko2}
  		d_S^\circ(\varphi,\psi):=\inf_{\lambda \in \Lambda[a, b]}\big(\|\varphi \circ \lambda-\psi\|_{\infty}+\|\lambda\|^\circ_\infty\big),\quad \varphi,\psi\in D[a,b],
  	\end{equation}
  where
  \begin{equation}
  	\|\lambda\|^{\circ}_\infty=\sup _{s<t}\left|\log \frac{\lambda (t)-\lambda (s)}{t-s}\right|,\quad \lambda\in\Lambda[a,b],
  \end{equation}
  	are equivalent. The space $D[a,b]$ is separable whenever the topology is induced by either $d_S$ or $d_S^\circ$. Moreover,  the  metric space $(D[a,b],d_S^\circ)$ is complete.  	 
  \end{theorem}
  \begin{proof} We refer to either	  \cite[Thm. 12.1]{book:billingsley} and \cite[Thm. 12.2]{book:billingsley} or \cite[Ch. VI.1c]{book:jacod}. The `$+$'  in the definitions of the metrics  \eqref{eq:sko1} and \eqref{eq:sko2} is conform \cite{book:jacod}. In \cite{book:billingsley},  the `$+$' should be replaced by a maximum, but we like to point out that essentially  there is no  difference. 
  	\end{proof}
  Often, one introduces the metric $d_S^\circ$ as the Skorokhod metric and excludes the intermediate step of defining $d_S$; see \cite{article:stojkovic} for instance.  From now one, the results will be stated for   the space of either continuous or \cadlag functions defined on $[-\tau,0]$ instead of an arbitrary compact interval.
  
  \begin{remark}\label{remark:Sko}
In the remainder of this appendix, the space $C[-\tau,0]$ will always be endowed with the uniform  topology and $D[-\tau,0]$ with the Skorokhod topology, unless specified otherwise. Viewed as a metric space, we  consider $D[-\tau,0]$ with $d_S^\circ$ for the completeness property.
  \end{remark}
  
 There is a  strict inclusion between the Borel $\sigma$-algebra   of $D[-\tau,0]$ with its Skorokhod topology and the Borel $\sigma$-algebra of the space of \cadlag functions equipped with the uniform topology, i.e.,
\begin{equation}\label{eq:inclusion}\mathcal B((D[-\tau,0],d_S))\subsetneq \mathcal B((D[-\tau,0],\|\cdot\|_\infty)).\end{equation}
In particular, an important feature of the Skorokhod topology is that  generated Borel sets satisfy a  desired measurability property. That is, the Borel $\sigma$-algebra coincides with the cylindrical $\sigma$-algebra.
  
  \begin{theorem}\label{thm:Borel}
  	The Borel $\sigma$-algebra of $C[-\tau,0]$, i.e., $\mathcal B(C[-\tau,0])$, coincides with the smallest $\sigma$-algebra of subsets of $C[-\tau,0]$ such that the maps $\pi^{t}: x \rightarrow x(t)$ are measurable for all $t \in[-\tau,0] . $ Likewise,  the Borel $\sigma$-algebra of $D[-\tau,0]$, i.e., $\mathcal B(D[-\tau,0])$, coincides with the smallest $\sigma$-algebra of subsets of $D[-\tau,0]$ such that the maps $\pi^{t}: x \rightarrow x(t)$ are measurable for all $t \in[-\tau,0] . $ 
  \end{theorem}
\begin{proof}
We refer to \cite[Thm. 2.1]{book:partha} and \cite[Thm. 7.1]{book:partha}. 
\end{proof}

This means  the Skorokhod topology is really the right topology if one wants to interpret segments of stochastic processes as random variables; see the corollary below and \cite[p. 135]{book:billingsley} for additional information.

\begin{corollary}
    [Measurability of  segment processes]\label{remark:segment} Only under the Skorokhod topology, we have that $X$ is a $D[-\tau,0]$-valued random variable if and only if, for any $t\in [-\tau,0]$, $X(\omega,t)=\pi^t(X(\omega))$ defines a random variable on $\mathbb R$. In particular, for any $Y\in \mathbb D[-\tau,\infty),$ the segment process $(Y_t)_{t\geq 0}$ is an $\mathbb F$-adapted $D[-\tau,0]$-valued stochastic process.
\end{corollary}
	 Indeed, this statement is no longer  valid when $D[-\tau,0]$ is equipped with the {uniform} norm, due to the strict inclusion in \eqref{eq:inclusion}. Regardless of the topology on $D[-\tau,0]$,   we have for   $Y\in\mathbb D[-\tau,\infty)$ that, for each $t\geq 0$,  $\|Y_t\|_\infty=\sup_{s\in[-\tau,0]}|Y_t(s)|$ is a real-valued random variable.  
  
  
\subsection{Arzel\`a--Ascoli theorems:  compactness and modulus of continuity}\label{sec:AA}
      In this section, our goal is to characterise  compact subsets of  $C[-\tau,0]$ and $D[-\tau,0].$ But first, we will state a general Arzelà--Ascoli result; see Theorem \ref{thm:ArzAsc-gen}. Throughout this section, we will assume that $X$ is  a compact Hausdorff space and we denote by  $C(X)$  the space of real-valued continuous functions on $X$,  endowed with the topology induced by the supremum norm. 

\begin{definition}
   A subset $A\subset C(X)$ is said to be \texttt{equicontinuous} if for every  $\varepsilon>0, $ and every $x\in X$, there is a neighbourhood $U_{x}$ about $x$
  	such that   for all $ y\in U_x$ we have $|f(y)-f(x)|<\varepsilon$,  for every $f\in A$. Further, a subset $A\subset C(X)$ is called \texttt{pointwise bounded} if for every $x \in X$ we have
  	$\sup_{f\in A} |f(x)|<\infty$.
  \end{definition}
We see that a family of functions $A$ is equicontinuous when all   functions are continuous, hence uniformly continuous, and when  all   functions    variate controllably over an appropriately  given neighbourhood about any point in space.  
 We like to point out that the Arzelà--Ascoli
  theorem completely characterises relative compactness in $C(X)$; a set $A$ is called \texttt{relatively compact} if the closure of $A$ is compact. The following result  is analogous, in some sense, to the Heine–Borel theorem.
  
  \begin{theorem}[Arzelà--Ascoli]\label{thm:ArzAsc-gen}
  	Suppose $X$ is a compact Hausdorff space. Then any $A\subset C(X)$  is relatively compact if and only if it is equicontinuous and pointwise bounded.
  \end{theorem}
  \begin{proof}
  	A proof can be found in \cite[Thm. 7]{book:dunford}.
  \end{proof}

One can characterise compactness even more concretely by introducing a {modulus of continuity}. We  restrict ourselves to $X=[-\tau,0]$ together with the usual Euclidean topology.

  \begin{definition}
  	  For any real-valued function $\varphi$ with domain $[-\tau,0]$, and $T\subset [-\tau,0]$, we introduce
  \begin{equation}
  	\omega(\varphi,T):=\sup_{s,t\in T}|\varphi(s)-\varphi(t)|.
  \end{equation} The \texttt{modulus\,\,of\,\,continuity} for any $\varphi:[-\tau,0]\to\R$ is given by
  	\begin{equation}
  		\omega(\varphi,\delta):=\sup_{-\tau\leq t\leq-\delta}\omega(\varphi,[t,t+\delta])=\sup_{|s-t|\leq \delta}|\varphi(s)-\varphi(t)|,\quad 0<\delta\leq \tau.
  	\end{equation}
  \end{definition}
Clearly, a necessary and sufficient condition for $\varphi$ to be (uniformly) continuous over $[-\tau,0]$, i.e., $\varphi\in C[-\tau,0]$, is
  \begin{equation}
  \lim _{\delta \rightarrow 0} \omega(\varphi,\delta)=0.\label{eq:continuity}
  \end{equation}

    \begin{theorem}[{Arzelà--Ascoli for $C[-\tau,0]$}]\label{thm:ArzAsc}
  	   A necessary and sufficient condition for a subset $A\subset C[-\tau,0]$ to be relatively compact in the uniform topology is to require
  \begin{equation}
  	\sup _{\varphi \in A}|\varphi(s)|<\infty,\label{eq:cond1}\quad\text{for some $s\in[-\tau,0]$},
  	\end{equation}
  	and
  	\begin{equation}
  	\lim _{\delta \rightarrow 0} \sup _{\varphi \in A} \omega(\varphi,\delta)=0.\label{eq:cond2}
  	\end{equation}
  	
  \end{theorem}
  \begin{proof}Conditions \eqref{eq:cond1} and \eqref{eq:cond2} combined is equivalent to being pointwise bounded and equicontinuous (see further discussion), hence the statement is an immediate consequence of the general Arzelà-Ascoli theorem in Theorem \ref{thm:ArzAsc-gen}.
  	
  	 On the other hand, a much more direct proof for this statement can be found, for instance, in \cite[Thm. 7.2]{book:billingsley}. Note  there is a small difference in \eqref{eq:cond1}  present; we only assume $	\sup _{\varphi \in A}|\varphi(s)|<\infty,$ for some $s\in[-\tau,0]$, instead of $s$ necessarily being the starting point $s=-\tau$. We can do this because condition \eqref{eq:cond2}, which is   (uniform) equicontinuity over $[-\tau,0]$, together with  \eqref{eq:cond1} yields  
  \begin{equation}
  	\sup_{t\in[-\tau,0]}\sup_{\varphi \in A}|\varphi(t)| <\infty.\label{eq:uni1}
  \end{equation}
  This uniform bound implies pointwise boundedness and allows us to let $s$ in \eqref{eq:cond1} be arbitrary.
  \end{proof}
  
The  idea would now be  to   introduce a modulus ``of continuity'' for the space of \cadlag functions. Differently put, we would like to have a mapping $\varpi$ which ensures us that  $\varphi\in D[-\tau,0]$ holds true whenever a similar condition as in  \eqref{eq:continuity} is satisfied.
  \begin{proposition}
  	A function $\varphi:[-\tau,0]\to \R$ is in $D[-\tau,0]$ if and only if
  	\begin{equation}
  		 \lim _{\delta \rightarrow 0} \varpi(\varphi,\delta)=0,
  	\end{equation}
  where
  \begin{equation}
  	\varpi(\varphi,\delta):=\inf\left\{\max_{1\leq i\leq k}\omega(\varphi,[t_{i-1},t_i)):k\in\N,\, \{t_i\}_{i=0}^k\in \Xi\right\},\quad 0<\delta\leq \tau,
  \end{equation}
with $\Xi$   the set of all finite sequences $\{t_i\}_{i=0}^k$, $k\in \N$, with $-\tau=t_0<t_1<...<t_{k-1}<t_k=0$ and $\min_{1 \leq i \leq k}|t_i-t_{i-1}|>\delta$.
  \end{proposition}
\begin{proof}
 This is an immediate consequence of \cite[Lem. 1]{book:billingsley} and the discussion thereafter.
\end{proof}
   Notice that $\varpi(\varphi,\delta)$ is unaffected if the value of $\varphi(0)$ changes, which is a relevant property. Let us now compare the moduli $\omega$ and $\varpi$ for functions in $D[-\tau,0]$.  Observe that the interval $[-\tau,0)$ can be divided into subintervals $\left[t_{i-1}, t_{i}\right)$ satisfying $\delta<t_{i}-t_{i-1} \leq 2 \delta<\tau,$ hence
  \begin{equation}\varpi (\varphi,\delta) \leq \omega(\varphi,2 \delta), \end{equation} 
  in case $\delta<\tau / 2$ holds. Obviously, 
  there cannot be such an inequality in the other direction, because the condition in \eqref{eq:continuity} does not hold for discontinuous functions $\varphi$. Nonetheless, we do have
  \begin{equation}
  	\omega(\varphi,\delta)\leq 2\varpi(\varphi,\delta)+\Delta_{\text{sup}}(\varphi),
  \end{equation}
for all $0<\delta\leq \tau$ and $\varphi\in D[-\tau,0]$, where
\begin{equation}
	\Delta_{\text{sup}}(\varphi)=\sup_{t\in[-\tau,0]}|\varphi(t)-\varphi(t-)|,\quad \varphi\in D[-\tau,0].
\end{equation}
 Any \cadlag function defined on a compact space only allows  finitely many jumps to exceed a given positive number \cite[p. 122]{book:billingsley}. Therefore, the overall maximum absolute jump $	\Delta_{\text{sup}}(\varphi)$ of any \cadlag function $\varphi$ is finite and attained.
    In conclusion, for $\delta>0$ sufficiently small, we find
  \begin{equation}\label{eq:modcad}
  	\varpi (\varphi,\delta/2) \leq \omega(\varphi,\delta)\leq 2\varpi(\varphi,\delta)+\Delta_{\text{sup}}(\varphi),\quad\varphi \in D[-\tau,0].
  \end{equation}
  It is interesting to explicitly mention the fact 
  \begin{equation}\label{eq:modcont}
  	\varpi (\varphi,\delta/2) \leq \omega(\varphi,\delta)\leq 2\varpi(\varphi,\delta),\quad \varphi\in C[-\tau,0],
  \end{equation}
  and thus the moduli $\omega$ and $\varpi$ are essentially equivalent for continuous functions. This observation is extremely  useful in Corollary \ref{cor:Ctight}, for example.

It turns out that  the modulus $\varpi$ ``of continuity'' allows us to state an analogue of Theorem \ref{thm:ArzAsc} and characterise compact sets in the Skorokhod space $D[-\tau,0]$.

 \begin{theorem}[{Arzelà--Ascoli for $D[-\tau,0]$}]\label{thm:ArzAsc-sko}
	  A necessary and sufficient condition for a set $A\subset D[-\tau,0]$ to be relatively compact in the Skorokhod topology is to require
	\begin{equation}
	\sup _{\varphi \in A}\|\varphi\|_\infty<\infty,\label{eq:uni2}
	\end{equation}
	and
	\begin{equation}
\label{eq:cond3}	\lim _{\delta \rightarrow 0} \sup _{\varphi \in A} \varpi(\varphi,\delta)=0.
	\end{equation}
\end{theorem}
\begin{proof}
	We refer to the proof of \cite[Thm. 12.3]{book:billingsley}. In the setting of \cadlag functions with their  domain being the half line $\Rplus$, see the slightly more involved proof in  \cite[Thm. 1.14]{book:jacod}.
\end{proof}
  Obviously, the    theorem   above cannot be a corollary of  Theorem \ref{thm:ArzAsc-gen},  because we are simply no longer in the continuous setting. Due to its high resemblance  with Theorem \ref{thm:ArzAsc}, we will still call it an  Arzelà--Ascoli theorem.  
  It is interesting to note that  \eqref{eq:uni1} and \eqref{eq:uni2} coincide, since we may interchange the order of the suprema, i.e.,
  \begin{equation}
  \sup_{t\in[-\tau,0]}\sup_{\varphi\in A}|\varphi(t)|\quad\text{and}\quad   \sup_{\varphi\in A}\sup_{t\in[-\tau,0]}|\varphi(t)|
  \end{equation} are  equal. The main  difference with the Skorokhod space  is that now no longer a single $s$ satisfying $	\sup _{\varphi \in A}|\varphi(s)|<\infty$ together with condition \eqref{eq:cond3} implies \eqref{eq:uni2}. In order to see this, we construct a  counterexample: for $A=\{n\mathbf{1}_{[-\tau/2,0]}:n\in\N\}$, we   have $\varpi(\varphi,\delta)=0$ for all $\varphi\in A$ if $\delta<\tau/2$ holds, but at the same time we have $\sup_{\varphi\in A}\|\varphi\|_\infty=\infty$.

Ultimately, the important part of  Theorem \ref{thm:ArzAsc-sko} is in fact the sufficiency, which will be used to prove tightness, as we will particularly see in Proposition \ref{prop:Dtight}.

\subsection{Tightness}
\label{Sec2.2.3}
 \noindent  Suppose  $X$ is  a Hausdorff space. Denote by $\mathscr{M}(X)$  the  set of   Borel signed measures  $\mu$ on $X$. We write $\mathscr{M}^+(X)\subset \mathscr M(X) $ for the subset of all Borel measures.

   \begin{definition}
   A collection $M\subset \mathscr M^+(X)$   is said to be \texttt{tight} if for every $\varepsilon>0$ there exists a compact  set $K_{\varepsilon} \subset X$ with
 	\begin{equation} 
 		\mu(X\backslash K_{\varepsilon}) <\varepsilon, \quad \text { for all } \mu \in M.
 	\end{equation}
 \end{definition}
 
Limiting ourselves to probability distributions, or in other words  random variables, results into the following definition of tightness.
 
 \begin{definition}
  A family of $X$-valued random variables $(Z_{\eta})_{\eta\in I}$ is called \texttt{tight} in $X$
  if for every $\varepsilon>0$ there exists a compact set $K_{\varepsilon} \subset X$ such that
	\begin{equation} 
	\mathbb{P}\left(Z_{\eta} \notin K_{\varepsilon}\right) <\varepsilon, \quad \text { for all } \eta \in I.
\end{equation}
  \end{definition}

For $X$ a finite dimensional normed space, we have that $(Z_{\eta})_{\eta \in I}$ is tight if and only if it is bounded in probability. On the other hand, if the normed space is infinite dimensional, e.g., $ C[-\tau,0] $, then tightness is no longer equivalent with boundedness in probability. This is  due to the fact    closed balls in a normed space $X$ are compact  if and only if the dimension of $X$ is finite. 
Moreover, if $X=(X,d)$ is a {complete} metric space and {separable}, then this results into any $X$-valued random variable, or any finite sequence of $X$-valued random variables, to be tight. This is because separability together with completeness  implies that  Borel measures are  also Radon \cite{book:bogachev}. 

Now assume $X=(X,d)$ is a separable metric space and write  $\mathscr P(X)$  for the space of all Borel probability measures $\mu$ on $X$. Thus, we have $\mu(X)=1$ for all $\mu\in\mathscr{P}(X)$. We endow $\mathscr P(X)$ with  the \text{weak topology}. This topology is, for example, induced by the metric
  \begin{equation}
 \index{$d_0$}	d_{0}(\mu,\nu)=\sup_{f\in \rm Lip_{\infty}}\left|\int_X f\,\mathrm d\mu-\int_X f\,\mathrm d\nu\right|,\quad \mu,\nu\in\mathscr{P}(X),\label{eq:KRmetric}
  \end{equation} 
where $\rm Lip_{\infty}$ is the space of  functions $f:X\to\R$
 with Lipschitz constant at most 1 and $\|f\|_\infty\leq 1.$
Separability ensures that $\mathscr P(X)$ is metrisable, where $d_0$ is an appropriate metric \cite[p. 193]{book:bogachev}, and $\mathscr P(X)$  itself becomes separable within the weak topology  \cite[p. 213]{book:bogachev}. Furthermore, if the  metric space  $X$  is in addition assumed to be complete,   then $(\mathscr P(X),d_0)$ is complete too \cite[p. 232]{book:bogachev}.

   \begin{theorem}[{Prokhorov}]\label{thm:prokhorov}   	Let $(X,d)$ be a complete separable metric space and suppose $\Gamma$ is a subset of $\mathscr{P}(X)$. Then the following two statements are equivalent:
   \begin{enumerate}[\normalfont(i)]
   	\item  $ \Gamma $ is relatively compact in the weak topology;
   \item $\Gamma$ is tight.
   \end{enumerate}
 \end{theorem}
\begin{proof}
	A proof  can be found in many textbooks; see, e.g.,   \cite[Ch. 5]{book:billingsley} and \cite[Thm. 8.6.2]{book:bogachev}. 
\end{proof}


When sufficient regularity of a continuous-time family $(Z_t)_{t\geq 0}$ is presumed, we find with the help of Prokhorov's theorem that any finite time horizon process $(Z_t)_{t\in[0,T]}$, $T\geq 0$, is tight.  Sufficient regularity turns out  to be  continuity in probability in   metric spaces. 
\begin{definition}
	Suppose $(X,d)$ is a metric space. An $(X,d)$-valued process $Z=(Z_t)_{t\geq 0}$ is said to be \texttt{continuous\,\,in\,\,probability}, or \texttt{stochastically\,\,continuous}, if for any time $t_0\geq 0$ fixed, we have
	\begin{equation}
		\mathbb P(d(Z_t,Z_{t_0})>\varepsilon)\to 0,\quad \text{as }t\to t_0,
	\end{equation}
for all $\varepsilon>0.$
\end{definition}

The next result  can be interpreted as a generalisation of  tightness of finite sequences of measures.
 \begin{proposition}\label{prop:wow}
 	Let $(X,d)$ be a complete separable metric space and consider $(Z_t)_{t\geq 0}$ to be a\linebreak continuous-time family of $X$-valued random variables. If  $(Z_t)_{t\geq 0}$ is  stochastically  continuous in   $X$, then any finite time horizon   process $(Z_t)_{t\in[0,T]}$, $T\geq 0$,  is tight in $X$.
 \end{proposition}
\begin{proof}
Let $d_E$ denote the Euclidean metric on the non-negative reals $\Rplus$, and  introduce the map
\begin{equation}
	P:(\Rplus,d_E)\to (\mathscr P(X),d_0),\,t\mapsto \mathcal L(Z_t),
\end{equation}
where $\mathcal L(Z_t)$ denotes the law of $Z_t$. We claim that $P$ is continuous, hence the image
\begin{equation}
	P([0,T])=\{\mathcal L(Z_t):t\in[0,T]\},
\end{equation}
for any $T\geq 0$, is compact in the weak topology.   Theorem \ref{thm:prokhorov} concludes that $\{\mathcal L(Z_t):t\in[0,T]\}$ is tight. In other words, the finite time horizon process $(Z_t)_{t\in[0,T]}$ is tight in $X$.

We now prove that $P$ is indeed continuous. 
Let $(t_n)_{n\in\N}$ be any sequence such that $t_n\to t_0$, as $n\to\infty$, with $t_0\geq 0$ arbitrary yet fixed. Define
\begin{equation}
	\Omega_{n,\varepsilon}:=\big\{\omega\in\Omega:d\big(Z_{t_n}(\omega),Z_{t_0}(\omega)\big)>\varepsilon\big\}, \quad n\in\N,
\end{equation}
for any $\varepsilon>0$. Since $(Z_t)_{t\geq 0}$ is stochastically continuous in $(X,d)$, we have for any $\varepsilon'>0$ that there exists some natural number $n_0\in\N$ such that $\mathbb P(\Omega_{n,\varepsilon})<\varepsilon'$ holds for all $n\geq n_0.$
Let $f\in\rm Lip_{\infty}$ be arbitrarily given---recall the notation in \eqref{eq:KRmetric}---and observe
\begin{align}
	\left|\int_Xf\,\mathrm d\mathcal L(Z_t)-\int_Xf\,\mathrm d\mathcal L(Z_{t_0})\right|\nonumber&=	\left|\int_\Omega f(Z_t)\,\mathrm d\mathbb P-\int_\Omega f(Z_{t_0})\,\mathrm d\mathbb P\right|\\
	\nonumber&\leq \int_{\Omega}\big|f(Z_t)-f(Z_{t_0})\big|\,\mathrm d\mathbb P\\
	\nonumber&\leq 	\int_{\Omega\backslash \Omega_{n,\varepsilon}}\big|f(Z_t)-f(Z_{t_0})\big|\,\mathrm d\mathbb P+\int_{  \Omega_{n,\varepsilon}}\Big(\big|f(Z_t)\big|+\big|f(Z_{t_0})\big|\Big)\,\mathrm d\mathbb P  \\
		\nonumber&\leq \text{Lip}(f)	\int_{\Omega\backslash \Omega_{n,\varepsilon}}d(Z_t,Z_{t_0})\,\mathrm d\mathbb P+2\|f\|_\infty\mathbb P(\Omega_{n,\varepsilon})  \\
		\nonumber&\leq \text{Lip}(f)\,\varepsilon\,	 \mathbb P(\Omega\backslash \Omega_{n,\varepsilon})+2\|f\|_\infty\mathbb P(\Omega_{n,\varepsilon})  \\
			&\leq  \varepsilon +2  \varepsilon',
\end{align}
for any $n\geq n_0$, where 
$\text{Lip}(u)=\sup _{x,y\in X,\,x\neq y} {|u(x)-u(y)|}/{d(x,y)}$ is the Lipschitz constant of a function $u:X\to\R$. Since we may take $\varepsilon\to 0$ and $\varepsilon'\to 0$, we obtain
\begin{equation}
	d_0\big(P(t_n),P(t_0)\big)=d_0\big(\mathcal L(Z_{t_n}),\mathcal L(Z_{t_0})\big)\to 0,\quad\text{as } n\to \infty.
\end{equation}
Because this holds for any sequence $(t_n)_{n\in\N}$ with $t_n\to t_0$, for every instant $t_0\geq 0$, we deduce the (sequential) continuity of the map $P$.
\end{proof}

\begin{remark} If   $Y\in\mathbb D[0,\infty)$  is  stochastically continuous, e.g., a Lévy process, then the segment process $(Y_t)_{t\geq \tau}$ is a family of $D[-\tau,0]$-valued random variables which is stochastically continuous  \cite[Lem. 2.3]{article:reiss}.
\end{remark}
 
Remember  $D[-\tau,0]$ is endowed  with the Skorokhod topology, recall Remark \ref{remark:Sko}, and that this  is a separable completely metrisable space; see the previous section.  Due to the fact Prokhorov's theorem expresses tightness in terms of compactness, the  Arzelà--Ascoli theorem is often encountered in combination with Prokhorov's theorem. In the continuous setting,  tightness can be characterised in terms of the modulus of continuity; see, e.g., \cite{book:billingsley, book:revuz}. Similar results hold for  the right-continuous setting; see, e.g., \cite{book:billingsley,book:jacod}, but observe that in these references they restrict to sequences. 

 \begin{proposition}\label{prop:Dtight}
 	  A continuous-time family of  $  D[-\tau,0]$-valued   random variables $(Z_{t})_{t\geq 0}$  is tight if and only if for every $\varepsilon>0$ we have
 	  	\begin{enumerate}[\normalfont(i)]
 	  	\item $\lim _{R \rightarrow \infty} \sup _{t \geq 0} \mathbb{P}\left(\left\|Z_{t}\right\|_{\infty} \geq R\right)=0;$
 	  	\item $\lim _{\delta \rightarrow 0} \sup _{t \geq 0} \mathbb{P} (\varpi \left(Z_{t}, \delta )\geq \varepsilon\right)=0$.
 	  \end{enumerate}
 \end{proposition}
 \begin{proof}
In essence, the result directly follows from    Arzelà--Ascoli   in $D[-\tau,0]$; see Theorem \ref{thm:ArzAsc-sko}. 	The proof is entirely conform to the proof  of   \cite[Thm. 13.2]{book:billingsley} or \cite[Prop. VI.3.26]{book:jacod}. For  continuous-time families we are required  to replace  the $ \limsup _{n \rightarrow \infty}$-part in \cite{book:billingsley,book:jacod} with the slightly stronger $\sup _{t \geq 0}$-part, since we can no longer use the fact that we have tightness for a finite horizon of random variables.
\end{proof}

The following   is a direct consequence of Proposition \ref{prop:wow} and Proposition \ref{prop:Dtight}. 
It is worthwhile to point out  that, under sufficient regularity of the family, we are only interested in what happens at infinity like in the case of sequences.
 \begin{corollary} \label{cor:stoch}
	A stochastically continuous process $(Z_{t})_{t\geq 0}$ in $D[-\tau,0]$  is tight if and only if for every $\varepsilon>0$ we have
	\begin{enumerate}[\normalfont(i)]
		\item $\lim _{R \rightarrow \infty} \limsup _{t \to \infty} \mathbb{P}\left(\left\|Z_{t}\right\|_{\infty} \geq R\right)=0;$
		\item $\lim _{\delta \rightarrow 0} \limsup _{t \to\infty} \mathbb{P} (\varpi \left(Z_{t}, \delta )\geq \varepsilon\right)=0$.
	\end{enumerate}
\end{corollary}

Conform to \cite{book:jacod}, we will introduce the concept of $C$-tightness for continuous-time families. 

 \begin{definition}\label{def:Ctight}
	    A continuous-time family of  $  D[-\tau,0]$-valued   random variables $(Z_{t})_{t\geq 0}$  is said to be \texttt{$C$-tight} if  for every $\varepsilon>0$ we have
	\begin{enumerate}[\normalfont(i)]
	\item $\lim _{R \rightarrow \infty} \sup _{t \geq 0} \mathbb{P}\left( \left\|Z_{t} \right\|_{\infty} \geq R\right)=0$;
	\item $\lim _{\delta \rightarrow 0} \sup _{t \geq 0} \mathbb{P}\left(\omega (Z_{t}, \delta )\geq \varepsilon\right)=0$.
\end{enumerate}
\end{definition}
Every $C$-tight family is obviously tight and follows from   \eqref{eq:modcad}. 
 Introducing $C$-tightness is done in a more abstract way in \cite{book:jacod} by defining it as a tight sequence or family whose  laws converge  to the laws of continuous processes (where the convergence is in the sense of subsequences, and thus ``limit points'' is indeed plural). 
 The equivalence of these two approaches is due to    \cite[Prop. VI.3.26]{book:jacod}.
As a result of the inequalities in equation \eqref{eq:modcont}, we obtain the following handy corollary for  families of continuous processes.
\begin{corollary}
	A tight family $(Z_t)_{t\geq 0}$ of $C[-\tau,0]$-valued random variables is   $C$-tight.\label{cor:Ctight}
\end{corollary}

We end this section by stating (and proving) two useful lemmas.   All results in \cite{book:jacod} are actually for sequences, i.e., with index set $I=\mathbb N,$ only.  Nonetheless, those results   that we implement in our work easily extend to families with index set $I=\Rplus.$  In particular, we see that the proof of the result below is no different and simply requires  Proposition \ref{prop:Dtight} and our  definition of $C$-tightness.
 \begin{lemma}[Corollary 3.33 of \cite{book:jacod}]\label{lem:sum}
  Let $(Y_t)_{t\geq 0}$ be a $C$-tight family. In addition, assume
 $ (Z_t )_{t\geq 0}$ is a tight $($resp., $C$-tight\,$)$ family of $D[-\tau,0]$-valued random variables. Then  the sum $\left(Y_t+Z_t\right)_{t\geq 0}$ is  also tight $($resp., $C$-tight\,$)$.
 \end{lemma}

We say $Z=(Z(s))_{s\in[-\tau,0]}$ is a \texttt{non-decreasing\,\,process}   when $Z$ is  \cadlag in the pathwise sense, non-negative, non-decreasing, and $Z(-\tau)=0.$ A non-decreasing process  $Y$ is said to \texttt{strongly}\linebreak \texttt{majorise} another non-decreasing process $Z$ whenever $Y-Z$ is also a non-decreasing process.
 \begin{lemma}[Proposition 3.35 of \cite{book:jacod}]\label{lem:increasing}
 	 Suppose, for any $t\geq 0$ fixed, that $Y_{t}$ is a non-decreasing process that strongly majorises the non-decreasing process $Z_{t}.$ If the family $ (Y_{t})_{t\geq 0}$ is tight, then so is the family $(Z_t)_{t\geq 0}$. The same result holds when exchanging tight for $C$-tight.
 \end{lemma}
Indeed, this statement   easily extends to continuous-time families as well, since the result  immediately follows from Proposition \ref{prop:Dtight} and Definition \ref{def:Ctight}, once one notices that $\left|Z_{t}(s)\right| \leq \left|Y_{t}(s)\right|$, $s\in[-\tau,0]$, $\varpi \left(Z_t, \delta\right) \leq \varpi\left(Y_t, \delta\right)$, and $\omega\left(Z_t, \delta\right) \leq \omega\left(Y_t, \delta\right)$ holds.


\subsection{Direct applications}\label{sec:applications}
In this section, we provide two intermediate results that are useful for showing tightness of solution segments; see {\S}\ref{sec:towards}. The following tightness result is applied in {\S}\ref{sec:tight_cont}, where we consider the stochastic integrator of the SDDE to be continuous.
\begin{theorem}[Kolmogorov's tightness criterion]\label{thm:kolm}
	The  family  $(Z_t)_{t\geq 0}$   of $C[-\tau,0]$-valued random variables is tight in the uniform topology if
	\begin{enumerate}[\normalfont(i)]
  \item the  family $(Z_t(-\tau))_{t\geq 0}$ of starting values is tight in $\mathbb R$, i.e., bounded in probability\textnormal ;\item there exists constants $\gamma>0, \delta>1, $ and $K>0$ such that
\begin{equation}
	\mathbb{E}\big|Z_{t}(t_{2})-Z_{t}(t_{1})\big|^{\gamma} \leq K\left|t_{2}-t_{1}\right|^{\delta},\quad  \text {for all } t \geq t_0 \text{ and every } t_{1}, t_{2} \in[-\tau,0],\label{eq:kolm}
\end{equation}
holds, for some $t_0\geq 0.$
	\end{enumerate}
\end{theorem}

\begin{proof} The proof is similar to the one in \cite{book:revuz}. Our goal is to show that the family is $C$-tight. Recall that \eqref{eq:cond1} implies  \eqref{eq:uni2}, in the continuous setting, hence part (i) of $C$-tightness in Definition \ref{def:Ctight} is because of tightness in $\mathbb R$ of the starting values. 
Part (ii) of the $C$-tightness follows from applying Markov's inequality together with \eqref{eq:kolm}. Observe that it suffices to have \eqref{eq:kolm} for $t\geq t_0$ for some $t_0\geq 0, $ due to the result in Proposition \ref{prop:wow}; the continuous-time family $(Z_t)_{t\geq 0}$ is  stochastically continuous in $C[-\tau,0]$, since $(Z(t))_{t\geq 0}$ is a (stochastically) continuous process \cite[Lem. 2.3]{article:reiss}.
\end{proof}
 
 The second tightness result  plays a crucial role in  {\S}\ref{sec:special}, which investigates when solution segments in  the \cadlag case are tight. Note that the below is more or less a direct consequence of the Arzelà--Ascoli theorem, yet we  will still provide a detailed proof.

 \begin{lemma}\label{lem:reduction}
 	Let $ (f(s))_{-\tau\leq s<\infty}$ be a stochastic process in $\mathbb D[-\tau,\infty)$. Define, for every   $t\geq 0$,  the continuous-time family $(F_t)_{t\geq 0}$ of    $  C[-\tau,0]$-valued random variables by
 	\begin{equation}
 		F_t(u):=\int_{t-\tau}^{t+u}f(s)\,\mathrm ds,\quad u\in[-\tau,0].
 	\end{equation}
Let us  write $f_t:=(f(t+s))_{s\in [-\tau,0]}$, for each $t\geq 0$. If $\left(\|f_t\|_\infty\right)_{t\geq 0}$  is bounded in probability,   then the   family    $(F_t)_{t\geq 0}$ is  $C$-tight. 
 \end{lemma}
\begin{proof}
Define for all $M\geq 0$ the sets
	\begin{equation}
		C_M:=\{F\in C[-\tau,0]: \|F\|_{\infty}\leq \tau M,\quad \text{Lip}(F)\leq M\},
	\end{equation}
where
$\text{Lip}(g)=\sup _{x,y\in[-\tau,0],\,x\neq y} {|g(x)-g(y)|}/{|x-y|}$ denotes the Lipschitz constant of a function $g:[-\tau,0] \to\mathbb R$. Any set $C_M\subset C[-\tau,0]$ is pointwise bounded and equicontinuous, which directly follows from the uniform bounds  $\|F\|_{\infty}\leq \tau M$ and $ \text{Lip}(F)\leq M$, respectively. Theorem \ref{thm:ArzAsc-gen}  allows us to deduce that $C_M$ is relatively compact in $ C[-\tau,0]. $

    Now, let $\varepsilon>0$ be arbitrarily given. By the boundedness in probability assumption, there exists a constant $M_\varepsilon\geq 0$ such that, for all $t\geq 0$ fixed, there is some  $\Omega_t\subset\Omega$ satisfying $\mathbb P(\Omega_t)\geq 1-\varepsilon$ and $\sup_{s\in[t-\tau,t]}|f(s,\omega)|\leq M_\varepsilon$ for every $\omega \in \Omega_t$.  Hence, for any $ -\tau\leq v\leq u\leq 0$, this results into the pointwise estimates
\begin{equation}|F_t(u)-F_t(v)|=\left|\int_{t+v}^{t+u} f(s)\,\mathrm d s\right| \leq \int_{t+v}^{t+u}|f(s)|\, \mathrm d s \leq  \int_{t+v}^{t+u}M_\varepsilon\, \mathrm d s = M_\varepsilon(u-v),\quad  \text{on }\,  \Omega_t,\end{equation}and
\begin{equation}
\|F_t\|_\infty\leq \sup_{u\in[-\tau,0]}\int_{t-\tau}^{t+u}|f(s)|\,\mathrm ds\leq \tau M_\varepsilon,\quad\text{on }\,\Omega_t.
\end{equation}
Consequently, for any time $t\geq 0$, we have $[u\mapsto F_t(u,\omega)]\in C_{M_\varepsilon}$  for all $\omega\in\Omega_t.$  

Recall that $C_{M_\varepsilon}$ is relatively compact and let us  denote its closure by $K_\varepsilon$. Then, for any $t\geq 0$, we have $
	\mathbb P(F_t\in K_\varepsilon)=\mathbb P(\Omega_t)\geq 1-\varepsilon. $
We obtain that the family $(F_t)_{t\geq 0}$ is  tight in $ C[-\tau,0]$ and the assertion  follows from Corollary \ref{cor:Ctight}.
\end{proof}




\subsection{Semimartingale characteristics}\label{app:C}

        \noindent In this part of the appendix we introduce the notion of {semimartingale characteristics}, which is to be understood as the generalisation of  Lévy characteristic triplets. These characteristics turn out to be a  useful tool for proving tightness of segments when the integrators are in (HInt) $\subset$ (HSpec); see {\S}\ref{sec:special}.
         We   restrict  ourselves to the one-dimensional setting and   refer to \cite{book:jacod} for its $d$-dimensional analogue. 
         A \texttt{truncation\,\,function} is a bounded function  $h:\R\to\R$ which   satisfies $h(x)=x$ for all $x$ in a neighbourhood of 0. Typically, we set $h(x)=x\mathbf 1_{\{|x|\leq 1\}}$.

\begin{definition}\label{def:semi}
	Let $X\in \mathbb D\Rplus$ be a semimartingale. Its \texttt{semimartingale} \texttt{characteristic}, with respect to the truncation function $h$, is denoted by the triplet $(B_X,C_X,\nu_X) $, where $B_X=(B_X(t))_{t\geq 0}$ and $C_X=(C_X(t))_{t\geq 0}$ are the unique predictable processes,    and $\nu_X$ is the unique non-negative random measure
on $\Rplus\times \R$ with $\nu_X( \{0\}\times \R;\omega)=0$ for all $\omega\in\Omega$, such that 
	\begin{equation}
		\left(\frac{\exp(iuX(t))}{\exp \psi_t(u)}\right)_{t\geq 0},
	\end{equation} 
with
	\begin{equation}\label{eq:psi}\psi_{t}(u):=i u B_X(t)-\frac{u^{2}}{2} C_{X}(t)+\int_{[0,t]\times \R}\left(e^{i u x}-1-i u h(x)\right) \nu_X(\mathrm ds,\mathrm d x),\end{equation}
  becomes a complex-valued martingale.
\end{definition} 
The precise statement for existence and uniqueness (up to indistinguishability) can be found in \cite[Sec. II.2]{book:jacod}.
As a matter of fact, one can explicitly construct the characteristics by means of the observations in the proof  of \cite[Prop. II.2.9]{book:jacod}. For our intents and purposes, there is no need to go deeper into the concept except for the following observations.

Observe that   the integrand   in \eqref{eq:psi} is independent of the time $t$, hence we can rewrite
\begin{equation}\label{eq:newnot}
	\int_{[0,t]\times \R}\left(e^{i u x}-1-i u h(x)\right) \nu_X(\mathrm ds,\mathrm d x)=\int_{\R}\left(e^{i u x}-1-i u h(x)\right) \nu_X([0,t]\times\mathrm d x).
\end{equation}
The right hand side of \eqref{eq:newnot} is conform \cite{book:jacod}, yet we rather stick with our initial notation. Furthermore, only the $B_X$ depends on the choice of  the truncation function $h$ \cite{book:jacod}, just like that only the $b$ of a Lévy triplet $(b,\sigma,\nu)$  depends on the choice of truncation. Therefore, it may convene the reader to explicitly write $b_h$ and $B_{X,h}$ for $b$ and $B_X,$ respectively. On the other hand, the choice of $h$ will not be of any significance in our applications. We will drop the $h$ if the choice of the truncation is clear from the context.

\begin{examplex}\label{ex:char}
Suppose $L=(L(t))_{t\geq0}$ is a (one-dimensional) Lévy process with   triplet $(b_h,\sigma,\nu)$.  As an  immediate corollary of the Lévy--Khintchine formula \cite{book:applebaum, book:sato}, we obtain that  the semimartingale characteristic $(B_{L,h},C_L,\nu_L$) is   given by\footnote{The $\mathrm dt$ in the term $\nu(\mathrm dx)\,\mathrm dt$ denotes the Lebesgue  measure on $\Rplus$.}
\begin{equation}
	B_{L,h}(t,\omega)=b_ht,\quad C_L(t,\omega)=\sigma^2 t,\quad\text{and}\quad \nu_L(\mathrm dt,\mathrm dx;\omega)=\nu(\mathrm dx)\, \mathrm dt.
\end{equation}
Both the processes $B_{L,h}$ and $C_L$ are deterministic, and so is the random measure $\nu_L$.
\end{examplex}
Once the semimartingale characteristic of a semimartingale $X$ is known, we    obtain by means of Theorem \ref{thm:char} the semimartingale characteristic of a  process $H\in\mathbb L\Rplus$ integrated with respect to $X$. For this we require $X$ to be a special semimartingale. 
Let $X=X(0)+M+A$ be a special semimartingale. Then   \cite[Cor. II.2.38]{book:jacod} yields
\begin{equation}
	A(t)=B_{X,h}(t)+ \int_{[0,t]\times \R}\big(x-h(x)\big)\nu_{X}(\mathrm ds,\mathrm dx),\quad t\geq 0,\label{gen1}
\end{equation}
and
\begin{equation} M(t)=X^{c}(t)+ \int_{[0,t]\times \R}x\big(\mu_{X}(\mathrm ds,\mathrm dx)-\nu_{X}(\mathrm ds,\mathrm dx)\big),\quad t\geq 0,\label{gen2}
\end{equation}
where   $X^c$ is the continuous local martingale part of $X$, and $\mu_X$ is the random measure associated with the jumps of $X$. Observe that $\nu_X$ is referred to as the \texttt{compensator} of $\mu_X$  \cite{book:jacod}. Equations \eqref{gen1}--\eqref{gen2} generalise the Lévy--Itô decomposition \cite{book:applebaum, book:kyp, book:protter,book:sato} to special semimartingales and  a  further   generalisation   can be found in   \cite[Thm. II.2.34]{book:jacod}.
\begin{examplex}\label{ex:sum}
	Suppose $X$ is a special semimartingale and let $V$ be  a continuous process which is of  finite variation. Then, clearly,  $Y=X+V$ is a special martingale too. From  \eqref{gen1}--\eqref{gen2} we   easily deduce that $B_{Y}=B_X+V$, $C_Y=C_X$ and $\nu_Y=\nu_X$ hold. This is due to the fact   $\mu_Y$ and $\mu_X$ coincide, since $V$ is continuous, and because the local martingale part of $Y$ is also given by $M$.
\end{examplex}

\begin{theorem}[Proposition 7.6 of \cite{phdthesis:Raible}]\label{thm:char}
 Let $H\in \mathbb L\Rplus$   and  $X$ be of class \textnormal{(HSpec)} with characteristics $(B_{X,h}, C_{X},\nu_X)$. Consider the stochastic integral  
	\begin{equation}
	I(t):=\int_0^{t} H(s) \,\mathrm d X(s),\quad t\geq 0.
	\end{equation}
Then $I=(I(t))_{t\geq 0}$ is a special semimartingale with the following characteristics:
	\begin{align}
		B_{I,h}(t)&=\int_0^tH(s)\,\mathrm dB_{X,h}(s)+ \int_{[0,t]\times \R}\big(h(xH(s))-H(s)h(x)\big)\,\nu_X(\mathrm ds,\mathrm dx),\\
		C_{I}(t)&=\int_0^t|H(s)|^2\,\mathrm dC_X(s),
	\end{align}
and
\begin{equation}
\label{eq:not-orig}	\nu_I(S,A;\omega)=\int_{\Rplus\times \R}\mathbf 1_S(s)\mathbf 1_A(H(s,\omega)x)\, \nu_X(\mathrm ds,\mathrm dx;\omega),
\end{equation}
	for all Borel measurable sets $S\in\mathcal B( \Rplus)$ and $A\in\mathcal B(\R)$, and for all $\omega\in\Omega$.
\end{theorem}
 Theorem \ref{thm:char} is a simplification of  \cite[Prop. 7.6]{phdthesis:Raible}, as  we assume $H\in\mathbb L\Rplus$ instead of $H$ being a predictable process, i.e., $\mathcal P$-measurable. Moreover, observe that    \eqref{eq:not-orig} is a direct consequence of the original formulation in \cite{phdthesis:Raible}, since $\nu_I$ is the unique random measure  satisfying
 \begin{equation}
 	\int W( s, \omega,  x)   \nu_I(\mathrm ds,\mathrm dx;\omega)=\int W(s,\omega,  H(s,\omega) x)    \nu_X(\mathrm ds,\mathrm dx;\omega),
 \end{equation}
for all $\mathcal P\times \mathcal B(\R)$-measurable $W=W(t,\omega,x)$. Indeed, consider the specific  choice $W=\mathbf 1_{S}(s)\mathbf 1_{A}(x)$ with $S\in\mathcal B(\Rplus)$ and $A\in\mathcal B(\R)$, where we note that  $\mathcal P$-measurability of the function $\mathbf 1_S$ follows from the fact that all deterministic processes are predictable.

When $L=(L(t))_{t\geq 0}$ is a Lévy process which is also a special semimartingale, we find that $L$ is a process of class (HIntL). Example \ref{ex:char} and Theorem \ref{thm:char}  imply that $\nu_I$, where $I$ is a stochastic integral of $H\in\mathbb L\Rplus$ with respect to $L$, can be written as\footnote{From  \cite[Prop. II.2.9]{book:jacod} and  \cite[Thm. II.2.42]{book:jacod}, we know that any random measure $\nu_X$ can be written in terms of a transition kernel $K_X$ and a predictable finite variation process $V_X$, i.e., $\nu_X( \mathrm d t, \mathrm d x)= K_X(t,d x)\,\mathrm d V_X({t})$.
In fact, this statement holds in  more generality  \cite[Thm. II.1.8]{book:jacod}. 
} 
\begin{equation}
	\nu_I(\mathrm dt,\mathrm dx)=K_I(t,\mathrm dx)\,\mathrm dt,
\end{equation}
where $K_I(t,\mathrm dx)=K_I(t,\omega;\mathrm dx)$ is a transition kernel from $(\Rplus\times \Omega,\mathcal P)$ into $(\R,\mathcal B(\R))$, with
\begin{equation}
	K_I(t,\omega, A)=\int_{\R} \mathbf{1}_{A \backslash\{0\}}(H(t,\omega) x) \,\nu(\mathrm{d} x),
\end{equation}
defined for all $t\in\Rplus$ and every $A\in\mathcal B(\R)$. In other words, for all   measurable sets $S\in\mathcal B( \Rplus)$ and  $A\in\mathcal B(\R)$, we have  $\nu_{I }(S,A)=\int_S K_{I}(s,A)\,\mathrm ds.$


     Completely analogous to Lemma \ref{lem:sum} and Lemma \ref{lem:increasing}, we now claim---in line with  \cite{article:reiss}---that the following theorems, initially formulated for sequences only, extend to time-continuous families. 
      Note that if   $Y\in\mathbb D[-\tau,\infty),$ then each $Y_t$ can be seen as a semimartingale on $[-\tau,0]$ with respect to the filtration $\mathbb F_t=(\mathcal F_s)_{s\in[t-\tau,t]} $. 

      \begin{theorem}[{Theorem VI.4.18 and Remark VI.4.20 of \cite{book:jacod}}] 
	A\label{thm:semi}  continuous-time family $(Z_t)_{t\geq 0}$ of $D[-\tau,0]$-valued random variables, where each $Z_t$ is a semimartingale on the filtered probability space $(\Omega,\mathcal F,\mathbb F_t,\mathbb P)$, is tight if the following conditions hold:
	\begin{enumerate}[\normalfont(i)]
		\item the family $(Z_t(-\tau))_{t\geq 0}$ of starting values is tight in $\mathbb R$, i.e., bounded in probability\textnormal ;
		\item for all $\varepsilon>0$, we have
		\begin{equation}\label{eq:change}
		\textstyle \lim _{N \to \infty} \sup _{t\geq 0} \mathbb P\left(\nu_{Z_t}\big([-\tau,0] \times\{x\in\R:|x|>N\}\big)>\varepsilon\right)=0;
		\end{equation}
	\item the continuous-time family $(a_{Z_t})_{t\geq 0}$, defined by
		 \begin{equation}\label{eq:abs}
		a_{Z_{t}}(s):= {\rm TV}(B_{Z_{t}})(s)+C_{Z_{t}}(s)+\int_{[-\tau, s] \times \mathbb{R}}(|x|^{2} \wedge 1) \,\nu _{Z_{t}}(\mathrm{d} u, \mathrm{d} x), \quad s \in[-\tau,0],
	\end{equation}
for each $t\geq 0$,
	is $C$-tight.
	\end{enumerate}
In fact, conditions {\normalfont(i)} and {\normalfont(ii)} are necessary  for tightness.
\end{theorem}

Observe that equation \eqref{eq:change} is  somewhat stronger in the continuous-time setting compared to the original result in \cite{book:jacod}, because we have replaced   $\limsup_{n\to\infty}$  with $\sup_{t\geq 0}$; recall the discussion in the proof of Proposition \ref{prop:Dtight}. 
 Finally,      in many of our applications we take the class of integrators to be square integrable. Then
 the family of continuous-time processes of our interest will be square integrable too, for which it thus suffices to exploit the following result and hence fully circumvent the application of semimartingale characteristics. 

\begin{theorem}[{Theorem VI.4.13 of \cite{book:jacod}}]\label{thm:squaretight}
	  Suppose $(Z_t)_{t\geq 0}$ is a family of $D[-\tau,0]$-valued random variables. For each $t\geq 0$, consider  the process  $Z_t $ to be a square integrable martingale on the filtered probability space $(\Omega,\mathcal F,\mathbb F_t,\mathbb P)$. Then $(Z_t)_{t\geq 0}$ is tight if
	  \begin{enumerate}[\normalfont(i)]
	  \item the  family $(Z_t(-\tau))_{t\geq 0}$ of starting values is tight in $\mathbb R$, i.e., bounded in probability\textnormal;
	  \item  the continuous-time family of predictable quadratic variations $\left(\langle Z_t\rangle \right)_{t\geq 0}$  is $C$-tight.
	  \end{enumerate}
\end{theorem}




\section{Krylov--Bogoliubov existence theorem}\label{app:D}
In this appendix, we state and prove the Krylov--Bogoliubov theorem in the setting of SDDEs with \cadlag noise. Although the proof is standard (for the non-delay case), we present a detailed outline with references to the literature to guide the interested reader. 

To this end, let $M=(M(t))_{t\geq 0}$  be a    semimartingale, $M(0)=0$, and consider the autonomous initial value problem
\begin{equation}\label{eq:SDDE2.2.2.2}\left\{
	\begin{array}{rlll}
		\mathrm d X(t)&=&a(X_t)\,\mathrm d t+b(X_{t-}) \,\mathrm d M(t), &\quad \text { for } t\geq 0, \\[.05cm]
		X(u)&=&\Phi(u),& \quad \text { for } u \in[-\tau, 0],
	\end{array}\right.
\end{equation}
with $\Phi\in\mathbb D[-\tau,0]$. 
Recall that well-posedness of \eqref{eq:SDDE2.2.2.2}  is ensured when we assume that the functionals $a,b$ are  proper locally Lipschitz; see {\S}\ref{sec:existunique}. As explained below, we require $M$ to have independent and stationary increments, i.e., we need $M$ to be a L\'evy process. 
Then, since we restrict ourselves to time-independent coefficients, we are  not able to tell the time while perceiving the system. These conditions result into  translation invariance properties, which are relevant for proving the theorem below.


\begin{theorem} [{Krylov--Bogoliubov: \cadlag case}]\label{thm:Krylov--Bogoliubov} Consider $M=(M(t))_{t\geq 0}$ to be a Lévy process. Suppose    initial value problem \eqref{eq:SDDE2.2.2.2} has a solution with maximal existence time $T_{\infty}=\infty$ $\mathbb P$-a.s., for any    $\Phi\in\mathbb D[-\tau,0]$ $($independent of $M)$. If the  partial  segment process $ (X_{t} )_{t \geq \tau}$ of some global solution $X=(X(t))_{-\tau\leq t<\infty}$ is tight in the Skorokhod space $D[-\tau,0]$, then the SDDE in \eqref{eq:SDDE2.2.2.2} admits an invariant measure $\nu.$ 
	In particular, if $\varepsilon>0$ and $K_{\varepsilon} \subset D[-\tau,0]$   is compact  such that $\mathbb{P}\left(X_{t} \in K_{\varepsilon}\right) \geq 1-\varepsilon $ {holds for all} $ t \geq \tau$, then there is an invariant measure $\nu$ with $\nu\left(K_{\varepsilon}\right) \geq 1-\varepsilon$.
\end{theorem}
Recall that by Proposition \ref{prop:wow} every (stochastically) continuous finite time horizon process is tight, which yields the following immediate corollary.
\begin{corollary} [{Krylov--Bogoliubov: continuous case}] Consider $M=(M(t))_{t\geq 0}$ to be a Brownian motion. Suppose    initial value problem \eqref{eq:SDDE2.2.2.2} has a solution with maximal existence time $T_{\infty}=\infty$ $\mathbb P$-a.s., for any continuous initial process $\Phi$ independent of $M$. If the    segment process $ (X_{t} )_{t \geq 0}$ of some global solution $X=(X(t))_{-\tau\leq t<\infty}$ is tight in $C[-\tau,0]$, then the SDDE in \eqref{eq:SDDE2.2.2.2}  admits an invariant measure $\nu.$ 
In particular, if $\varepsilon>0$ and $K_{\varepsilon} \subset C[-\tau,0]$  is compact  such that $\mathbb{P}\left(X_{t} \in K_{\varepsilon}\right) \geq 1-\varepsilon $ {holds for all} $ t \geq \tau$, then there is an invariant measure $\nu$ with $\nu\left(K_{\varepsilon}\right) \geq 1-\varepsilon$.
\end{corollary}

We follow the  method as first introduced by 
Krylov and 
Bogoliubov \cite{kryloff1937theorie}. Loosely speaking, this approach constructs an invariant measure $\nu$ by averaging   over all the distributions of a given
tight  segment process.
Recall that the push-forward measure of $\nu$ under  evaluation map $D[-\tau,0]\to \mathbb R,$ $\varphi\mapsto\varphi(0)$ is a stationary distribution of \eqref{eq:SDDE2.2.2.2}; let us denote it by $\mu.$ Formally,  the invariant measure $\nu$ that we find via the Krylov--Bogoliubov method, and hence the stationary distribution $\mu$, satisfy the (subsequential) limits
\begin{equation}
	\nu=\lim_{T\to\infty}\frac{1}{T}\int_0^T\mathcal L(X_{s})\,\mathrm ds,\qquad \mu=\lim_{T\to\infty}\frac{1}{T}\int_0^T\mathcal L(X({s}))\,\mathrm ds,\label{eq:eta}
\end{equation}
 where $\mathcal L(X_t)$ and $\mathcal L(X(t))$ denotes the laws of $X_t$ and $X(t)$, respectively. In the proof, we actually have to replace $s$ by $s+\tau$, as  certain properties only hold for $t\geq \tau$ within the \cadlag case.   Clearly,  $\nu$  has much  richer structure than $\mu$, and it is really fascinating that such a highly information dense object  exists in certain cases.
 Furthermore,  because the measure construction is achieved by means of single tight segment process,   there is   no (immediate) guarantee that for some other tight segment process we get the same invariant measure. In fact, the existence of two distinct invariant measures allows for infinitely many by taking convex combinations.

	

 We are now ready to prove  Theorem \ref{thm:Krylov--Bogoliubov}.  
 In the proof,  it is important that we consider the evolution of the segments $X_t$ rather than the scalar solution $X(t)$ itself; $D[-\tau,0]$ is the natural state space to work with. Indeed, the semigroup property in equation \eqref{eq:KryBoSemigroup} is generally not true for a scalar   solution \cite{article:reiss}. Differently put, the segment   process is a time-homogeneous Markov process, whereas the scalar process is not, due to the time delay in the equation. 
In particular, recall that for autonomous SDEs driven by a semimartingale,    solutions satisfy the Markov property when the noise has independent increments (even for non-autonomous systems), and that a solution is a time-homogeneous Markov process when the noise additionally has stationary increments  \cite{protter1977markov}. These facts translate to the SDDE setting, but then only for the segment process. Lastly, to fill in any possible remaining gaps in the proof below, we refer to  \cite[Sec. 3 and 4]{article:reiss},  \cite[Ch. 3]{thesis:geerten}, and  \cite[Ch. 2 and 3]{book:daprato} for more details.

 \begin{proof}[Proof of Theorem \ref{thm:Krylov--Bogoliubov}]  We split up this proof into two parts: the necessary setup (step 1) and the actual construction of an invariant measure (step 2). \\
	
	\textit{Step 1.}
	 \noindent Let    $D[-\tau,0]$ be endowed with the Skorokhod topology. Denote by   $B_{b}(D[-\tau, 0])$  the space of  Borel measurable functions $f:D[-\tau,0]\to \R$ such that $\|f\|_\infty:=\sup \,\{|f(\varphi)|: \varphi \in D[-\tau, 0]\}<\infty$. The normed vector space $(B_{b}(D[-\tau, 0]),\|\cdot \|_\infty)$ is   Banach   \cite{book:hitch}, for which we write again $B_{b}(D[-\tau, 0])$. 
	 Introduce the collection of operators $(P_t)_{t\geq 0}$, defined by
	 \begin{equation}
	 	P_t:B_{b}(D[-\tau, 0])\to B_{b}(D[-\tau, 0]), \quad (P_tf)(\varphi):=\mathbb E[f(X_t^\varphi)],\, \varphi\in D[-\tau,0],
	 \end{equation}
 where $(X_t^\varphi)_{t\geq 0}$ is the segment process of the solution $X^\varphi=(X^\varphi(t))_{t\geq 0}$ to  \eqref{eq:SDDE2.2.2.2} with initial process  the deterministic process $\Phi=\varphi$. 
	 
        Note that $(P_t)_{t\geq 0}$ is a {Markovian semigroup} \cite{article:reiss}. That is, every segment process $(X_t^\varphi)_{t\geq 0}$ is  Markov |this is true because the noise $M$ has independent increments|and we have $P_0=\text{Id}$ together with  $P_{s+t}=P_{s} P_{t}=P_{t} P_{s}$   for any $t, s \geq 0$|this follows immediately from the fact that system  \eqref{eq:SDDE2.2.2.2} is autonomous and that   $M$ has stationary increments\footnote{Without stationary increments,   segment processes $(X_t^\varphi)_{t\geq 0}$ are time-inhomogeneous  Markov processes, resulting merely in the property $E_{s,u}E_{u,t}=E_{s,t}$, $0\leq s\leq u\leq t,$ for the evolution family  $ (E_{s,t}f)(\varphi):=\mathbb E[f(X_t)|X_s=\varphi]$. Time-homogeneity allows us to define  $P_t:=E_{s,t+s}$, independent of  $s\geq 0$, resulting in the semigroup property.}. Indeed, we have
	 \begin{equation}
	 	\left(P_{s+t} f\right)(\varphi)  =\mathbb{E}\left[f(X_{s+t}^{\varphi})\right]
	 	 =\mathbb{E}\left[\mathbb{E}\big[f(X_{s+t}^{\varphi}) \mid X_{t}^{\varphi}\big]\right]
	 	 =\mathbb{E}\left[(P_{t} f)\left(X_{s}^{\varphi}\right)\right]  
	 	 =\left(P_{s}(P_{t} f)\right)(\varphi), \label{eq:KryBoSemigroup}
	 \end{equation}
  which tells us that the system is translation invariant.
  
	  Let   $C_{b}(D[-\tau, 0])\subset B_b(D[-\tau,0])$ be the subspace of functions that are continuous	 with respect to the Skorokhod topology on $D[-\tau, 0] .$ Then, for any $f\in C_b(D[-\tau,0])$, we have the properties  
	 \begin{equation}
	 	P_tf\in C_{b}(D[-\tau, 0])\quad\text{and}\quad \lim_{s\searrow t} P_sf(\varphi)=P_t(
	 	\varphi),\label{eq:two}
	 \end{equation}
	 for all $t\geq t_0$ and $\varphi\in D[-\tau,0]$, where $t_0=\tau$.  This implies  the Markovian semigroup is \textit{eventually Feller}---it would have been a Feller semigroup, as in \cite{book:daprato}, if $t_0=0.$ The immediate Feller property fails to hold in the \cadlag setting \cite[p. 1416]{article:reiss}.
  The second property in \eqref{eq:two} follows easily from the fact $\mathbb P(\Delta X^\varphi(t)\neq 0)\leq \mathbb P(\Delta M(t)\neq 0)=0$, and hence  $X^\varphi$ is stochastically continuous. Thanks to \cite[Lem. 3.2]{article:reiss}, we   obtain  that the partial segment process  $(X_t^\varphi)_{t\geq \tau}$ is stochastically continuous as well.  The first property in \eqref{eq:two} tells us  that  $P_t$ maps $C_{b}(D[-\tau, 0])$ into $C_{b}(D[-\tau, 0])$ for every time $  t\geq \tau$, resulting into the well-defined mappings
	 \begin{equation}
	 	P_t:C_{b}(D[-\tau, 0])\to C_{b}(D[-\tau, 0]),\,f\mapsto  \big[\varphi\mapsto \mathbb E[f(x_t^\varphi)]\big],\quad \text{for all } t\geq \tau. 
	 \end{equation} 
 
Recall  $ \mathscr P(D[-\tau, 0])$ denotes the collection of   Borel probability measures on $D[-\tau, 0],$  endowed with the topology of weak convergence of measures. Introduce the   duality pairing   $\langle\cdot, \cdot\rangle$,  given by \begin{equation}\langle\xi, f\rangle=\int_{D[-\tau, 0]} f(\varphi)  \,\xi(\mathrm{d}\varphi),\quad  \xi \in \mathscr P(D[-\tau, 0]), \,f \in B_{b}(D[-\tau, 0]) .\end{equation}
The terminology duality pairing is justified (to a certain extend). Indeed, note that the dual space $B_{b}(D[-\tau, 0])^*$ is isomorphic to  $ba(D[-\tau,0])$, the space   of all  finitely additive signed Borel measures    of bounded variation  \cite[Thm. 13.4]{book:hitch}. One typically considers  $B_{b}(D[-\tau, 0])^*\to   ba(D[-\tau, 0]),$  $\zeta\to \mu_\zeta$, where $\mu_\zeta(A)=\zeta(\mathbf 1_A)$, for all $A\in\mathcal B(D[-\tau, 0]).$
In particular, we have that $ \mathscr P(D[-\tau, 0])$ is   a subspace of $  ba(D[-\tau, 0])$. 
 Subsequently, let us  define for any $t \geq 0$ and   $\xi\in \mathscr P(D[-\tau, 0])$
 the continuous linear functional $P_{t}^{*} \xi \in B_{b}(D[-\tau, 0])^*$ that acts on bounded Borel measurable functions as follows:
 \begin{equation}
 \left(P_{t}^{*} \xi\right) f:=\left\langle\xi, P_{t} f\right\rangle, \quad f \in B_{b}(D[-\tau, 0]).
 \end{equation}
Continuity of   $P_t^*\xi$ is evident, as $|(P_t^*\zeta)f|\leq \|f\|_\infty$.  By the isomorphism  above, we can identify  $P_{t}^{*} \xi$ as an element in $  ba(D[-\tau, 0])$. Whenever $P_{t}^{*} \xi$ is a true signed measure---thus, when   the $\sigma$-additivity property is satisfied---the natural identification allows to write $P_{t}^{*} \xi=[f\mapsto 	\left\langle P_{t}^{*} \xi, f\right\rangle]$ and we find 
\begin{equation}
	\left\langle P_{t}^{*} \xi, f\right\rangle=\left\langle\xi, P_{t} f\right\rangle.\end{equation}
As a matter of fact, the functional $P_{t}^{*} \xi$ turns out to be an element in $\mathscr{P}(D[-\tau,0])$. Indeed, if $\xi$ is the distribution of an initial value process $\Phi,$ then $P_{t}^{*} \xi$ is the distribution of $X_{t}^\Phi,$ because
\begin{equation}
\left\langle P_{t}^{*} \xi, f\right\rangle=\int_{D[-\tau, 0]} \mathbb{E}\left[f (X_{t}^\varphi )\right] \xi(\mathrm{d} \varphi)=\mathbb{E}\left[\mathbb{E}\left[f (X_{t}^\Phi  ) \mid \mathcal{F}_{0}\right]\right]=\mathbb{E}\left[f (X_{t}^\Phi  )\right],
\end{equation}
 for all $ f \in B_{b}(D[-\tau, 0]).$
 This makes   \begin{equation}
	P_t^*:\mathscr P(D[-\tau, 0])\to \mathscr P(D[-\tau, 0]),\,\xi\mapsto\big[f\mapsto \langle \xi,P_tf\rangle\big],\quad \text{for all }t\geq 0, 
\end{equation}
  a well-defined mapping (under the natural identification).
 Lastly, this results into $P_{s+t}^{*}=$  $P_{s}^{*} P_{t}^{*}=P_{t}^{*} P_{s}^{*}$, for any $s,t\geq 0$, since for   $\xi \in \mathscr{P}(D[-\tau,0])$ and $f \in B_{b}(D[-\tau,0])$ we have
\begin{equation}
	 (P_{s+t}^{*} \xi ) f  =\left\langle\xi, P_{s+t} f\right\rangle \\
	=\left\langle\xi, P_{t} (P_{s} f)\right\rangle \\
	=\left\langle P_{s}^{*}( P_{t}^{*} \xi), f\right\rangle \\
	=\left(P_{s}^{*} (P_{t}^{*} \xi)\right) f .
\end{equation}
We have introduced all the necessary objects and discussed their properties accordingly.\\

 \textit{Step 2.} A  measure $\xi \in \mathscr P(D[-\tau,0])$ is  an invariant measure   if $P_{t}^{*} \xi=\xi$ holds for all $t \geq 0,$ i.e., if $\left\langle\xi, P_{t} f\right\rangle=\langle\xi, f\rangle$ holds for all $f \in B_{b}(D[-\tau,0])$ and all $t \geq 0$. It suffices to show $\left\langle\xi, P_{t} f\right\rangle=\langle\xi, f\rangle$ for all $f \in C_{b}(D[-\tau,0])$, because this implies $d_0(P_{t}^{*} \xi,\xi)=0$ and   hence $P_{t}^{*} \xi=\xi$.
 
  Let $\zeta$ be the distribution of an initial value process $\Phi$ such that   the set $\left\{P_{t}^{*} \zeta: t \geq \tau \right\}$ is tight. In other words, we choose $\Phi$ such that the partial segment process $(X_t)_{t\geq \tau}$ is tight.
As a result of the Riesz--Bourbaki representation theorem  \cite[Prop. XI.5]{book:bourbaki}, there exists a unique family  of Borel signed measures $(\theta_{t})_{t \geq 0}$  of bounded variation such that,  
 for each $t\geq 0$ fixed, we have
 \begin{equation}
 	\left\langle\theta_{t}, f\right\rangle=\frac{1}{t} \int_{0}^{t}\langle P_{s+\tau}^{*} \zeta, f\rangle \,\mathrm d s,
 \end{equation}
 for every $f \in C_{b}(D[-\tau,0])$. 
  In fact, it directly follows that we have $(\theta_{t})_{t \geq 0}\subset \mathscr P(D[-\tau,0])$, because  it is contained in the closed convex hull of $\left\{P_{t}^{*} \zeta: t \geq \tau \right\}$. For properties of a convex hull, we refer to  \cite{book:hitch}. 
In particular, the convex hull of any tight set is again tight, hence by Prokhorov's theorem we obtain that the convex hull of $\left\{P_{t}^{*} \zeta: t \geq \tau\right\}$ is relatively compact. 
Hence, we can find a subsequence $\left(\theta_{t_{n}}\right)_{n \in \mathbb{N}} \subset\left(\theta_{t}\right)_{t \geq 0}$ such that $ \theta_{t_{n}}\to\nu$ holds, with $\nu \in \mathscr{P}(D[-\tau,0])$ and where the convergence is in the weak topology. In particular, we have   $\langle \nu,g\rangle =\lim_{n\to\infty}\langle \theta_{t_n},g\rangle$, for all  $g\in C_b(D[-\tau,0])$.

 Finally, for   $t \geq \tau$ and every $f \in C_{b}(D[-\tau,0])$, we have $P_tf\in C_b(D[-\tau,0])$, and therefore
%
%
\begin{equation}
\begin{aligned}
\left\langle\nu, P_{t} f\right\rangle &=\lim _{n \rightarrow \infty}\left\langle\theta_{t_{n}}, P_{t} f\right\rangle \\
&=\lim _{n \rightarrow \infty} \frac{1}{t_{n}} \int_{0}^{t_{n}} \langle P_{s+\tau}^{*} \zeta, P_{t} f \rangle \,\mathrm d s \\
&=\lim _{n \rightarrow \infty} \frac{1}{t_{n}} \int_{0}^{t_{n}} \langle P_{t+s+\tau}^{*} \zeta, f \rangle \,\mathrm d s\\
&=\lim _{n \rightarrow \infty} \frac{1}{t_{n}} \int_{t}^{t+t_{n}} \langle P_{s+\tau}^{*} \zeta, f \rangle \,\mathrm d s\\
&=\lim _{n \rightarrow \infty}\left(\frac{1}{t_{n}} \int_{0}^{t_{n}}\langle P_{s+\tau}^{*} , f\rangle \,\mathrm{d} s+\frac{1}{t_{n}} \int_{t_{n}}^{t_{n}+t}\langle P_{s+\tau}^{*}\zeta , f\rangle \,\mathrm{d} s-\frac{1}{t_{n}} \int_{0}^{t}\langle P_{s+\tau}^{*}\zeta , f\rangle \,\mathrm{d} s\right)\\
&=\lim _{n \rightarrow \infty} \frac{1}{t_{n}} \int_{0}^{t_{n}}\langle P_{s+\tau}^{*}\zeta , f\rangle\, \mathrm{d} s+0+0\\
&= \lim _{n \rightarrow \infty}\left\langle\theta_{t_{n}}, f\right\rangle \\
&=\langle\nu, f\rangle. 
\end{aligned}\end{equation}
The other two limits equal zero because we have\begin{equation}\left|\int_{\alpha}^{\alpha+t}\langle P_{s+\tau}^{*}\zeta , f\rangle\, \mathrm{d} s\right|\leq \|f\|_\infty t, \end{equation}
for any $\alpha\geq 0.$
We conclude $P_t^*\nu=\nu$, for all $t\geq \tau. $ Consequently, for $0\leq t<\tau$, we also obtain
\begin{equation}
	P_t^*\nu=P_t^*(P_\tau^* \nu)=P_{t+\tau}^*\nu=\nu.
\end{equation}
This completes the proof.
\end{proof}


\printbibliography

\end{document}